\documentclass[11pt,times,twoside]{article}
\usepackage{graphicx,latexsym,euscript,makeidx,color,bm}
\usepackage{amsmath,amsfonts,amssymb,amsthm,thmtools,mathtools,mathrsfs,enumerate}
\usepackage[colorlinks,linkcolor=blue,anchorcolor=green,citecolor=red]{hyperref}
\usepackage{geometry}
\def\5n{\negthinspace \negthinspace \negthinspace \negthinspace \negthinspace }
\def\4n{\negthinspace \negthinspace \negthinspace \negthinspace }
\def\3n{\negthinspace \negthinspace \negthinspace }
\def\2n{\negthinspace \negthinspace }
\def\1n{\negthinspace }
 \def\sB{\mathscr{B}}

\def\dbE{\mathbb{E}}   \def\cE{{\cal E}}  
\def\dbF{\mathbb{F}} \def\sF{\mathscr{F}}  \def\cF{{\cal F}}

   \def\cM{{\cal M}}  
     
   \def\cO{{\cal O}}  
\def\dbP{\mathbb{P}}   \def\cP{{\cal P}}  
   \def\cQ{{\cal Q}}  
\def\dbR{\mathbb{R}}     
     
 \def\sT{\mathscr{T}}

   \def\cW{{\cal W}}     
     
\def\dbY{\mathbb{Y}}   \def\cY{{\cal Y}}      
\def\dbZ{\mathbb{Z}}   \def\cZ{{\cal Z}}      
\def\Om{\Omega}           \def\Th{\Theta} \def\sgn{\mathop{\rm sgn}}

\def\ss{\smallskip}                
\def\ms{\medskip}                
               
\def\ds{\displaystyle}           
\def\ra{\rightarrow}      
 
\def\no{\noindent}        \def\q{\quad}                      
\def\ns{\noalign{\ss}}    \def\qq{\qquad}                    
    \def\hb{\hbox}                     
                   
         \def\rf{\eqref}                    
  \def\deq{\triangleq}               
            \def\({\Big (}
\def\les{\leqslant}                  \def\){\Big )}
\def\leq{\leqslant}       \def\geq{\geqslant}
\def\ges{\geqslant}       \def\esssup{\mathop{\rm esssup}}   \def\[{\Big[}
\def\cl{\overline}           \def\]{\Big]}
\def\h{\widehat}                   \def\cd{\cdot}
\def\wt{\widetilde}              \def\cds{\cdots}
           \def\cl{\overline}                 
\def\nn{\nonumber}        \def\ts{\times}                      
\def\a{\alpha}        \def\G{\Gamma}   \def\g{\gamma}   \def\O{\Omega}   \def\o{\omega}
\def\b{\beta}         \def\D{\Delta}   \def\d{\delta}        
         \def\Th{\Theta}  \def\th{\theta}    
\def\e{\varepsilon}   \def\L{\Lambda}  \def\l{\lambda}        \def\n{\nu}
    \def\t{\tau}       \def\i{\infty}   \def\k{\kappa}

\theoremstyle{plain}

\theoremstyle{rmk}

\makeatletter

\@addtoreset{equation}{section}
\makeatother
\newtheorem{theorem}{Theorem}[section]
\newtheorem{definition}[theorem]{Definition}
\newtheorem{proposition}[theorem]{Proposition}

\newtheorem{lemma}[theorem]{Lemma}
\newtheorem{remark}[theorem]{Remark}

\newtheorem{assumption}{Assumption}
\makeatletter
   
   \@addtoreset{equation}{section}
\makeatother

\allowdisplaybreaks[4]
\begin{document}

\title{\bf Mean-field backward stochastic Volterra integral equations: well-posedness and related particle systems}

\author{Tao Hao\thanks{School of Statistics and Mathematics, Shandong University of Finance and Economics, Jinan 250014, China (Email: {\tt taohao@sdufe.edu.cn}).
TH is  supported by Natural Science Foundation of Shandong Province (Grant Nos.ZR2024MA039, ZR2022MA029),
and NSFC (Grant No. 72171133).}~,~~~
Ying Hu\thanks{Univ. Rennes, CNRS, IRMAR - UMR 6625, F-35000, Rennes, France (Email: {\tt ying.hu@univ-rennes1.fr}). YH is partially supported by Lebesgue Center of Mathematics ``Investissements d'avenir" program-ANR-11-LABX-0020-01, by CAESARS-ANR-15-CE05-0024 and by MFG-ANR16-CE40-0015-01.}~,~~~
%%%%
%%%Shanjian Tang\thanks{Department of Finance and Control Sciences, School of Mathematical Sciences, Fudan University, Shanghai 200433, China (Email: {\tt sjtang@fudan.edu.cn}). ST is partially supported by National Science Foundation of China (Grant No. 11631004) and
%%%National Key R/D Program of China (Grant No. 2018YFA0703903).
%%%}~,~~
%%%%
%Juan Li\thanks{School of Mathematics and Statistics, Shandong University, Weihai 264209, China
%	 (Email: {\tt juanli@sdu.edu.cn}). JL is supported by NSFC (Nos. 12031009, 11871037),
%	 National Key R\&D Program of China (No. 2018YFA0703900), and NSFC-RS (No. 11661130148; NA150344).}~,~~
% ~,~~
 Jiaqiang Wen\thanks{Corresponding author. Department of Mathematics, Southern University of Science and Technology, Shenzhen 518055, China (Email: {\tt wenjq@sustech.edu.cn}). JW is supported  by National Natural Science Foundation of China (Grant No. 12571478), Guangdong Basic and Applied Basic Research Foundation (Grant No. 2025B151502009), and Shenzhen Fundamental Research General Program grant (Grant
 No. JCYJ20230807093309021).}
 %
%%%%
%%Qi Zhang\thanks{School of Mathematical Sciences, Fudan University, Shanghai 200433, China \&
%%Laboratory of Mathematics for Nonlinear Science, Fudan University, Shanghai 200433, China (Email:{\tt qzh@fudan.edu.cn}). QZ is supported by National Key R\&D Program of China (Nos.2022YFA1006101, 2018YFA0703900), National Natural Science Foundation of China (Nos.12371445, 12071292), and Science and Technology Commission of Shanghai Municipality (No.22ZR1407600).}~,~~
}

%\date{}
\maketitle

\no\bf Abstract: \rm
 This paper studies the mean-field backward stochastic Volterra integral equations (mean-field BSVIEs) and  associated particle systems.
 We establish the existence and uniqueness of solutions to mean-field BSVIEs when  the generator $g$ is of linear growth or  quadratic growth with respect to $Z$, respectively.
Moreover, the propagation of chaos is analyzed for the corresponding particle systems under two conditions.  When $g$ is of  linear growth in $Z$,   the convergence rate is proven to be of order $\cQ(N)$. 
When  $g$ is of  quadratic growth in $Z$ and is independent of the law of $Z$,  we not only establish the convergence of the particle systems but also derive a convergence rate of order   $\cO(N^{-\frac{1}{2\l}})$, where  $\l>1$.

\ms

\no\bf Key words:\rm\
Backward stochastic Volterra integral equation; mean-field;  particle system; quadratic growth; convergence rate

\ms

\no\bf AMS subject classifications. \rm  60H10; 60H30
%93E20, 60H10, 35K15

%%%%%%%%%%%%%%%%%%%%%%%%%%%%%%%%%%%%%%%%%%%%%%%%%%%%%%%%%%%%%%%%%%%%%%%%%%%%%%%%%%%%%%%%%%%%%%%%%%%%%%%%%%%%%%%%%%%%%%%%%%%%%%%%%%%%%%%%%%%%
%%%%%%%%%%%%%%%%%%%%%%%%%%%%%%%%%%%%%%%%%%%%%%%%%%%%%%%%%%%%%%%%%%%%%%%%%%%%%%%%%%%%%%%%%%%%%%%%%%%%%%%%%%%%%%%%%%%%%%%%%%%%%%%%%%%%%%%%%%%%

\section{Introduction}\label{sec 1}

    Let  $(\Omega,\sF,\mathbb{P})$ be a complete probability space, where a $d$-dimensional standard Brownian motion $W$ is defined, and the completed augmented natural filtration is denoted by $\mathbb{F}=\{\sF_t\}_{0\leq t\leq T}$. In 1990, Pardoux and Peng \cite{Pardoux-Peng-1990}  first established the existence and uniqueness of solutions to nonlinear backward stochastic differential equations (BSDEs, for short) under Lipschitz assumptions.
    Since then, this theory has found extensive applications in stochastic optimal control, partial differential equations, and financial mathematics. One important application of BSDEs in economics is their use in modeling stochastic differential utility (SDU, for short), as illustrated below:
       \begin{equation*}
         Y(t)=\dbE\[ \xi+\int_t^T h(s,Y(s))ds \Big|\sF_t\],
        \end{equation*}
     where $\xi$ is a square-integrable random variable, $h:[0,T] \times \mathbb{R} \rightarrow \mathbb{R}$ is an appropriate function, and $\mathbb{E}[\cdot | \sF_t]$ represents the conditional expectation under the probability $\dbP$.
     This SDU  can be interpreted as the solution to a particular BSDE with terminal value $\xi$ and  generator $h$ (see Duffie and Epstein \cite{Duffie-Epstein-1992}). Moreover, El Karoui, Peng, and Quenez \cite{El Karoui-Peng-Quenez-1997} extended their work by introducing the concept of a general SDU. Nonetheless, the standard expected utility function (EUF, for short) is typically expressed as
        \begin{equation*}
           Y(t)=\dbE\[\xi e^{\g(T-t)}+\int_t^Tu(c(s))e^{-\g(s-t)}ds\Big|\sF_t\],
        \end{equation*}
     where $u(\cdot)$ is a utility function, $c(\cdot)$ is a consumption process, and $\gamma$ is the discount rate.
     The standard EUF, which includes a consumption process, a discount rate, and a utility function, emphasizes that both the terminal value $\overline{\psi}(t) \equiv \xi e^{\gamma(T-t)}$ and the generator $\overline{g}(t,s) \equiv u(c(s))e^{-\gamma(s-t)}$ are time-dependent and have  memory effects, i.e., a feature not naturally captured by classical BSDEs. To better describe this phenomenon, many researchers have turned it to be backward stochastic Volterra integral equations (BSVIEs, for short) of the following form:
        \begin{equation*}\label{1.0}
           Y(t)=  \psi(t)+\int_t^T g(t,s,Y(s),Z(t,s))ds -\int_t^TZ(t,s)dW(s),\q t\in[0,T].
        \end{equation*}
      For this class of equations, Lin \cite{Lin-2002} established the existence and uniqueness of adapted solutions under Lipschitz conditions,
      Yong \cite{Yong-2008} introduced a class of general BSVIEs and proved their  well-posedness.
      Readers may refer to  \cite{Hamaguchi-AAP-2025,Hamaguchi-JDE-2023,Shi-Wang-2012,Shi-Wang-Yong-2013,Shi-Wen-Xiong-2020,Wang-SD-2021,Wang-Sun-Yong-2019,Wang-2020,Wang-Yong-2015}   for new developments of BSVIEs and their applications.
  Note that    although the Lipschitz condition facilitates mathematical derivations, it often restricts the range of interesting applications, such as risk-sensitive problems.
    To address this limitation, Peng \cite{Peng-1998}   proposed relaxing the Lipschitz assumption on the generator  $g$  when studying the existence and/or uniqueness of adapted solutions to BSDEs.  This suggestion has inspired numerous research efforts in the field.
   For instance, Lepeltier and San Martin \cite{Lepeltier-Martin-1997} studied the case where the generator    $g$ is continuous in $(Y,Z)$ and has linear growth.
   Kobylanski \cite{Kobylanski-2000}, Briand and Hu \cite{Briand-Hu-2006, Briand-Hu-2008}, as well as Bahlali, Eddahbi, and Ouknine \cite{Bahali-Eddahbi-Ouknine-2017},
   worked on one-dimensional BSDEs with quadratic growth (QBSDEs).
    Hu and Tang \cite{Hu-Tang-2016}, Xing and Zitkovic \cite{Xing-Zitkovic}, and Fan, Hu, and Tang \cite{Fan-Hu-Tang-2023} focused on multi-dimensional QBSDEs.
For BSVIEs,  similar efforts have been made to relax the Lipschitz assumption on the generator
$g$.  For instance, Shi and Wang \cite{Shi-Wang-2012} investigated general BSVIEs under non-Lipschitz conditions. Moreover, Wang, Sun, and Yong \cite{Wang-Sun-Yong-2019} extended Lin's work \cite{Lin-2002} from linear growth to quadratic growth and established the well-posedness of both local and global solutions for BSVIEs with quadratic growth.

However, to the best of our knowledge, no fundamental results have yet been established for general mean-field BSVIEs (i.e.,  the generator $g$ depends on the distribution of the solution $(Y,Z))$, let alone for related particle systems. Note that the work of Shi, Wang, and Yong \cite{Shi-Wang-Yong-2013} dealt with classical mean-field BSVIEs, in which the generator $g$ depends on the expectation of the solution $(Y,Z)$; meanwhile, Jiang, Li, and Wei \cite{Jiang-Li-Wei-2023}  discussed a  special case within the general setting.

In this paper, motivated by the wide applications of particle systems, we study general mean-field BSVIEs. Specifically, we focus on analyzing the convergence rate of particle systems associated with the following BSVIEs:
      For $i=1,\cdots,N$,
         \begin{equation}\label{1.3}
         \begin{aligned}
          Y^{N,i}(t)&=\psi^i(t)+\int_t^Tg(t, s,Y^{N,i}(s),Z^{N,i,i}(t,s), \mu^{N}(s),\nu^N(t,s) )ds\\
                &\q-\int_t^T\sum\limits_{j=1}^NZ^{N,i,j}(t,s)dW^j(s),\q t\in[0,T],
            \end{aligned}
         \end{equation}
      where $\{W^j;1\leq j\leq N\}$ are $N$ independent $d$-dimensional standard Brownian motions on the complete probability space  $(\O,\sF,\dbP)$, and
            \begin{equation}\label{5.3}
            \mu^N(s)\deq  \frac{1}{N}\sum\limits_{i=1}^N\delta_{Y^{N,i}(s)}\q\hb{and}\q
             \nu^N(t,s)\deq  \frac{1}{N}\sum\limits_{i=1}^N\delta_{Z^{N,i,i}(t,s)}.
              \end{equation}
      The equation above is a multi-dimensional BSVIE. As the system size becomes very large, i.e., as $N \rightarrow \infty$, based on the theory of convergence under the Wasserstein distance for the empirical measure (as discussed in Fournier and Guillin \cite{Fournier-Guillin-2015}), the particle system described in Eq. (\ref{1.3}) will converge in structure to the following mean-field BSVIE:
       \begin{equation}\label{1.1}
                   Y(t)=\psi(t)+ \int_t^Tg(t,s, Y(s),Z(t,s), \dbP_{Y(s)}, \dbP_{Z(t,s)})ds-\int_t^TZ(t,s)dW(s),\q  t\in[0,T].
       \end{equation}
The process $\psi(\cdot)$ and the mapping $g$ are referred to as the free term and the generator, respectively. We denote equation (\ref{1.1}) as a general mean-field BSVIE. The term ``general''  here indicates that the generator
$g$ depends on the distributions of  $Y(s)$  and that of  $Z(t,s)$, rather than merely their expectations. This convergence implies  that various analyses related to the particle system in equation (\ref{1.3}), such as error estimates of solutions and asymptotic Nash equilibrium, can be effectively addressed using the framework of equation (\ref{1.1}).

The main goal of this paper is to study  the  convergence and  its rate for the particle system described by  (\ref{1.3})-(\ref{5.3}).
Before studying our main results, we first establish the well-posedness of  BSVIE \eqref{1.1}. To obtain more general results, we consider both linear growth and quadratic growth cases.
For the linear growth case, we initially prove the existence and uniqueness of the solution to BSVIE \eqref{1.1} by parameterizing the time variable $t$ and applying the contraction mapping principle (see \autoref{th3.1}). Subsequently, we derive a convergence rate of order $\cQ(N)$ by combining the $t$-parameterization method with convergence theory in the Wasserstein distance for empirical measures (see \autoref{th4.2}).
{
	Note that when $g$ depends on the law of $Z(\cd,\cd)$, we need to work within the space 
$L^p_{\dbF}(\D [0,T];\dbR^d)$ with $1<p<2$, because we only have that  $\int_{0}^T\int_t^T\dbE|\overline{Z}^{i}(t,s)|^2 dsdt<\i$
(see (\ref{4.100})). However, when $g$ is independent of the law of $Z(\cd,\cd)$, we can work in the space
$L^p_{\dbF}(\D [0,T];\dbR^d)$ for $p>1$ (see \autoref{pro4.44}).
Furthermore, even for general mean-field BSDEs with Lipschitz conditions, our result concerning the convergence rate is new. To the best of our knowledge, all previous works assume that the generator is independent of the law of $Z(\cdot)$.
}
For the quadratic growth case, we first establish the existence and uniqueness of global solutions to BSVIE \eqref{1.1} using BMO martingale techniques, considering the generator $g$ under both bounded and unbounded conditions with respect to the distribution   $\dbP_{Z(t,s)}$  (see \autoref{th3.4} and \autoref{th 4.3}). {Next, under the extra assumption that $g$ is independent of the law of $Z(\cd,\cd)$,}
we prove the  convergence for the particle system, using BMO techniques and Girsanov's theorem (see \autoref{th 5.1111}).
Finally, based on these convergence results, we derive a convergence rate of order $\cO(N^{-\frac{1}{2\l}})$  with $\l>1$ for the particle system (see \autoref{th 5.6}).
In summary, the main innovations of this work include:
(i) Providing, for the first time, the convergence and precise convergence rates of particle systems for mean-field BSVIEs.
(ii) Allowing the generator $g$ to depend on the distributions of $Y(t)$ and $Z(t,s)$ rather than merely on their expectations.
(iii) Allowing the generator $g$ to be of  quadratic growth  in $Z(t,s)$.
(iv) Establishing new existence and uniqueness results when the generator $g$ is of linear growth in $Z(t,s)$.
(v)  \autoref{ass4.1-1} and \autoref{ass4.1} here are weaker than the condition (A2) of Wang, Sun, and Yong \cite{Wang-Sun-Yong-2019},  as we relax the boundedness requirements of the generator $g$  in both $(t, s)$ and its dependence on distributions.

This paper is organized as follows.   \autoref{sec 2}  introduces some notations and spaces.  \autoref{sec 3} presents the solvability of general  mean-field BSVIEs.  \autoref{sec 5} addresses the convergence and the convergence rate  of associated particle systems.
    Finally,  several supporting results are presented in  \autoref{appendix}.

%%%%%%%%%%%%%%%%%%%%%%%%%%%%%%%%%%%%%%%%%%%%%%%%%%%%%%%%%%%%%%%%%%%%%%%%%%%%%%%%%%%%%%%%%%%%%%%%%%%%%%%%%%%%%%%%%%%%%%%%%%%%%%%%%%%%%%%%%%%%
%%%%%%%%%%%%%%%%%%%%%%%%%%%%%%%%%%%%%%%%%%%%%%%%%%%%%%%%%%%%%%%%%%%%%%%%%%%%%%%%%%%%%%%%%%%%%%%%%%%%%%%%%%%%%%%%%%%%%%%%%%%%%%%%%%%%%%%%%%%%
\section{Preliminaries}\label{sec 2}

       This section introduces some useful spaces and properties of BMO martingales.
       We begin with the introduction of the upper  triangle domains. For $0\leq a < b \leq T$,  denote
        \begin{equation*}
        \begin{aligned}
        \D[a,b]&=\Big\{(t,s)\in[a,b]\Big|a\leq t\leq s\leq b\Big\}.
        \end{aligned}
        \end{equation*}
     In what follows, we denote $\d_0$ the Dirac measure at $0$ and $\sT[a,b]$  the set of all $\dbF$-stopping times $\t$ valued in $[a,b]$.
     For almost all $t\in[0,T]$ and for stopping time $\t\in\sT[t,T]$,
by $\widehat{\dbE}^t_\t[\cd]:=\widehat{\dbE}^t[\cd|\sF_\t]$ we denote the conditional expectation
  on the $\sigma$-field $\sF_\t$ under the probability $\widehat{\dbP}^t$. One can
understand $ \widetilde{\dbE}^t_\t[\cd],\ \tilde{\tilde{\dbE}}^t_\t[\cd],\ \overline{\dbE}^t_\t[\cd],\  \breve{\dbE}^t_\t[\cd]$ in the same sense.

\subsection{ Useful spaces   }
       Let us introduce spaces of random variables and  stochastic processes. To avoid repetition,
       all processes $(t,\o)\mapsto f(t,\o)$ are assumed to be at least $\sB[a,b]\otimes \sF_b$-measurable  without further mention, where  $\sB[a,b]$ is the Borel-field on $[a,b].$   For $p,q \geq1,$ and $r>0$
\begin{align*}
 L^p_{\sF_b}(\O;\dbR^n)=&\ \Big\{ \xi:\O\ra \dbR^n\Big| \xi\ \text{is}\ \sF_b\text{-measurable},
\|\xi\|^p_{L^p_{\sF_b}(\O)}:= \dbE[|\xi|^p]<\i  \Big\},\\
  L^{q,p}_{\sF_b}([a,b]; \dbR^n)=&\ \Big\{ f:\O\ts[a,b]\ra \dbR^n\Big|
         \dbE\(\int_a^b|f(s)|^q ds\)^\frac{p}{q}  < \i  \Big\},\\
   L^{\i,p}_{\sF_b}([a,b]; \dbR^n)=&\ \Big\{ f:\O\ts[a,b]\ra \dbR^n\Big|
         \dbE\(\esssup\limits_{t\in[a,b]}|f(t)|^p\) <\i  \Big\},\\
  L^p_{\sF_b}(\O;C([a,b]; \dbR^n))=&\ \Big\{ f:\O\ts[a,b]\ra \dbR^n\Big|t\mapsto f(t,\o)\ \text{is\ continuous},
              \dbE\(\sup\limits_{t\in[a,b]}|f(t)|^p\) <\i  \Big\},\\
   L^\i_{\sF_b}([a,b]; \dbR^n)\deq  &\ L^{\i,\i}_{\sF_b}([a,b]; \dbR^n)=\Big\{ f:\O\ts[a,b]\ra \dbR^n\Big|
         \esssup\limits_{(\o, t)\in\O\ts[a,b]}|f(t)|  <\i  \Big\},\\
         %%%%%%%%
         \cE^r_{\sF_b}(\O; \dbR^n )=&\ \Big\{ \xi\in L^1_{\sF_b}(\O;\dbR^n)\Big | \  \dbE[\exp(r|\xi|)]<\i  \Big\},\\
             \cE^r_{\sF_b}(\O;L^p([a,b];\dbR^n))=&\ \Big\{f:\O\ts[a,b]\ra \dbR^n\Big|
           \big(\int_a^b|f(s)|^pds\big)^{\frac{1}{p}}\in \cE^r_{\sF_b}(\O; \dbR^n )  \Big\}.
\end{align*}
In addition, the subset of $L^{q,p}_{\sF_b}([a,b]; \dbR^n)$, consisting of all $\dbF$-progressively measurable processes,
          is denoted by $L^{q,p}_{\dbF}([a,b]; \dbR^n)$.  All the $\dbF$-progressively measurable
         version of the above other spaces  can be understood in the same way. In analogy with
        the space  $ L^{q,p}_{\dbF}([a,b]; \dbR^n),$ we set
        % $L^p_{\dbF}(\O;L^q(\D[a,b]; \dbR^n))$ can be understood similar to   $L^p_{\dbF}(\O;L^q(a,b; \dbR^n))$, in other word,
\begin{align*}
  L^{q,p}_{\dbF}(\D[a,b]; \dbR^n)=\Big\{
          &  f(\cd,\cd):\O\ts\D[a,b]\ra \dbR^{n}\ |\ f(\cd,\cd)\ \text{is}\  \sB(\D[a,b])\otimes \sF_b
             \text{-measurable\  such }    \text{that}\\
             & s\mapsto f(t,s)  \  \text{is}\ \dbF \text{-progressively\ measurable\ for\ all}\ t\in[a,b],\ \text{and}\\
           & \|f(\cd,\cd)\|^p_{L^{q,p}_{\dbF}(\D[a,b])}=\dbE\(\int_a^b\int_t^b|f(t,s)|^q dsdt\)^\frac{p}{q}  < \i \Big\}.
\end{align*}
The space $L^{q,p}_{\dbF}( [a,b]\ts[b,c]; \dbR^n)$ can be defined similarly. Meanwhile, for simplicity of presentation,  we write
         \begin{equation*}
         L^p_{\dbF}([a,b];\dbR^n)=L^{p,p}_{\dbF}([a,b];\dbR^n),\ \forall p\ges1.
         \end{equation*}
     Similarly, one can understand $L^p_{\sF_T}(\D[a,b];\dbR^n),$ $L^p_{\dbF}(\D[a,b];\dbR^n),$
        and so on.
        % %
%          Note that the following relations hold true,
%           \begin{equation*}
%           \begin{aligned}
%             &L^\i_{\dbF}([a,b];\dbR^n)\subset L^p_{\dbF}([a,b];\dbR^n).
%             \end{aligned}
%           \end{equation*}
%
              Finally, we want to introduce the following spaces for $Z(\cd,\cd)$:  For  $p,q\geq1$,
\begin{align*}
 &    L^\i([a,b];L^{q,p}_{\dbF}([\cd,b]; \dbR^n))
           =\Big\{ f:\D[a,b]\ts \O\ra \dbR^n\ |\ \text{for all } t\in[a,b], f(t,\cd)\in  L^{q,p}_{\dbF}([t,b];\dbR^n), \text{and}\\
& \qq\qq\qq\qq\qq\qq\          \|f(\cd, \cd)\|^p_{ L^{\i,q,p}_{\dbF}(\D[a,b])}:= \esssup_{t\in[a,b]}\dbE\(\int_t^b|f(t,s)|^qds\)^\frac{p}{q}<\i\Big\}.
        \end{align*}
           In particular,  if $f(\cd,\cd)\in  L^\i([a,b];  L^{1,\i}_{\dbF}([\cd,b];\dbR^n))$, by
           $\|f(\cd,\cd)\|_{L^{\i}_{\dbF}(\D[a,b])}$ we denote its norm, i.e.,
           \begin{equation*}
             \|f(\cd,\cd)\|_{L^{\i}_{\dbF}(\D[a,b])}\deq \esssup_{(t,\o)\in[a,b]\ts\O}  \Big|\int_{t}^{b}f(t,s)ds \Big|.
           \end{equation*}
            For brevity, we write $L^p_{\mathbb{F}}(\Delta[0,T];\mathbb{R})$, $L^{\infty}_{\mathbb{F}}([0,T];\mathbb{R})$, and $L^p_{\mathbb{F}}([0,T];\mathbb{R})$ as $L^p_{\mathbb{F}}(\Delta[0,T])$, $L^{\infty}_{\mathbb{F}}[0,T]$, and $L^p_{\mathbb{F}}[0,T]$, respectively.

       \subsection{BMO martingale}

    For $0\leq a < b  \leq T$, let $M=\{M(t),\sF_t; a\leq  t\leq b\}$ be a uniformly integrable martingale with $M(0)=0$, and we set, for $q\geq2$,
           $$\|M\|_{\text{BMO}_{\dbP,q}[a,b]}\deq \sup_{\t\in\sT[a,b]}
           \bigg\|\dbE_\t\Big[\(\langle M\rangle(b)-\langle M \rangle(\t)\)^\frac{q}{2}\Big]^{\frac{1}{q}}\bigg\|_{\i}.$$
             We write $ \|M\|_{\text{BMO}_{\dbP,2}[a,b]}$ as   $ \|M\|_{\text{BMO}_{\dbP}[a,b]}$ for short.
             If $\|M\|_{\text{BMO}_\dbP[a,b]}<\i$, then $M$ is called a \emph{BMO martingale} on $[a,b].$
          The space of all BMO martingales on $[a,b]$ is denoted by $\text{BMO}_\dbP[a,b]$.
           Note that $\text{BMO}_\dbP[a,b]$ is a Banach space under the norm $\| \cd \|_{\text{BMO}_\dbP[a,b]}.$

           We now recall the following   properties of BMO martingales, which can be found in He, Wang, Yan \cite{He-Wang-Yan-1992}
           and  Kazamaki \cite{Kazamaki-1994}.
\begin{itemize}
	\item  [$\bullet$] Denote by $\cE(M)$ the Dol\'{e}ans--Dade exponential of a continuous local martingale  $M$, i.e.,
	$\cE(M)_a^b = \exp\{M_t-\frac{1}{2} \langle M\rangle_t\}$,    for any $t \in  [a, b]$.
	If $M\in\text{BMO}_\dbP[a,b]$,     then $\cE(M)$ is a uniformly integrable martingale.
	\item[$\bullet$] The energy inequality: If $M\in\text{BMO}_\dbP[a,b]$, then for any positive integer $n$, one has
	$$
		\dbE[(\langle M\rangle(b))^n]\leq n!\|M\|_{\text{BMO}_\dbP[a,b]}^{2n}.
	$$
	\item[$\bullet$] For {$L>0$}, there exists two positive constants $c_1$ and $c_2$ depending on {$L$} such that, for any
	BMO martingale $M$ on $[a,b]$ and any one-dimensional BMO martingale $N$ such that {$\|N\|_{_{ \text{BMO}_\dbP[a,b]}}\les L$}, we have
	\begin{equation}\label{2.5}
		c_1\|M\|_{ \text{BMO}_\dbP[a,b]}\les\|\widetilde{M}\|_{ \text{BMO}_{\widetilde{\dbP}}[a,b]}\les c_2\|M\|_{ \text{BMO}_\dbP[a,b]},
	\end{equation}
	where $\widetilde{M}\deq M-\langle M,N\rangle$ and $d\widetilde{\dbP}\deq\cE(N)_a^{b}d\dbP$.
	\item [$\bullet$]   Let $p \in(1, \infty)$ and let $M$ be a one-dimensional continuous BMO martingale.
	If $\|M\|_{\text{BMO}_\dbP[a,b]}<\Phi(p)$, then $\cE(M)$ satisfies the reverse H\"{o}lder's inequality:
	\begin{equation}\label{2.2}
		\dbE_\tau\left[\cE(M)_\tau^{\infty}\right]^p \leq c_p,
	\end{equation}
	for any stopping time $\tau$, with a positive constant $c_p$ depending only on $p$.
\end{itemize}

               Next, we introduce some spaces associated with BMO martingales. Let $0\leq a < b < c\leq T$, denote
\begin{align*}
         \cZ^2_{\dbF}([a,b];\dbR^n)   =&\ \Big\{ f: \O\ts[a,b]\ra \dbR^n \big| f(\cd)\in L^{2}_{\dbF}([a,b]; \dbR^n),   \\
                      &\q     \|f(\cd)\|^2_{\cZ^2_{\dbF}[a,b]}\deq \sup\limits_{\t\in\sT[a,b]}\Big \|\dbE_\t\[\int_\t^b|f(s)|^2ds\]\Big \|_\i<\i    \Big\},\\
          \cZ^2_{\dbF}(\D[a,b];\dbR^n)   =&\ \Big\{ f: \O\ts\D[a,b]\ra \dbR^n \big| f(\cd,\cd)\in  L^{2}_{\dbF}(\D[a,b]; \dbR^n),   \\
        &\q  \|f(\cd,\cd)\|^2_{\cZ^2_{\dbF}(\D[a,b])}\deq\esssup\limits_{t\in[a,b]}
          \|f(t,\cd)\|^2_{\cZ^2_{\dbF}[t,b]}<\i    \Big\},\\
%
         %$\bullet$\q $ \overline{\cZ}^2_{\dbF}(\D^*[a,b])   =\Big\{ f: \O\ts\D^*[a,b]\ra \dbR \big| f(\cd,\cd)\in L^2_{\dbF}(\D^*[a,b]),   \\
%         \mbox{} \hskip 5cm \|f(\cd,\cd)\|^2_{\overline{\cZ}^2_{\dbF}(\D^*[a,b])}\deq \esssup\limits_{s\in[a,b]}
%         \sup\limits_{\t\in\sT[s,b]}\Big\|\dbE_\t\[\int_\t^b|f(s,t)|^2dt\]\Big\|_\i<\i    \Big\},$
%
          \cZ^2_{\dbF}([a,b]\ts [b,c];\dbR^n)   =&\  \Big\{ f: \O\ts [a,b]\ts [b,c]\ra \dbR^n \big| f(\cd,\cd)\in
        L^{2}_{\dbF}( [a,b]\ts [b,c];\dbR^n),   \\
         &\q  \|f(\cd,\cd)\|^2_{\cZ^2_{\dbF}([a,b]\ts[b,c])}\deq
         \esssup\limits_{t\in[a,b]} \|f(t,\cd)\|^2_{\cZ^2_{\dbF}[b,c]}<\i\Big\}.
\end{align*}

%
% ????????
\begin{remark} \rm
For $0\leq a \leq t\leq b\leq T$,
the process $\int_a^\cd  f(r)dW(r)$ (denoted by $f\cd W$) belongs to $\text{BMO}_{\dbP}[a,b]$ if and only if $f\in \cZ^2_{\dbF}([a,b];\dbR^n)$, i.e.,
\begin{equation}\label{4.3.1}
	\|f\cd W\|_{\text{BMO}_{\dbP}[a,b]}\equiv \|f(\cd)\|_{\cZ^2_\dbF[a,b]}.
\end{equation}
The process $\int_t^\cd f(t,r)dW(r)$  belongs to $\text{BMO}_{\dbP}[t,b]$ for almost all $t\in[a,b)$
if  and only if $f(\cd,\cd)\in \cZ^2_{\dbF}(\D[a,b];\dbR^n)$, i.e.,
\begin{equation*}
	\|f\cd W\|_{\text{BMO}_{\dbP}(\D[a,b])}\deq \esssup_{t\in[a,b]}\|f(t)\cd W\|_{\text{BMO}_{\dbP}[t,b]}
	=\esssup_{t\in[a,b]}\|f(t,\cd) \|_{\cZ^2_\dbF([t,b])}
	\equiv \|f(\cd,\cd)\|_{\cZ^2_\dbF(\D[a,b])}.
\end{equation*}
\end{remark}

%        The following definition introduces the notions of adapted solutions to Eq. (\ref{1.1}) and adapted $M$-solutions to Eq. (\ref{1.5}).

        \begin{definition}\rm
         A pair of processes $(Y(\cd), Z(\cd,\cd))\in L^2_{\dbF}[0,T]\ts L^2_{\dbF}(\D[0,T];\dbR^d)$
         (resp., $L^\i_{\dbF}[0,T]\ts\cZ^2_{\dbF}(\D[0,T];\dbR^d)$)
         is called an \emph{adapted solution} of the general mean-field BSVIE (\ref{1.1}) in the linear growth case (resp., quadratic growth case) if it satisfies the equation in the usual It\^{o}'s sense for Lebesgue almost all $t \in [0,T]$.

               \end{definition}

  Finally, we recall the Wasserstein distance to end this section.
  For $q\geq1$, let $\mathcal{P}_{q}(\mathbb{R}^{d})$ be the set of all probability measures $\mu$ on $(\mathbb{R}^{d}, \sB(\mathbb{R}^{d}))$ with finite $q$-th moment, i.e., $\int_{\mathbb{R}^{d}}|x|^{q} \mu(d x)<\infty$. Here $\sB(\mathbb{R}^{d})$ denotes the Borel $\sigma$-field over $\mathbb{R}^{d}$.
In addition, the set $\mathcal{P}_{q}(\mathbb{R}^{d})$ is endowed with the following $q$-Wasserstein metric: for $\mu, \nu \in \mathcal{P}_{q}(\mathbb{R}^{d})$,
\begin{align*}
  \cW_{q}(\mu, \nu)\deq \inf \bigg\{  \(\int_{\mathbb{R}^{d} \times \mathbb{R}^{d}}|x-y|^{q}  \rho(d x d y)\)^{\frac{1}{q}}\Bigm|
    \rho \in \mathcal{P}_{q}(\mathbb{R}^{2 d}),\  \rho(\cd \times \mathbb{R}^{d})=\mu(\cd),\ \rho(\mathbb{R}^{d} \times \cd)=\nu(\cd)\bigg\}.
\end{align*}
Note that for $\xi,\eta\in\cP_q(\dbR^d),$  $\cW_q(\dbP_\xi, \dbP_\eta)\leq \|\xi-\eta\|_{L^q(\O)}\deq  \{\dbE|\xi-\eta|^q\}^\frac{1}{q}.$

%Now we let $p=2$ and suppose that there exists a sub-$\sigma$-algebra $\sG$ of $\sF$ which is independent of $\sF_{\infty}$ and will be assumed ``rich enough'', as explained below:
%%
%for every $\mu \in \mathcal{P}_{2}(\mathbb{R}^{d})$ there is a random variable $\vartheta \in L^{2}_{\sG}(\Om; \mathbb{R}^{d})$ such that $\mathbb{P}_{\vartheta}=\mu$. It is well known that the probability space $([0,1], \sB([0,1]), d x)$ has this property.

%%%%%%%%%%%%%%%%%%%%%%%%%%%%%%%%%%%%%%%%%%%%%%%%%%%%%%%%%%%%%%%%%%%%%%%%%%%%%%%%%%%%%%%%%%%%%%%%%%%%%%%%%%%%%%%%%%%%%%%%%%%%%%%%%%%%%%%%%%%%

%%%%%%%%%%%%%%%%%%%%%%%%%%%%%%%%%%%%%%%%%%%%%%%%%%%%%%%%%%%%%%%%%%%%%%%%%%%%%%%%%%%%%%%%%%%%%%%%%%%%%%%%%%%%%%%%%%%%%%%%%%%%%%%%%%%%%%%%%%%%

\section{Solvability of Global Solutions}\label{sec 3}

In this section, we would like to prove the global solvability of BSVIE  (\ref{1.1}).
Before presenting the main results, we first introduce the $t$-parameterization approach for mean-field BSVIEs.

\begin{proposition}\label{pro3.1}\rm
For any given $(y(\cd), z(\cd,\cd))$ and for almost all $t\in[0,T]$,
if the following mean-field BSDE
\begin{equation}\label{3.211}
\cY(t,s)=\psi(t)+\int_s^Tg(t,r,y(r),\cZ(t,r),\dbP_{y(r)}, \dbP_{\cZ(t,r)})dr-\int_s^T\cZ(t,r)dW(r),\q s\in[t,T]
\end{equation}
admits a unique solution, denoted by $(\cY(t,\cd),\cZ(t,\cd))$,
then the following mean-field BSVIE
\begin{equation}\label{3.1110}
Y(t)=\psi(t)+\int_t^Tg(t,s,y(s),Z(t,s),\dbP_{y(s)}, \dbP_{Z(t,s)})ds-\int_t^TZ(t,s)dW(s), \q  t\in[0,T]
\end{equation}
also admits a unique solution $(Y(\cd),Z(\cd,\cd))$.
Moreover,
\begin{align*}
  Y(t) = \mathcal{Y}(t,t)\q  \text{and}\q Z(t,s) = \mathcal{Z}(t,s)\q  \text{for}\ (t,s) \in \Delta[0,T].
\end{align*}
\end{proposition}

\begin{proof}\label{3.1109}
	For the existence, for any fixed $t\in[0,T]$, we
set
$$Y(t)\deq \cY(t,t)\q\hbox{and}\q  Z(t,s)\deq \cZ(t,s),\q (t,s)\in\D[0,T].$$
 Obviously,  $(Y(\cd),Z(\cd,\cd))$ is a solution of BSVIE  (\ref{3.1110}).

Regarding uniqueness, suppose $(\overline{Y}(\cdot), \overline{Z}(\cdot,\cdot))$ is another adapted solution of BSVIE (\ref{3.1110}). By $(\overline{\cY}(t,\cd), \overline{\cZ}(t,\cd))$ we denote  the unique adapted solution of the following mean-field  BSDE (parameterized by $t\in[0,T]$):
	\begin{equation}\label{3.1111}
		\overline{ \mathcal{Y}}(t,s)=\psi(t)+\int_s^Tg(t,r,y(r),\overline{Z}(t,r),\dbP_{y(r)}, \dbP_{\overline{Z}(t,r)})dr-\int_s^T\overline{\cZ}(t,r)dW(r), \q  s\in[t,T].
	\end{equation}
Setting $s = t$ and taking the conditional expectation $\mathbb{E}_t[\cdot](=\dbE[\cd|\sF_t])$ on both sides of BSDE (\ref{3.1111}), and then comparing the resulting expression with the BSVIE satisfied by $(\overline{Y}(\cdot), \overline{Z}(\cdot,\cdot))$, i.e.,
\begin{equation*}
\overline{Y}(t)=\psi(t)+\int_t^Tg(t,s,y(s),\overline{Z}(t,s),\dbP_{y(s)}, \dbP_{\overline{Z}(t,s)})ds-\int_t^T\overline{Z}(t,s)dW(s), \q t\in[0,T],
\end{equation*}
we find that for almost all $t\in[0,T]$,
	\begin{equation} \label{3.1112}
		\overline{Y}(t)=	\overline{\cY}(t,t)\q
		 \hb{and}\q
		\overline{Z}(t,s)=\overline{\cZ}(t,s),\q  s\in[t,T].
	\end{equation}
Substituting (\ref{3.1112}) into (\ref{3.1111}) yields that
	\begin{equation*}
		\overline{ \mathcal{Y}}(t,s)=\psi(t)+\int_s^Tg(t,r,y(r),\overline{\cZ}(t,r),\dbP_{y(r)}, \dbP_{\overline{\cZ}(t,r)})dr-\int_s^T\overline{\cZ}(t,r)dW(r), \q s\in[t,T].
	\end{equation*}
Now,  comparing    the above equation with (\ref{3.211}),  it follows that
	\begin{equation*}
		\cY(t,\cd)=\overline{\cY}(t,\cd)\q\hb{and} \q \cZ(t,\cd)=\overline{\cZ}(t,\cd),\q  t\in[0,T].
	\end{equation*}
	Therefore,  one has
	\begin{equation*}
		Y(t)=\cY(t,t)=\overline{\cY}(t,t)=\overline{Y}(t)\q\hb{and}\q  Z(t,s)=\cZ(t,s)=\overline{\cZ}(t,s)=\overline{Z}(t,s),\q  (t,s)\in\D[0,T].
	\end{equation*}
	This completes the proof.
\end{proof}

\subsection{Linear growth case}\label{subsec3.1}

In this subsection, we consider the case where  the generator $g$ is of linear growth with respect to $z$.

\begin{assumption}\label{ass3.1}\rm
	Suppose that $g:  \O\ts \D[0,T]\ts \dbR\ts\dbR^d\ts \cP_2(\dbR)\ts\cP_2(\dbR^d)\ra \dbR $ is
	$\sF_T\otimes\sB(\D[0,T]\ts \dbR\ts \dbR^{d}\ts \cP_2(\dbR)\ts  \cP_2(\dbR^{d}))$
	measurable  such that $s\mapsto g(t,s,y,z, \mu,\nu)$ is $\dbF$-progressively measurable for
	all $(t,y,z,\mu,\nu)\in  [0,T]   \ts \dbR\ts  \dbR^{d} \ts \cP_2(\dbR) \ts
	\cP_2(\dbR^{d}) $ and assume that $\psi: \O\ts[0,T]\ra \dbR$ is an $\sF_T$-measurable stochastic process. Moreover, they satisfy the following:
	\begin{enumerate}[~~\,\rm (i)]
		\item There exists a positive constant $L$ such that for any $(t,s)\in\D [0,T]$  and  $y,\bar{y}\in \dbR$, $z,\bar{z}\in \dbR^d$, $\mu,\bar{\mu}\in\cP_2(\dbR)$,
		$\nu,\overline{\nu}\in\cP_2(\dbR^d)$, $\dbP$-a.s.,
		$$
		\begin{aligned}
			&|g(t,s,y,z,\mu,\nu)-g(t,s,\bar{y}, \bar{z},\bar{\mu},\bar{\nu})|
			 \leq L (|y-\bar{y}|+|z-\bar{z}|+\cW_2(\mu,\bar{\mu})+ \cW_2(\nu,\overline{\nu})\big).
		\end{aligned}$$
       \item The free term $\psi(\cdot) \in L^2_{\mathcal{F}_T}([0,T])$ and the mapping $(t,s) \mapsto g(t,s,0,0,\delta_0,\delta_0)$ belongs to $L^2_{\mathbb{F}}(\Delta [0,T])$.
		
	\end{enumerate}
\end{assumption}

\begin{theorem}\label{th3.1}\rm
Under \autoref{ass3.1},  BSVIE (\ref{1.1}) admits a unique solution
$(Y(\cd), Z(\cd,\cd))\in L^2_{\dbF}[0,T]\ts L^2_{\dbF}(\D [0,T];\dbR^d)$.
\end{theorem}
\begin{proof}
For any $(y(\cdot), z(\cdot,\cdot)) \in L^2_{\mathbb{F}}[0,T] \times L^2_{\mathbb{F}}(\Delta[0,T];\mathbb{R}^d)$, by \cite[Theorem A.1]{Li-2018}, the mean-field BSDE (\ref{3.211}) admits a unique solution $(\mathcal{Y}(t,\cdot),\mathcal{Z}(t,\cdot)) \in L^2_{\mathbb{F}}(\Omega;C([t,T]; \mathbb{R})) \times L^2_{\mathbb{F}}([t,T];\mathbb{R}^d)$. Hence, by \autoref{pro3.1}, Eq. (\ref{3.1110}) admits a unique solution $(Y(\cdot), Z(\cdot,\cdot)) \in L^2_{\mathbb{F}}[0,T] \times L^2_{\mathbb{F}}(\Delta[0,T];\mathbb{R}^d)$. This allows us to define a mapping
\begin{equation*}
\L(y(\cd),z(\cd,\cd))\deq(Y(\cd),Z(\cd,\cd)).
\end{equation*}

Next, we prove that $\L$ is a contraction mapping. To this end, let $(\h{y}(\cd),\h{z}(\cd,\cd))\in L^2_{\dbF}[0,T]\ts L^2_{\dbF}(\D[0,T];\dbR^d)$ and define  $(\h{Y}(\cd),\h{Z}(\cd,\cd))\deq\L(\h{y}(\cd),\h{z}(\cd,\cd)).$
By $(\h{\cY}(t,\cd),\h{\cZ}(t,\cd))$ we denote the unique solution to
Eq. (\ref{3.211}) with $(\h{y}(\cd),\h{z}(\cd,\cd))$.
Thanks to \autoref{pro3.1}, it follows that
 $$\h{Y}(t)=\h{\cY}(t,t)\q\hb{and}\q \h{Z}(t,s)=\h{\cZ}(t,s),\q  (t,s)\in \D[0,T].$$
Denote $\Delta h = h - \widehat{h}$ for $h = y, z, \mathcal{Y}, \mathcal{Z}, Y, Z$, and let $\beta = 16L^2T + 8L^2 + 1$.
Applying It\^{o}'s  formula to $e^{\beta s}|\Delta \mathcal{Y}(t,s)|^2$ and using the Lipschitz condition for $g$ yields
\begin{align*}
&\ e^{\b s}|\D\cY(t,s)|^2+\int_s^Te^{\b r}(\b |\D \cY(t,r)|^2+|\D \cZ(t,r)|^2)dr\\
&\leq \int_s^T  e^{\b r}2L|\D \cY(t,r)|(|\D y(r)|+\{\dbE |\D y(r)|^2\}^\frac{1}{2}
 +|\D \cZ(t, r)|+ \{\dbE|\D \cZ(t, r)|^2\}^\frac{1}{2} )dr,
\end{align*}
which implies
\begin{align*}
&e^{\b s}|\D\cY(t,s)|^2+\frac{1}{2}\int_s^Te^{\b r} |\D \cZ(t,r)|^2dr
 \leq \frac{1}{8T}\int_s^Te^{\b r}(|\D y(r)|^2+\dbE |\D y(r)|^2)dr.
\end{align*}
Setting $s = t$ and recalling that $\Delta Y(t) = \Delta\mathcal{Y}(t,t)$ and $\Delta Z(t,s) = \Delta\mathcal{Z}(t,s)$ for $(t,s) \in \Delta[0,T]$, we obtain
\begin{align*}
&\ \dbE\int_0^Te^{\b t}|\D Y(t)|^2dt+\dbE\int_0^T\int_t^Te^{\b r} |\D Z(t,r)|^2drdt\\
&\leq \frac{1}{2T} \dbE\int_0^T\int_t^Te^{\b r} |\D y(r)|^2 drdt \leq\frac{1}{2}\dbE\int_0^Te^{\b t}|\D y(t)|^2dt\\
&\leq \frac{1}{2}\(\dbE\int_0^Te^{\b t}|\D y(t)|^2dt+\dbE\int_0^T\int_t^Te^{\b r} |\D z(t,r)|^2drdt\).
\end{align*}
Hence, $\L$ is a contraction mapping under the following norm
\begin{align*}
\| y(\cd),z(\cd,\cd) \|_{\cM^{2,\b}_{\dbF}}\deq\dbE\int_0^Te^{\b t}|y(t)|^2dt+\dbE\int_0^T\int_t^Te^{\b r} |z(t,r)|^2drdt.
\end{align*}
Since the norm $\|\cd \|_{\cM^{2,\b}_{\dbF}}$ is equivalent to the norm $\|\cd \|_{\cM^{2,0}_{\dbF}}$,
BSVIE \rf{1.1} admits a unique solution in
 $L^2_{\dbF}[0,T]\ts L^2_{\dbF}(\D [0,T];\dbR^d).$
\end{proof}

\subsection{Quadratic growth case}\label{subsec3.2}

In this subsection, we consider the case where  the generator $g$ is of quadratic growth with respect to $z$. Furthermore, depending on whether the generator $g$ is bounded in the distribution w.r.t. $z$ or not, we will discuss it in two separate parts.

In what follows, we always let $\a\in[0,1)$,  let $\b,\b_0,K_1,K_2,K_3,\g,\g_0,\widetilde{\g}$  be  given positive constants.  Moreover,
we assume that
$\ell:\D [0,T]\ts \O\ra \dbR$ is   a non-negative valued process,
$\phi:[0,\i)\ra [0,\i)$   is a continuous and monotonically increasing function,
and $\psi:[0,T]\ts\O\ra \dbR$ is an $\sF_T$-measurable stochastic process.

 %%%%%%%%%%%%%%%%%%%%%%%%%%%%%%%%%%%%%%%%%%%%%%%%%%%%%%%%%%%%%%%%%%%%%%%%%%%%%%%%%%%%%%%%%%%%%%%%%%%%%%%%%%%%%%%%%%%%%%%%%%%%%%%%%%%%%%%%%%%%
 \subsubsection{$g$ is bounded in $\dbP_{Z(t,s)}$}
\begin{assumption}\label{ass4.1-1}\rm
	Suppose that $g:\O\ts [0,T]\ts \dbR \ts \dbR^d \ts \cP_2(\dbR) \ts \cP_2(\dbR^d)\ra \dbR$ is
	$\sF_T\otimes\sB(\D[0,T]\ts \dbR\ts \dbR^{d}\ts \cP_2(\dbR)\ts  \cP_2(\dbR^{d}))$
	measurable  such that $s\mapsto g(t,s,y,z, \mu,\nu)$ is $\dbF$-progressively measurable for
	all $(t,y,z,\mu,\nu)\in  [0,T]   \ts \dbR\ts  \dbR^{ d} \ts \cP_2(\dbR) \ts
	\cP_2(\dbR^{d}) $, and the following conditions hold:
	\begin{enumerate}[~~\,\rm (i)]
		\item   For $(t,s)\in\D[0,T]$,  $y\in \dbR$, $z\in \dbR^d$, $\mu\in\cP_2(\dbR), \nu\in\cP_2(\dbR^d)$,
		$\dbP$-a.s.,
		\begin{equation*}
			\begin{aligned}
				& |g(t,s, y,z,\mu,\nu)|\leq \frac{\gamma}{2}|z|^2+\ell(t,s)+\b|y|
				+\b_0\cW_2(\mu,\d_{0}).
			\end{aligned}
		\end{equation*}
		\item  For $(t,s)\in\D [0,T]$  and for all $y,\bar{y}\in \dbR$, $z,\bar{z}\in \dbR^d$, $\mu,\bar{\mu}\in\cP_2(\dbR)$,
		$\nu,\overline{\nu}\in\cP_2(\dbR^d)$, $\dbP$-a.s.,
		$$
		\begin{aligned}
			&|g(t,s,y,z,\mu,\nu)-g(t,s,\bar{y}, \bar{z},\bar{\mu},\bar{\nu})|\\
			&\leq\phi\big(|y|\vee|\bar{y}|\vee \cW_2(\mu,\d_{0})\vee \cW_2(\bar{\mu},\d_{0})\big)\cd
			\big[(1+|z|+|\bar{z}|)|z-\bar{z}| )+\cW_2(\nu,\bar{\nu})]\\
			&\qq
			+\b |y-\bar{y}|+\b_0\cW_2(\mu,\bar{\mu}).
		\end{aligned}$$
		\item   The free term $\psi(\cd)$ is bounded with $\|\psi(\cd)\|_{L^\i_{\sF_T}[0,T]}\leq K_1$, and
		the process  $|\ell(\cd,\cd)|^2$ belongs to the space  $L^\i([0,T];  L^{1,\i}_{\dbF}([\cd,T];\dbR^+))$ with $ \||\ell(\cd,\cd)|^2\|_{L^{\i}_{\dbF}(\D[0,T])} \leq K_2$.
\end{enumerate}
\end{assumption}

\begin{remark}\rm
A typical example that satisfies \autoref{ass4.1-1} is
\begin{align*}
g(t,s,y,z,\mu,\nu)=(s-t)^{-\frac{1}{3}}+2|y| +|z_1|^2+\cW_2(\mu,\d_0)+\arctan \cW_2(\nu,\d_0),
\end{align*}
for $(t,s)\in\D [0,T], y\in\dbR,z=(z_1,\cds, z_d)\in\dbR^d,\mu\in\cP_2(\dbR), \nu\in \cP_2(\dbR^d).$
%
%Additionally, note that the above example does not satisfy the strictly quadratic condition in \autoref{ass4.1} below.
\end{remark}

For positive values $R_1$ and $R_2$, and a small positive constant $\e$,  we define
\begin{equation}\label{3.2}
	\begin{aligned}
		\sB_\e(R_1,R_2)\deq  \Big\{ &
		(P(\cd), Q(\cd,\cd))\in L_\dbF^\infty[T-\e,T]\ts \cZ^2_\dbF(\D[T-\e,T]; \dbR^d) \Big|\\
		& \| P(\cd)\|_{L_\dbF^\infty[T-\e,T]}\les R_1\q
		\text{and}\q \|Q(\cd,\cd)\|^2_{\cZ^2_\dbF( \D[T-\e,T]) }\les R_2\Big\},
	\end{aligned}
\end{equation}
endowed with the norm
\begin{equation*}
	\|(P(\cd), Q(\cd,\cd))\|_{\sB_\e}\deq \sqrt{\|P(\cd)\|^2_{L_{\dbF}^\i[T-\e,T]}
		+ \|Q(\cd,\cd)\|^2_{\cZ^2_{\dbF}(\D[T-\e,T])}}.
\end{equation*}

\begin{proposition}\label{pro3.4}\rm
Under \autoref{ass4.1-1},  there exists a positive constant $\e$ depending on $\b,\b_0,\g,K_1,K_2$ such that
       BSVIE (\ref{1.1}) admits a unique local solution $(Y(\cd),Z(\cd,\cd))\in \sB_\e(\bar{R}_1,\bar{R}_2)$
       with
       $$\bar{R}_1=2(K_1+\sqrt{K_2})\q  \text{and}\q  \bar{R}_2=\frac{2}{\g^2}e^{2\g K_1}+\frac{4\sqrt{K_2}}{\g} e^{2\g\bar{R}_1}.$$
\end{proposition}

We postpone the proof to the appendix to avoid disrupting the overall readability.

\begin{theorem}\label{th3.4}\rm
Let \autoref{ass4.1-1} hold, BSVIE (\ref{1.1})  admits a unique global adapted solution
	$(Y(\cd), Z(\cd,\cd))\in L^\i_{\dbF}[0,T]\ts \cZ^2_{\dbF}(\D[0,T];\dbR^d).$  Furthermore there exist two positive constants
	$\overline{M}_1$ and $\overline{M}_2$  depending on  $K_1,K_2, T, \g, \b,\b_0$ such that
	\begin{equation*}
		\|Y(\cd)\|^2_{L_{\dbF}^\i[0,T]}\leq \overline{M}_1\q\hb{and}\q\ \|Z(\cd,\cd)\|^2_{ \overline{\cZ}^2_{\dbF}(\D[0,T])}\leq \overline{M}_2.
	\end{equation*}
\end{theorem}

\begin{proof}
The proof is split into three steps.

\textbf{Step 1.}  Boundedness of local solutions.
Set $\widetilde{L}=e^{(\b^2+(\b_0)^2+1) T}[(K_1)^2+K_2+2] $ and by $\Th(\cd)$ we denote the unique solution to the
ordinary differential equation (ODE)
\begin{align*}
\Th(t)=\widetilde{L}+\widetilde{L}\int_t^T\Th(s)ds,\q t\in[0,T].
\end{align*}
Clearly, $\Th(\cd)$ is non-increasing and continuous in $t$. Moreover, by $\varrho$ we denote its maximum value, i.e.,
$$\varrho\deq\sup\limits_{t\in[0,T]}\Th(t)=\Th(0).$$
Note that due to $\|\psi(\cd)\|^2_{L^\i_{\sF_T}[0,T]}\leq (K_1)^2\leq \widetilde{L}=\Th(T)$ and thanks to \autoref{pro3.4},
BSVIE (\ref{1.1})  exists a unique local adapted solution $(Y(\cd),Z(\cd,\cd))$ on the interval
$[T-\k_\varrho,T]$, where $\k_\varrho$ is a sufficiently small constant depending on $\varrho$.
Under  \autoref{ass4.1-1}, according to Hao et al. \cite[Theorem 3.5]{Hao-Hu-Tang-Wen-2022}, the mean-field BSDE (\ref{3.211})  with $(Y(\cd), Z(\cd,\cd))$
 admits a unique solution $(\cY(t,\cd), \cZ(t,\cd))$ on the interval $[T-\k_\varrho,T]$.
Then, by \autoref{pro3.1}, one has
$$Y(t)=\cY(t,t)\q\hbox{and}\q Z(t,s)=\cZ(t,s), \q (t,s)\in \D[T-\k_\varrho,T].$$
On the other hand,for almost all $t\in[0,T]$, Eq. (\ref{3.211}) can be rewritten as
\begin{equation*}
\begin{aligned}
\cY(t,s)=&\psi(t)+ \int_s^Tg(t,r, Y(r),\cZ(t,r), \dbP_{Y(r)}, \dbP_{\cZ(t,r)})-g(t,r,0,\cZ(t,r), \d_0, \dbP_{\cZ(t,r)})\\
&\qq\qq+g(t,r,0,0,\d_0,\dbP_{\cZ(t,r)})dr-\int_s^T\cZ(t,r)d\widetilde{W}(r;t),\q   s\in[t,T],
\end{aligned}
\end{equation*}
where $$\widetilde{W}(r;t)\deq W(r)-\int_t^r \widetilde{L}(t,s)ds,\  r\in[t,T]$$ and
\begin{equation*}
\left\{
\begin{aligned}
& g(t,r,0, \cZ(t,r), \d_0, \dbP_{\cZ(t,r)})-g(t,r,0, 0, \d_0, \dbP_{\cZ(t,r)})= \widetilde{L}(t,r)\cZ(t,r),\\
&| \widetilde{L}(t,r)|\leq \phi(0)(1+|\cZ(t,r)|),\q  r\in[t,T].
\end{aligned}
\right.
\end{equation*}
By Girsanov theorem, $\widetilde{W}(\cd;t)$ is a Brownian motion under
the probability $d\wt{\dbP}^t\deq\cE ( \widetilde{L}(t,\cd)\cd W)_t^Td\dbP.$
%Let $\wt\dbE^t_r[\cd]\deq\wt\dbE^t[\cd|\sF_r]$ denote the
%conditional expectation with respect to $\wt\dbP^t$.
Next, we set $\th=\b^2+(\b_0)^2+1$ and apply It\^{o}'s formula to $e^{\th s} |\cY(t,s)|^2$, which leads to the following result: for any $s\in [t,T]$,
\begin{equation}\label{3.20}
\begin{aligned}
&\ e^{\th s}|\cY(t,s)|^2 +\wt\dbE^t_s\[\int_s^Te^{\th r}(\th|\cY(t,r)|^2+ |\cZ(t,r)|^2)dr\]\\
&\leq\wt\dbE^t_s[e^{\th T}|\psi(t)|^2]
+\wt\dbE^t_s\int_s^Te^{\th r}2|\cY(t,r)|
(\b|Y(r)|+\b_0\cW_2(\dbP_{Y(r)},\d_{0})+\ell(t,r) )dr\\
&\leq e^{\th T}[(K_1)^2+K_2]+2e^{\th T}\int_s^T\|Y(\cd)\|^2_{L^\i_{\dbF}[r,T]}dr+ \wt\dbE^t_s \int_s^Te^{\th r}\th|\cY(t,r)|^2dr.
\end{aligned}
\end{equation}
Taking $s=t$, it follows that
 $|\cY(t,t)|^2 \leq  e^{\th T}[(K_1)^2+K_2]+2e^{\th T}\int_t^T\|Y(\cd)\|^2_{L^\i_{\dbF}[r,T]}dr,$
which implies that
 $$\|Y(\cd)\|^2_{L^\i_{\dbF}[t,T]}\leq  e^{\th T}[(K_1)^2+K_2]+2e^{\th T}\int_t^T\|Y(\cd)\|^2_{L^\i_{\dbF}[r,T]}dr.$$
Hence
 $$\|Y(\cd)\|^2_{L^\i_{\dbF}[t,T]}\leq \Th(t)\leq \Th(0)=\varrho,\q  t\in[T-\k_\varrho,T].$$
Inserting this estimate into (\ref{3.20}) provides that for any $s\in[t,T]$,
\begin{align}\label{3.8-1111}
|\cY(t,s)|^2\leq e^{\th T} [(K_1)^2+K_2+ 2\varrho T]\deq\widehat{M}_1.
\end{align}
Clearly,
\begin{align}\label{3.8-111}
\|Y(\cd)\|^2_{L^\i_{\dbF}[t,T]}\leq \overline{M}_1 \q\hbox{and}\q \| \cY(t,\cd)\|^2_{L^\i_{\dbF}[t,T]}\leq \overline{M}_1,
\end{align}
where  $\overline{M}_1\deq\max\big\{\varrho,   \widehat{M}_1\big\}.$

Next, we would like to find   the bounds for $\|Z(\cd,\cd)\|^2_{ \cZ^2_\dbF(\D[T-\k_\varrho,T])}.$ For this, we define
\begin{align}\label{3.9-111}
\Xi(x)\deq\frac{1}{\g^2}\(\exp(\g  |x|-\g |x|-1) \), \q x\in\dbR,
\end{align}
          where $\g$ is given in \autoref{ass4.1-1}. It is easy to check that the function $\Xi(\cd)$ satisfies  the following results,
          \begin{equation*}
          \Xi'(x)=\frac{1}{\gamma}\big[\exp(\gamma |x|)-1\big]\sgn(x), \quad
            \Xi''(x)=\exp(\gamma|x|), \quad  \Xi''(x)-\gamma |\Xi'(x)|=1,
              \end{equation*}
              where  the notation $\sgn(x)=-\mathbf{1}_{\{x\leq 0\}}+\mathbf{1}_{\{x> 0\}}$.
Now,      applying It\^{o}'s formula to $\Xi(\cY(t,s))$, it yields that for any stopping time $\t\in\sT[t,T]$,
\begin{align*}
   \Xi(\cY(t,\t))=&\ \Xi(\cY(t,T))
+\int_{\t}^{T}\Xi'(\cY(t,r))g(t,r, Y(r),\cZ(t,r), \dbP_{Y(r)}, \dbP_{\cZ(t,r)})dr\\   &-\int_{\t}^{T}\Xi'(\cY(t,r))\cZ(t,r)dW(r)-\frac{1}{2}\int_{\t}^{T}\Xi''(\cY(t,r))
|\cZ(t,r)|^2dr\\
   \les &\ \Xi(\psi(t))+\int_{\t}^{T}|\Xi'(\cY(t,r))|
  [\ell(t,r)+\b|Y(r)|+\b_0\|Y(r)\|_{L^2(\Omega)}\big] dr\\
&   -\int_{\t}^{T}\Xi'(\cY(t,r))\cZ(t,r)dW(r)
   +\frac{1}{2}\int_{\t}^{T}\big[\gamma|\Xi'(\cY(t,r))| -\Xi''(\cY(t,r))\big]|\cZ(t,r)|^2dr.
\end{align*}
Then, note that $\gamma|\Xi'(x)| -\Xi''(x)=-1, x\in\dbR$, we have
\begin{equation}\label{3.9-112}
\begin{aligned}
  \mathbb{E}_{\t}\int_{\t}^T|\cZ(t,r)|^2dr
   & \leq 2\Xi(K_1)+2|\Xi'(\sqrt{\overline{M}_1})|\big[\sqrt{K_2 T}+ (\b+\b_0)\sqrt{\overline{M}_1} T\big]\deq \overline{M}_2,
\end{aligned}
\end{equation}
which implies that
\begin{equation*}
\|Z(\cd,\cd)\|^2_{ \cZ^2_\dbF(\D[0,T])}=\| \cZ(\cd,\cd)\|^2_{ \cZ^2_\dbF(\D[0,T])}\leq \overline{M}_2.
\end{equation*}

\textbf{Step 2.}    Solvability of related  mean-field SFIEs.
	
	\ms
	
	Note that Step 1 implies  that the  solution $(Y(s), Z(t,s))$  of  BSVIE (\ref{1.1})
	has been determined  when $t\in[T-\k_\varrho,T]$ with   $(t,s)\in\D[T-\k_\varrho,T]$.  Next,  we want to solve BSVIE \rf{1.1} when $t\in[0,T-\k_\varrho]$.  For this, we    rewrite BSVIE (\ref{1.1})  in the interval  $[0,T-\k_\varrho]$ as following:
	\begin{equation}\label{4.37-1}
		\begin{aligned}
			Y(t)= &\  \psi^{T-\k_\varrho}(t)+\int_t^{T-\k_\varrho}g(t,s,Y(s),Z(t,s),\dbP_{Y(s)}, \dbP_{Z(t,s)})ds\\
			&\  -\int_t^{T-\k_\varrho}Z(t,s)dW(s),\q  t\in[0,T-\k_\varrho],
		\end{aligned}
	\end{equation}
	where
	\begin{equation}\label{4.38}
		\begin{aligned}
			\psi^{T-\k_\varrho}(t)= &\  \psi(t)+\int^T_{T-\k_\varrho}g(t,s,Y(s),Z(t,s), \dbP_{Y(s)}, \dbP_{Z(t,s)})ds\\
			& \ -\int_{T-\k_\varrho}^{T}Z(t,s)dW(s),\q  t\in[0,T-\k_\varrho].\\
		\end{aligned}
	\end{equation}
	Note that Eq. (\ref{4.37-1}) is
	a mean-field BSVIE
	if $\psi^{T-\k_\varrho}(\cd)\in  L^\i_{\sF_{T-\k_\varrho}}[0,T-\k_\varrho] $.
	However, it is a pity that the term  $\psi^{T-\k_\varrho}(t)$ with $t\in[0,T-\k_\varrho]$ can not  been determined yet even though  the pair $(Y(s),Z(t,s))$ is well-defined on $\D[T-\k_\varrho,T]$.
	In fact, on the right-hand side of Eq. (\ref{4.38}),
	we only know the values of $Y(s)$ and $\dbP_{Y(s)}$ with $s\in[T-\k_\varrho,T]$,   but have no more information for the values of
	$Z(t,s)$ and $\dbP_{Z(t,s)}$ with $(t,s)\in [0,T-\k_\varrho]\ts[T-\k_\varrho,T]$.
	Furthermore,  in order to solve  Eq. \rf{4.37-1}, we need to show that the free term $\psi^{T-\k_\varrho}(t)$ is
	$\sF_{T-\k_\varrho}$-measurable, not just $\sF_T$-measurable.
	Hence, it is necessary to solve $\psi^{T-\k_\varrho}(t)$ when $t\in[0,T-\k_\varrho]$. The good news is that it could be  solved  by regarding  Eq. (\ref{4.38})
	as \emph{a mean-field  stochastic Fredholm integral equation}  (mean-field SFIE, for short).
	
	\ms

	Next, we first prove the existence of mean-field SFIE (\ref{4.38}).
	According to  Hao et al. \cite[Theorem 3.5]{Hao-Hu-Tang-Wen-2022}, for almost all $t\in[0,T],$
    the  mean-field QBSDE  (\ref{3.211}) with $(Y(\cd),Z(\cd,\cd))$ on the interval $[T-\k_\varrho,T]$ admits a unique solution
	$(\cY(t,\cd),  \cZ(t,\cd))\in L^\i_{\dbF}(\O;C([T-\k_\varrho,T];\dbR))\ts \overline{\cZ}^2_{\dbF}([T-k_\varrho,T];\dbR^d).$
    Now, we  set $\psi^{T-k_\varrho}(t)= \cY(t,T-k_\varrho)$ and $Z(t,s)= \cZ(t,s)$ when $(t,s)\in[0,T-k_\varrho]
   \ts[T-k_\varrho,T]$. Consequently, $(\psi^{T-k_\varrho}(\cd), Z(\cd,\cd))  $ is a solution of mean-field SFIE (\ref{4.38}),
   and $\psi^{T-k_\varrho}(t)$ is $\sF_{T-k_\varrho}$-measurable.
	Furthermore, thanks to  (\ref{3.8-111}) and  (\ref{3.9-112}), one has that for any stopping time
	$\t\in[T-k_\varrho,T]$,
	\begin{equation}\label{3.9200}
		 \|  \cY (t,\cd)\|^2_{L^\i_{\dbF}[t,T]}\leq \overline{M}_1\q \text{and}\q   \dbE_\t\int_\t^T  | \cZ(t,u)|^2du\leq \overline{M}_2.
	\end{equation}
	 Hence, it follows from (\ref{3.8-1111}) and (\ref{3.9200}) that
	$$
     \|\psi^{T-k_\varrho}(\cd)\|^2_{L^\i_{\sF_{T-k_\varrho}}[0,T-k_\varrho]}=\esssup_{(\o,t)\in\O\ts [0,T-k_\varrho]}|\psi^{T-k_\varrho}(t)|^2
      =\esssup_{(\o,t)\in\O\ts[0,T-k_\varrho]}|\cY(t,T-k_\varrho)|^2\leq  \overline{M}_1,
    $$
   and

\begin{equation}\label{4.41}
		 \|Z(\cd,\cd)\|^2_{\cZ^2_\dbF([0,T-k_\varrho]\ts[T-k_\varrho,T])} \leq \overline{M}_2.
	\end{equation}

	\ms

	 Now, we prove the uniqueness of mean-field SFIE (\ref{4.38}).  Assume that the pair
	$(\overline{\psi}^{T-k_\varrho}(\cd), \overline{Z}(\cd,\cd))\in L^\i_{\sF_{T-k_\varrho}}[0,T-k_\varrho]\ts
	\cZ^2_\dbF([0,T-k_\varrho]\ts[T-k_\varrho,T]; \dbR^d)$ is another solution to mean-field SFIE (\ref{4.38}).
	Then, for $t\in[0,T-{k_\varrho}]$,
	\begin{equation}\label{4.43}
		\begin{aligned}
			&  \psi^{T-k_\varrho}(t)-\overline{\psi}^{T-k_\varrho}(t)+\int_{T-k_\varrho}^T (Z(t,s)-\overline{Z}(t,s))dW(s)      \\
			&  =\int_{T-k_\varrho}^T\Big\{\overline{L}(t,s)[Z(t,s)-\overline{Z}(t,s)]+\overline{J}(t,s)\Big\}ds,
		\end{aligned}
	\end{equation}
	where
	\begin{equation*}
		\begin{aligned}
			\ds  &|\overline{L}(t,s)|\leq \phi(|Y(s)|\vee \cW_2(\dbP_{Y(s)},\d_0))
			(1+|Z(t,s)|+|\overline{Z}(t,s)|),\\
			\ns\ds  &\overline{J}(t,s)=g(t,s,Y(s),\overline{Z}(t,s),\dbP_{Y(s)}, \dbP_{Z(t,s)})-g(t,s,Y(s),\overline{Z}(t,s),\dbP_{Y(s)}, \dbP_{\overline{Z}(t,s)}).
		\end{aligned}
	\end{equation*}
	The fact
	\begin{equation*}
		\begin{aligned}
			\|\overline{L}(t,\cd)\|_{\cZ^2_{\dbF}(  [T-k_\varrho,T])}
			\leq   \phi(\|Y(\cd)\|_{L_\dbF^\infty[T-k_\varrho,T]})\Big\{1+ \|Z(t,s)\|_{\cZ^2_\dbF( [T-k_\varrho,T])}
               +\|\overline{Z}(t,s)\|_{\cZ^2_\dbF( [T-k_\varrho,T])}\Big\}
		\end{aligned}
	\end{equation*}
	 implies that  $\cE(\overline{L}(t,\cd)\cd W)_{T-k_\varrho}^T$ is a uniformly integrable martingale. Thus, the   process
	 $\overline{W}(s;t)=W(s)-\int_{T-k_\varrho}^s\overline{L}(t,u)du$ is a standard Brownian motion under the probability $\overline{\dbP}^t$,
     defined by $d\overline{\dbP}^t=\cE(\overline L(t,\cd)\cd W)_{T-k_\varrho}^{T}d\dbP.$
	From this, (\ref{4.43})  can be rewritten as
	\begin{equation*}
		\begin{aligned}
			&  \psi^{T-k_\varrho}(t)-\overline{\psi}^{T-k_\varrho}(t)+\int_{T-k_\varrho}^T (Z(t,s)-\overline{Z}(t,s))d\overline{W}(s;t)=\int_{T-k_\varrho}^T \overline{J}(t,s)ds.
		\end{aligned}
	\end{equation*}
	Now,   thanks to \autoref{ass4.1-1}-$\mathrm{(ii)}$, one has, for $(t,s)\in[0,T-k_\varrho]\ts[T-k_\varrho,T]$,
	\begin{equation*}
		\begin{aligned}
			\ds    |\overline{J}(t,s)|\leq &\ \phi(|Y(s)|\vee \|Y(s)\|_{L^2(\O)})
			\|Z(t,s)-\overline{Z}(t,s)\|_{L^2(\O)}\\
			\ns\ds     \leq&\ \phi(\|Y(\cd)\|_{L_\dbF^\infty[T-k_\varrho,T]}) \|Z(t,s)-\overline{Z}(t,s)\|_{L^2(\O)}.
		\end{aligned}
	\end{equation*}
	Consequently, it yields from H\"{o}lder's inequality that, for any stopping time $\t\in\sT[T-k_\varrho,T]$,
	\begin{equation*}
		\begin{aligned}
			&\ \overline{\dbE}^{t}_{\t}\[|\psi^{T-k_\varrho}(t)-\overline{\psi}^{T-k_\varrho}(t)|^2\]
			+\overline{\dbE}^{t}_{\t}\[\int_{T-k_\varrho}^T |Z(t,s)-\overline{Z}(t,s)|^2ds\] \\
			&\leq k_\varrho\phi^2(\|Y(\cd)\|_{L_\dbF^\infty[T-k_\varrho,T]})
		 \(\int_{T-k_\varrho}^{T}\|Z(t,s)-\overline{Z}(t,s)\|^2_{L^2(\O)}ds\).
		\end{aligned}
	\end{equation*}
	%
%	where $\dbE^{\overline{\dbP}^t}_{\t}[\cd]$ is the conditional expectation $\dbE^{\overline{\dbP}^t}[\cd|\sF_{\t}]$
%	under the probability $\overline{\dbP}^t$.
	Again, H\"{o}lder's inequality implies
	\begin{equation*}
		\begin{aligned}
			\ds   \int_{T-k_\varrho}^{T}\|Z(t,s)-\overline{Z}(t,s)\|^2_{L^2(\O)}ds& =
		 \dbE\Big\{\dbE_{T-k_\varrho}   \int_{T-k_\varrho}^{T}|Z(t,s)-\overline{Z}(t,s)|^2ds\Big\}\\
			\ns\ds      &\leq \|Z(\cd,\cd)-\overline{Z}(\cd,\cd)\|^2_{\cZ^2_{\dbF}([0,T-k_\varrho]\ts [T-k_\varrho,T])}.
		\end{aligned}
	\end{equation*}
	Hence, we deduce that for almost all $t\in[0,T-k_\varrho]$ and for $\t\in\sT[T-k_\varrho,T],$
	\begin{equation} \label{4.42}
		\begin{aligned}
			\ds
			&\ \overline{\dbE}^{t}_{\t} |\psi^{T-k_\varrho}(t)-\overline{\psi}^{T-k_\varrho}(t)|^2 +
               \overline{\dbE}^{t}_{\t}\int_{T-k_\varrho}^T |Z(t,s)-\overline{Z}(t,s)|^2ds \\
			\ns\ds    & \leq   k_\varrho \phi^2\(\sqrt{\overline{M}_1}\)
			\|Z(\cd,\cd)-\overline{Z}(\cd,\cd)\|^2_{\cZ^2_{\dbF}([0,T-k_\varrho]\ts [T-k_\varrho,T])}.
		\end{aligned}
	\end{equation}
	 For almost all
     $t\in [0,T-k_\varrho]$, it follows from (\ref{2.5}) and (\ref{4.42}) that
	\begin{equation*}
		\begin{aligned}
			\ds       &\  \|(Z(t,\cd) -\overline{Z}(t,\cd)\|^2_{\cZ^2_{\dbF} [T-k_\varrho,T]}
			=\|(Z(t,\cd)-\overline{Z}(t,\cd))\cd W\|^2_{ \text{BMO}[T-k_\varrho,T]}\\
			 \ds       &\leq   \frac{1}{c_1}\|(Z(t,\cd)-\overline{Z}(t,\cd))\cd
			\overline{W}( \cd;t)\|^2_{\text{BMO}_{\overline{\dbP}^t}[T-k_\varrho,T]}\\
             \ds       &=   \frac{1}{c_1}    \sup\limits_{\t\in\sT[T-k_\varrho,T]}
            \Big\| \overline{\dbE}^{t}_\t\[\int^T_{T-k_\varrho}|Z(t,s)-\cl Z(t,s)|^2ds\]\Big\|_\i\\
			 \ds      & \leq  \frac{k_\varrho}{c_1} \phi^2\(\sqrt{\overline{M}_1}\)
			\|Z(\cd,\cd)-\overline{Z}(\cd,\cd)\|^2_{\cZ^2_{\dbF}([0,T-k_\varrho]\ts [T-k_\varrho,T])}.
		\end{aligned}
	\end{equation*}
	Hence, recalling  the definition of $\|\cd\|_{\overline{\cZ}^2_{\dbF}([0,T-k_\varrho]\ts [T-k_\varrho,T])}$, we have by  (\ref{2.5}) that
    \begin{equation*}
		\begin{aligned}
			\ds             &\  \|(Z(\cd,\cd)-\overline{Z}(\cd,\cd)\|^2_{\cZ^2_{\dbF}([0,T-k_\varrho]\ts [T-k_\varrho,T])}\\
			\ns\ds       &\leq    \frac{k_\varrho}{c_1} \phi^2(\sqrt{\overline{M}_1})
			\|Z(\cd,\cd)-\overline{Z}(\cd,\cd)\|^2_{\cZ^2_{\dbF}([0,T-k_\varrho]\ts [T-k_\varrho,T])}.
		\end{aligned}
	\end{equation*}
	Note that  $k_\varrho$ is sufficiently small, then the above inequality implies   that
	\begin{equation*}
		\begin{aligned}
			\|Z(\cd,\cd)-\overline{Z}(\cd,\cd)\|^2_{\cZ^2_{\dbF}([0,T-k_\varrho]\ts [T-k_\varrho,T])}=0,
		\end{aligned}
	\end{equation*}
	in other words,
	\begin{equation*}
		Z(t,s)=\overline{Z}(t,s),\q \text{a.s.},\q (t,s)\in [0,T-k_\varrho]\ts [T-k_\varrho,T].
	\end{equation*}
	Consequently, thanks to  (\ref{4.42}) with  $\t=T-k_\varrho$, it yields that for almost all $t\in[0,T-k_\varrho]$,
	\begin{equation*}
		\begin{aligned}
			& |\psi^{T-k_\varrho}(t)-\overline{\psi}^{T-k_\varrho}(t)|^2\leq0,
		\end{aligned}
	\end{equation*}
	which means
	\begin{equation*}
		\psi^{T-\k_0}(t)=\overline{\psi}^{T-k_\varrho}(t),\q \text{a.s.},\q  t\in [0,T-k_\varrho].
	\end{equation*}
	Hence, the uniqueness of mean-field SFIE (\ref{4.38})  is proved.
	
	\ms
	
	\textbf{Step 3.}  Complete the proof by induction.
	
	\ms
	
	Based on the above steps,  we can prove the solvability of BSVIE \rf{1.1} by induction. Before going further,    let us recall what we have obtained. Note that from {Step 1} and {Step 2}, we have uniquely  confirmed the following values
	\begin{equation*}
		\left\{
		\begin{aligned}
			&Y^1(t),\q\q\ t\in[T-k_\varrho,T],\\
			&Z^1(t,s),\q\ (t,s)\in\D[T-k_\varrho,T]\bigcup\([0,T-k_\varrho]\ts [T-k_\varrho,T]\).
		\end{aligned}
		\right.
	\end{equation*}
	Next, let us solve BSVIE (\ref{1.1}) on the interval $[T-2k_\varrho, T-k_\varrho].$ According to  (\ref{4.41}), similar to {Step 1} and {Step 2},
	we can uniquely determine the following values
	\begin{equation*}
		\left\{
		\begin{aligned}
			&Y^2(t),\q\q\ t\in[T-2k_\varrho,T-k_\varrho),\\
			&Z^2(t,s),\q\  (t,s)\in\D[T-2k_\varrho,T-k_\varrho)\bigcup\([0,T-2k_\varrho]\ts [T-2\k_0,T-k_\varrho]\).
		\end{aligned}
		\right.
	\end{equation*}
	Now,  by defining
	\begin{equation*}
		\left\{
		\begin{aligned}
			\ds    Y(t)=    & \ Y^1(t)\mathbf{1}_{[T-k_\varrho,T]}(t)+Y^2(t)\mathbf{1}_{[T-2k_\varrho,T-k_\varrho)}(t),\\
			\ns\ds       Z(t,s)= &\  Z^1(t,s)\mathbf{1}_{\D[T-k_\varrho,T]\bigcup\big([0,T-k_\varrho]\ts [T-k_\varrho,T]\big)}(t,s)\\
			\ns\ds           & \ +Z^2(t,s)\mathbf{1}_{\D[T-2k_\varrho,T-k_\varrho)\bigcup\big([0,T-2\k_\varrho]\ts
                                  [T-2k_\varrho,T-k_\varrho]\big)}(t,s),
		\end{aligned}
		\right.
	\end{equation*}
	one has that the pair $(Y(\cd),Z(\cd,\cd))$ is the unique solution of mean-field SFIE (\ref{1.1}) on $[T-2\k_\varrho,T].$
	Moreover, on the interval $[T-2k_\varrho,T]$, due to the terminal value $\|\psi(\cd)\|^2_\i\leq (K_1)^2\leq  \overline{M}_1,$
	similar to {Step 1}, one could see
	\begin{equation*}
		\begin{aligned}
			\|Y(\cd)\|^2_{L^\i_{\dbF}[T-2\k_0,T]}\leq  \overline{M}_1\q\hb{and}\q
			\|Z(\cd,\cd)\|^2_{\cZ^2_\dbF (\D[T-2\k_0,T]\bigcup ([0,T-2\k_0]\ts [T-2\k_0,T]) )} \leq \overline{M}_2.
		\end{aligned}
	\end{equation*}
	Repeating the above procedure, the existence and uniqueness of adapted solutions to BSVIE \eqref{1.1} can be proved in finitely many steps.
\end{proof}

 \subsubsection{$g$ is unbounded in $\dbP_{Z(t,s)}$}

Let us   introduce  some constants, which will be used later. For $\d, L_0, \e_0>0$, denote
\begin{equation}\label{4.1-1}
	\begin{aligned}
		L_1=&\ \frac{(1-\alpha)\tilde{\gamma}\e_0}{8}
		\Big(\frac{1+\alpha}{2}\Big)^{\frac{1+\alpha}{1-\alpha}}
		\Big(\frac{4\gamma_0}{\tilde{\gamma}}  \Big)^{\frac{2}{1-\alpha}},\q\
		L_{2,\d}= \frac{(1-\alpha)\tilde{\gamma}\e_0}{8}\Big(\frac{1+\alpha}{2}\Big)^{\frac{1+\alpha}{1-\alpha}}
		\Big(\frac{4\d}{\tilde{\gamma}\e_0}  \Big)^{\frac{2}{1-\alpha}},\\
		L_3=&\  (K_1+K_3)\e_0+L_1 T+L_{2,L_0\g_0}T,\qq\
		L_4= L_0e^{L_0(K_1+K_3)+L_3}+(L_0+\e_0)(\b+\b_0),\\
		L_5=&\  \frac{2}{\widetilde{\g}}(\ln L_4+L_4), \q\q\qq\qq\qq\q\ \,\,\,
		L_6 =L_4e^{\frac{\tilde{\g}}{2}L_5}+L_4+\frac{\tilde{\g}}{2}L_5,\\
        L_7=&\ 3(K_1+K_3) +2L_1T+2L_{2,\g_0}T+3(\beta+\beta_0).
		%
		%L_8= \frac{4}{\tilde{\g}}\([1+(\beta+\beta_0)T]
%		L_7e^{L_7 T} +K_1+K_3+2L_1T\).
	\end{aligned}
\end{equation}
%

%\subsubsection{Type-I mean-field BSVIE}

\begin{assumption}\label{ass4.1}\rm
	Suppose that $g:\O\ts [0,T]\ts \dbR \ts \dbR^d \ts \cP_2(\dbR) \ts \cP_2(\dbR^d)\ra \dbR$  is
	$\sF_T\otimes\sB(\D[0,T]\ts \dbR\ts \dbR^{d}\ts \cP_2(\dbR)\ts  \cP_2(\dbR^{d}))$
	measurable  such that $s\mapsto g(t,s,y,z, \mu,\nu)$ is $\dbF$-progressively measurable for
	all $(t,y,z,\mu,\nu)\in  [0,T]   \ts \dbR\ts  \dbR^{ d} \ts \cP_2(\dbR) \ts
	\cP_2(\dbR^{d}) $, and the following conditions hold:
	\begin{enumerate}[~~\,\rm (i)]
		\item   For $(t,s)\in\D[0,T]$,  $y\in \dbR$, $z\in \dbR^d$, $\mu\in\cP_2(\dbR), \nu\in\cP_2(\dbR^d)$,
		$\dbP$-a.s.,
		\begin{equation*}
			\begin{aligned}
				& |g(t,s, y,z,\mu,\nu)|\leq \frac{\gamma}{2}|z|^2+\ell(t,s)+\b|y|
				+\b_0\cW_2(\mu,\d_{0})+ \gamma_0\cW_2(\nu,\d_{0})^{1+\alpha}.
			\end{aligned}
		\end{equation*}
		\item  For $(t,s)\in\D [0,T]$  and for $y,\bar{y}\in \dbR$, $z,\bar{z}\in \dbR^d$, $\mu,\bar{\mu}\in\cP_2(\dbR)$,
		$\nu,\overline{\nu}\in\cP_2(\dbR^d)$, $\dbP$-a.s.,
		$$
		\begin{aligned}
			&|g(t,s,y,z,\mu,\nu)-g(t,s,\bar{y}, \bar{z},\bar{\mu},\bar{\nu})|\\
			&\leq\phi\big(|y|\vee|\bar{y}|\vee \cW_2(\mu,\d_{0})\vee \cW_2(\bar{\mu},\d_{0})\big)\cd
			\big[(1+|z|+|\bar{z}|+\cW_2(\nu,\d_0)+\cW_2(\bar{\nu},\d_0))\\
			&\qq \cd (|z-\bar{z}|
			+|y-\bar{y}|+\cW_2(\mu,\bar{\mu}))
			+(1+\cW_2(\nu,\d_{0})^{\a}+\cW_2(\overline{\nu},\d_{0})^\a)\cW_2(\nu,\overline{\nu})\big].
		\end{aligned}$$
		%
%       \item For $(t,s)\in \D[0,T]$  and  $(y, z, \mu, \n)\in\dbR\times\dbR^d\times \cP_2(\dbR)\ts \cP_2(\dbR^d)$,
%%
%        $$\hb{{\rm sgn}}(y)g(t,s,y,z,\mu,\nu)\leq \ell(t,s)+\beta|y|+\frac{\gamma}{2}|z|^2+\beta_0\cW_2(\mu,\delta_{0})
%             +\gamma_0\cW_2(\n,\delta_{0})^{1+\alpha}.$$
             %
\item   For $(t,s)\in \D[0,T]$  and  $(y, z, \mu, \n)\in\dbR\times\dbR^d\times \cP_2(\dbR)\ts \cP_2(\dbR^d)$, it holds that $\dbP$-a.s.,
        $$
             g(t,s, y,z,\mu,\nu)\leq
            -\frac{\tilde{\gamma}}{2}|z|^2+\ell(t,s)+\beta|y|+\beta_0  \cW_2(\mu,\delta_{0})
             +\gamma_0  \cW_2(\n,\delta_{ 0 })^{1+\alpha}
         $$
or
          \begin{equation*}
           g(t,s, y,z,\mu,\nu)\geq \frac{\tilde{\gamma}}{2}|z|^2-\ell(t,s)
             -\beta|y|-\beta_0   \cW_2(\mu,\delta_{0})
                     -\gamma_0   \cW_2(\n,\delta_{0})^{1+\alpha}.
           \end{equation*}

\item   The free term $\psi(\cd)$ is bounded with $\|\psi(\cd)\|_{L^\i_{\sF_T}[0,T]}\leq K_1$, and
		the process  $\ell(\cd,\cd)$ belongs to the space  $L^\i([a,b];  L^{1,\i}_{\dbF}([\cd,b];\dbR^+))$ with $ \|\ell(\cd,\cd)\|_{L^{\i}_{\dbF}(\D[0,T])} \leq K_3$.
		
	\end{enumerate}
\end{assumption}

\begin{remark}\rm
      In \autoref{ass4.1},   condition (iii) is called a strictly
quadratic growth condition of the generator $g$ with respect to $z$.
Note that \autoref{ass4.1} is weaker than  the condition (A2) of Wang, Sun,Yong \cite{Wang-Sun-Yong-2019} without mean-field terms.
	For example, for
	$(t,s)\in \D[0,T],y\in \dbR, z\in \dbR^d,\mu\in \cP_2(\dbR), \nu\in\cP_2(\dbR^d),$ the following generator
	\begin{equation*}
		\begin{aligned}
			g(t,s,y,z,\mu,\nu)&\deq -|z|^2+\frac{1}{\sqrt{s-t}}+y +\cW_2(\nu,\d_0)^\frac{4}{3}+\cW_2(\mu,\d_0)
		\end{aligned}
	\end{equation*}
	 satisfies neither the condition (A2) in \cite{Wang-Sun-Yong-2019}  nor the aforementioned \autoref{ass4.1-1}, but it
    satisfies \autoref{ass4.1}. In fact, compared with those two conditions, here we  relax the dependence on $ \nu$.
\end{remark}

Next, before proving the existence and uniqueness of global solutions to BSVIE \rf{1.1}, we present the following result concerning the local solution of BSVIE \rf{1.1}. The proof is provided in the Appendix (see Section \ref{appendix}).

\begin{proposition}\label{th 4.2}\rm
	Under  \autoref{ass4.1},
	BSVIE (\ref{1.1})  admits a unique local solution
	$(Y(\cd), Z(\cd,\cd))\in \sB_\e(R_1, R_2)$ with
	\begin{equation*}
		R_1=2\overline{L}\q\hb{and}\q\ R_2=\overline{L} e^{2\overline{L}^2T},
	\end{equation*}
	where $\e$ and $\overline{L}$ are two  positive constants  depending only on
	$K_1,K_3,\tilde{\g},\g_0, \b,\b_0,\a,T$.
\end{proposition}

Now, we  give the existence and uniqueness of global solutions
to mean-field BSVIE (\ref{1.1}).

\begin{theorem}\label{th 4.3}\rm
	Under \autoref{ass4.1}, BSVIE (\ref{1.1})  admits a unique global adapted solution
	$(Y(\cd), Z(\cd,\cd))\in L^\i_{\dbF}[0,T]\ts  \cZ^2_{\dbF}(\D[0,T];\dbR^d).$  Moreover, there exist two positive constants
	$M_1$ and $M_2$  depending on  $K_1,K_3, T,\tilde{\g},\g_0, \b,\b_0,\a$ such that
	\begin{equation*}
		\|Y(\cd)\|_{L_{\dbF}^\i[0,T]}\leq M_1\q\hb{and}\q\ \|Z(\cd,\cd)\|^2_{\cZ^2_{\dbF}(\D[0,T])}\leq M_2.
	\end{equation*}
\end{theorem}

\begin{proof}
	 First, for a small enough constant $\kappa_0>0$, we prove that
	\begin{align*}
		\|Y(\cd)\|_{L_\dbF^\infty[T-\kappa_0,T]}\leq M_1\q\ \hb{and}\q\
		\|Z(\cd,\cd)\|^2_{\cZ^2_\dbF(\D[T-\kappa_0,T])}\leq M_2.
	\end{align*}
	%

	%               For this, chosen

     Let $(\cY(t,\cd),\cZ(t,\cd))$ be the solution to Eq. (\ref{3.211}) with $(Y(\cd), Z(\cd,\cd))$ on the interval $[T-\k_0,T]$.
    We denote
	\begin{equation*}\label{3.20-1}
		\Psi(u,x;t)\deq\exp\Big\{ \gamma x+\gamma\int_0^u\(\ell(t,s)+\beta|Y(s)|+\beta_0\|Y(s)\|_{L^2(\Om)}
		+\gamma_0\|\cZ(t,s)\|^{1+\alpha}_{L^2(\Om)}\)ds\Big\},\ x>0.\end{equation*}
	Then                applying It\^{o}-Tanaka's formula to $\Psi(u,|\cY(t,u)|;t)$, it follows from  \autoref{ass4.1}-(i)  that
	\begin{equation*}
		\begin{aligned}
			\ds d\Psi(u,|\cY(t,u)|;t)&=\gamma\Psi(u,|\cY(t,u)|;t)
			\cd\Big\{-\sgn(\cY(t,u))g(t,u,Y(u),\cZ(t,u),\mathbb{P}_{Y(u)},\mathbb{P}_{\cZ(t,u)} )\\
			& \q +\ell(t,u)+\beta|Y(u)|+\beta_0\|Y(u)\|_{L^2(\Om)} +\gamma_0\|\cZ(t,u)\|^{1+\alpha}_{L^2(\Om)}
			+\frac{1}{2}\gamma|\cZ(t,u)|^2\Big\} du \\
			&\q +\gamma\Psi(u,|\cY(t,u)|;t)\sgn(\cY(t,u))\cZ(t,u)dW(u)+\gamma\Psi(u,|\cY(t,u)|;t)dL(u)\\
			\ns\ds
			&\geq \gamma\Psi(u,|\cY(t,u)|;t)\sgn(\cY(t,u))\cZ(t,u)dW(u),
		\end{aligned}
	\end{equation*}
	where  the term $L(\cd)$ is the local time of the process $\cY(t,\cd)$ at time $t$.
By	integrating from $r$ to $T$ firstly and then   taking the conditional expectation $\mathbb{E}_r[\cdot]$ on
	both sides of the above inequality, one has
	\begin{equation*}
		\begin{aligned}
			\exp\big(  \gamma|\cY(t,r)|  \big)
			&\leq \dbE_r\exp\Big\{ \gamma|\psi(t)|+\gamma\int_r^T
			\(\ell(t,u)+\beta|Y(u)|+\beta_0\|Y(u)\|_{L^2(\Om)}+\gamma_0\|\cZ(t,u)\|^{1+\alpha}_{L^2(\Om)}\)du\Big\}\\
			\ns\ds &\les \exp\Big\{\gamma(K_1+K_3)+\g\int_r^T(\beta+\beta_0)\|Y(\cd)\|_{L_\dbF^\infty[u,T]}du
			+\g\gamma_0\int_r^T\|\cZ(t,u)\|^{1+\alpha}_{L^2(\Om)}du\Big\},
		\end{aligned}
	\end{equation*}
	which implies that for $t\leq r\leq T$,
	\begin{equation*}
		|\cY(t,r)|\leq K_1+K_3+\int_r^T(\beta+\beta_0)\|Y(\cd)\|_{L_\dbF^\infty[u,T]}du
		+\gamma_0\int_r^T\|\cZ(t,u)\|^{1+\alpha}_{L^2(\Om)}du.
	\end{equation*}
Let $\e_0$ be a positive constant which will be specified later. Now, we analyze   the last term of the above inequality.   Making use of  the relation (\ref{4.9-1})  in Appendix,
	it follows
	\begin{equation}\label{4.34-1}
		\begin{aligned}
			|\cY(t,r)|\leq K_1+K_3 +\int_r^T(\beta+\beta_0) \|Y(\cd)\|_{L_\dbF^\infty[u,T]}du  +  \int_r^T\frac{\widetilde{\g}\e_0}{4}\dbE|\cZ(t,u)|^2du+L_{2,\g_0}T,
		\end{aligned}
	\end{equation}
	where $L_{2,\g_0}$ is defined in (\ref{4.1-1}) with $\delta$ replaced by $\gamma_0$.
	In addition,  similar to the relation (\ref{4.9}) in Appendix, one has
	\begin{equation}\label{4.34-21}
		\begin{aligned}
			\dbE_r\int_r^T\frac{\widetilde{\g}\e_0}{4}|\cZ(t,u)|^2du
			\leq \e_0   |\cY(t,r)|  +\e_0(K_1+K_3)+L_1T
			+ \e_0(\beta+\beta_0)\int_r^{T}\|Y(\cd)\|_{L^\i_{\dbF}[u,T]} du.
		\end{aligned}
	\end{equation}
	Now, by taking the expectation on both sides of \rf{4.34-21} firstly and then inserting  it into (\ref{4.34-1}), we have
	\begin{equation*}
		\begin{aligned}
			|\cY(t,r)|& \leq (K_1+K_3)(\e_0+1)+L_1T+L_{2,\g_0}T+  \e_0  \dbE |\cY(t,r)| \\
			& \q+\int_r^T(\beta+\beta_0)(1+\e_0) \|Y(\cd)\|_{L_\dbF^\infty[u,T]}du .
		\end{aligned}
	\end{equation*}
	%                Here $L_{1}$ is introduced in (\ref{4.1-1}).
	%               %
By	taking the expectation on both sides of the above inequality  and letting  $\e_0=\frac{1}{2} $ leads to
	\begin{equation*}
		\begin{aligned}
			\frac{1}{2}\dbE|\cY(t,r)|&\leq \frac{3}{2}(K_1+K_3) +L_1T+L_{2,\g_0}T
			+\frac{3}{2}(\beta+\beta_0)\int_r^T\|Y(\cd)\|_{L_\dbF^\infty[u,T]}du.
		\end{aligned}
	\end{equation*}
	Combining the above two inequalities provides
	\begin{equation}\label{4.34-2}
		\begin{aligned}
			|\cY(t,r)|&\leq 3(K_1+K_3) +2L_1T+2L_{2,\g_0}T
			+3(\beta+\beta_0)\int_r^T\|Y(\cd)\|_{L_\dbF^\infty[u,T]}du.
		\end{aligned}
	\end{equation}
	In particular, when $r=t$, we deduce by \autoref{pro3.1} that
	\begin{equation}\label{4.34-1-1}
		\begin{aligned}
			|Y(t)| & \leq3(K_1+K_3) +2L_1T+2L_{2,\g_0}T
			+3(\beta+\beta_0)\int_t^T\|Y(\cd)\|_{L_\dbF^\infty[u,T]}du.
		\end{aligned}
	\end{equation}
	In order to further obtain the estimate of $Y(\cd)$, we  consider the following ordinary differential equation (ODE)
	\begin{equation*}
		\a(t) =L_7+\int_t^T L_7 \a(u)du,\q t\in[0,T],
	\end{equation*}
	whose solution  is expressly given by $	\a(t)=L_7e^{L_7(T-t)},\  t\in[0,T].$  Note that $\a(\cd)$ is a continuous, non-increasing function, and
	$ \|\psi(\cd)\|_{L^\i_{\cF_T}[0,T]}\leq L_7=\a(T)\leq \a(0). $
	Then, from (\ref{4.34-1-1}), we have
	\begin{equation}\label{4.35-1}
		\|Y(\cd)\|_{L^\i_{\dbF}[t,T]}\leq \a(t)\les \a(0),\  \forall t\in[T-\k_0,T].
	\end{equation}
  Inserting (\ref{4.35-1}) into (\ref{4.34-2}), we deduce
  \begin{align*}
  |\cY(t,r)|&\leq 3(K_1+K_3) +2L_1T+2L_{2,\g_0}T
			+3(\beta+\beta_0)T\a(0)\deq\widetilde{M}.
  \end{align*}
Set
 $$ M_1\deq \max\{ \a(0), \widetilde{M}\}.$$
Thence, we have that for any $(t,r)\in\D[T-\k_0,T]$,
\begin{equation}\label{4.35}
\|Y(\cd)\|_{L^\i_{\dbF}[t,T]}\leq M_1\q\hbox{and}\q   |\cY(t,r)| \leq M_1.
\end{equation}
On the other hand,
	by letting $\e_0=\frac{1}{2}$, then  (\ref{4.34-21}) becomes
	\begin{equation}\label{4.911}
       \begin{aligned}
		\dbE_r\int_r^T\frac{\widetilde{\g}}{8}|\cZ(t,u)|^2du&\leq \frac{1}{2}|\cY(t,r)|  +\frac{K_1+K_3}{2}+L_1T
		+ \frac{\beta+\beta_0}{2}\int_r^{T}\|Y(\cd)\|_{L^\i_{\dbF}[u,T]} du\\
        &\leq  \frac{1}{2}\(M_1+K_1+K_3+2L_1T+(\b+\b_0)TM_1\).
         \end{aligned}
	\end{equation}
    Note that (\ref{4.911}) still holds  if we replace $\g$ by stopping time $\t\in\sT[t,T]$.
   Finally,	by combining (\ref{4.35}), (\ref{4.911}) and recalling
   the definition of $\|\cd\|_{ \cZ^2_{\dbF}(\D[T-{\k_0},T])}$,  we obtain
	\begin{equation*}
		\|Z(\cd,\cd)\|^2_{ \cZ^2_\dbF(\D[T-\kappa_0,T])}\leq
\frac{4}{\widetilde{\g}}\(M_1+K_1+K_3+2L_1T+(\b+\b_0)TM_1\)\deq M_2.
	\end{equation*}

	Taking $T-\k_0$ as the terminal time and $\psi^{T-\k_\varrho}(t)=\cY (t, T-\k_0)$ as the terminal value,
making a similar analysis  as the proof of \autoref{th3.4},
 it is easy to prove  the well-posedness and boundedness of global solutions to BSVIE (\ref{1.1}) on the interval $[0,T]$. This completes the proof.
\end{proof}

%%%%%%%%%%%%%%%%%%%%%%%%%%%%%%%%%%%%%%%%%%%%%%%%%%%%%%%%%%%%%%%%%%%%%%%%%%%%%%%%%%%%%%%%%%%%%%%%%%%%%%%%%%%%%%%%%%%%%%%%%%%%%%%%%%%%%%%%%%%%
%%%%%%%%%%%%%%%%%%%%%%%%%%%%%%%%%%%%%%%%%%%%%%%%%%%%%%%%%%%%%%%%%%%%%%%%%%%%%%%%%%%%%%%%%%%%%%%%%%%%%%%%%%%%%%%%%%%%%%%%%%%%%%%%%%%%%%%%%%%%
%%%%%%%%%%%%%%%%%%%%%%%%%%%%%%%%%%%%%%%%%%%%%%%%%%%%%%%%%%%%%%%%%%%%%%%%%%%%%%%%%%%%%%%%%%%%%%%%%%%%%%%%%%%%%%%%%%%%%%%%%%%%%%%%%%%%%%%%%%%%
\section{Particle Systems}\label{sec 5}

This section focuses on the convergence and convergence rate of the particle system in Eq. \eqref{1.3} associated with Eq. \eqref{1.1}.
Let ${\psi^i(\cdot); 1 \leq i \leq N}$ be $N$ independent copies of the free term $\psi(\cdot)$, and let $\dbF^N = {\sF^N_t}, {t \geq 0}$ denote the natural filtration of $W^i, 1 \leq i \leq N$, augmented by all $\dbP$-null sets. Here, $W^i, 1 \leq i \leq N$ are $N$ independent $d$-dimensional Brownian motions, as introduced earlier.

In the following, we use the function: $\d_{ij}=1$, if $i=j$; or else, it equals $0$.
Moreover,  we always assume that $s\mapsto g(t,s,y,z, \mu,\nu)$ is  $\dbF^N$-progressively  measurable  for
          all $(t,y,z,\mu,\nu)\in[0,T]\ts \dbR\ts  \dbR^{ d} \ts \cP_2(\dbR) \ts
          \cP_2(\dbR^{d}) $.
 Additionally, let $(Y^{N,i}(\cd),Z^{N,i,j}(\cd,\cd))$  be the adapted solutions to the
         particle system (\ref{1.3}), and let  $(\cl{Y}^{i}(\cd),\cl{Z}^i(\cd,\cd))$ be the solution to
          the following mean-field  BSVIE:
          \begin{equation}\label{5.2}
          \begin{aligned}
            \cl{Y}^i(t)=\psi^i(t)+\int_t^Tg(t,s,\cl{Y}^i(s),\cl{Z}^{i}(t,s), \cl{\mu}(s),\cl{\nu}(t,s))ds-\int_t^T\cl{Z}^{i}(t,s)dW^i(s),
           \end{aligned}
           \end{equation}
           where
           \begin{equation}\label{5.3-11}
           \begin{aligned}
               \cl{\mu}(s)\deq \mathbb{P}_{\cl{Y}^i(s)}\q \hbox{and}\q \cl{\nu}(t,s)\deq  \mathbb{P}_{\cl{Z}^i(t,s)}.
              \end{aligned}
              \end{equation}

Next, following the approach analogous to that in \autoref{sec 3},   we   study the convergence and convergence rate of the particle system
\rf{1.3} by dividing the analysis into two cases: the linear growth case and the quadratic growth case.

\subsection{Linear growth case}

In this subsection, we study the  convergence rate of the particle system (\ref{1.3}) in the case where $g$ is of linear growth in $(y,z,\mu,\nu)$.
First, we give the  convergence of the particle system (\ref{1.3}).
\begin{theorem}\label{th4.1} \rm
Let \autoref{ass3.1} be in force and $g$ is Lipschitz continuous with respect to  $(\mu,\nu)$ in 1-Wasserstein distance.
Let $(Y^{N,i}(\cd),Z^{N,i,j}(\cd,\cd))_{1\leq j\leq N}$, $(\cl{Y}^{i}(\cd),\cl{Z}^i(\cd,\cd))$ be the solutions of   Eq. (\ref{1.3}) and  Eq. (\ref{5.2}), respectively. Then,  for  $1<p<2$,
\begin{equation*}
		\begin{aligned}
			&\ \mathbb{E}\bigg[\int_0^T|  Y^{N,i}(t)- \cl{Y}^i(t)|^pdt
			+ \int_0^T\(\int_t^T \sum\limits_{j=1}^N|Z^{N,i,j}(t,s)- \d_{ij}\cl{Z}^i(t,s)|^2ds\)^\frac{p}{2} dt\bigg] \nn\\
			   &\leq C  \dbE\int_0^T  \cW_p^p(\mu^N(t), \overline{\mu}(t))dt
              +C  \dbE\int_0^T\int_t^T\cW_p^p(\nu^N(t,r), \overline{\nu}(t,r))drdt.
		\end{aligned}
	\end{equation*}
\end{theorem}
\begin{proof}
Thanks to Lin \cite[Theorem 4.1]{Lin-2002}, the  multi-dimensional BSVIE (\ref{1.3}) has a unique solution
$(\mathbf{Y}^N,\mathbf{Z}^N)=(Y^{N,i},Z^{N,i,j} )_{i,j=1,\cds,N}\in L^2_{\dbF}([0,T];\dbR^N)\ts
L^2_{\dbF}([0,T];(\dbR^{N\ts d})^N).$ By $(\mathbb{Y}^{N,i}(\cd,\cd), \mathbb{Z}^{N,i,i}(\cd,\cd))$,
 we denote the unique solution to the following BSDE (parameterized by $t$)
\begin{equation}\label{4.3-111}
\begin{aligned}
\mathbb{Y}^{N,i}(t,s)&=\psi^i(t)+\int_s^Tg(t,r,  Y^{N,i}(r), \mathbb{Z}^{N,i,i}(t,r), \frac{1}{N}\sum\limits_{i=1}^N \d_{Y^{N,i}}(r),
\frac{1}{N}\sum\limits_{i=1}^N \d_{\mathbb{Z}^{N,i,i}}(t,r) )dr\\
&\q -\int_s^T\sum\limits_{j=1}^N  \mathbb{Z}^{N,i,j}(t,r) dW^j(r),\q  s\in[t,T].
\end{aligned}
\end{equation}
It follows from the uniqueness of BSVIEs that
\begin{equation}\label{4.4111}
Y^{N,i}(t)=\mathbb{Y}^{N,i}(t,t)\q\hbox{and}\q   Z^{N,i,j}(t,s)=\mathbb{Z}^{N,i,j}(t,s),\q  (t,s)\in\D[0,T].
\end{equation}
Additionally,   \autoref{pro3.1}  provides that
\begin{align*}
\overline{Y}^{i}(t)=\overline{\mathbb{Y}}^{i}(t,t)\q\hbox{and}\q   \overline{Z}^{i}(t,s)=\overline{\mathbb{Z}}^{i}(t,s),\q  (t,s)\in\D[0,T],
\end{align*}
where  for almost all $t\in[0,T]$, the pair
$(\overline{\mathbb{Y}}^{i}(t,\cd),\overline{\mathbb{Z}}^{i}(t,\cd))$ solves the following equation
\begin{equation*}
\begin{aligned}
\overline{\mathbb{Y}}^{i}(t,s)=\psi^i(t)+\int_s^Tg(t,r,  \overline{Y}^{i}(r), \overline{\mathbb{Z}}^{i}(t,r),
\cl{\mu}(r), \mathbb{P}_{\cl{\mathbb{Z}}^i(t,r)} )dr
 -\int_s^T   \overline{\mathbb{Z}}^{i}(t,r) dW^i(r),\q  s\in[t,T].
\end{aligned}
\end{equation*}
Next, for simplicity of presentation, we denote by  $\D(\cd)$  the corresponding differences of  solutions.
Then, according to  Briand et al. \cite{Briand-Delyon-Hu-Pardoux-Stoica-2003}, it follows that for $1<p<2$,
\begin{equation*}
\begin{aligned}
&\ \dbE\[ |\D \mathbb{Y}^{N,i}(t,s)|^p+\(\int_s^T\sum\limits_{j=1}^N |\D  \mathbb{Z}^{N,i,j}(t,r)|^2dr\)^\frac{p}{2}\]\\
& \leq C\dbE  \( \int_s^T (g(t,r,  Y^{N,i}(r), \overline{\mathbb{Z}}^{i}(t,r), \frac{1}{N}\sum\limits_{i=1}^N \d_{Y^{N,i}}(r),
\frac{1}{N}\sum\limits_{i=1}^N \d_{\mathbb{Z}^{N,i,i}}(t,r) )\\
&\qq \qq\q  -g(t,r,  \overline{Y}^{i}(r), \overline{\mathbb{Z}}^{i}(t,r),
\cl{\mu}(r), \mathbb{P}_{\cl{\mathbb{Z}}^i(t,r)} ))dr    \)^p\\
& \leq C \(\dbE\int^T_s|\D Y^{N,i}(r)|^pdr+ \dbE\int_s^T \cW_1^p(\mu^N(r), \overline{\mu}(r))dr\)\\
&\q +C_1(T-s)^{p-1} \dbE\int_s^T\cW_1^p(\nu^N(t,r), \overline{\nu}(t,r))dr,
\end{aligned}
\end{equation*}
where  $C_1$ depends on $L,p$ and is independent of $T$.
Thanks to  the definition of the $p$-Wasserstein metric, we have that  for any $r>1$ and $\vartheta_1,\vartheta_2\in \cP_1(\dbR^d),$
\begin{equation*}
\cW_1^r(\vartheta_1,\vartheta_2)\leq \cW_r^r(\vartheta_1,\vartheta_2).
\end{equation*}
Taking $s=t$, it follows from Gronwall inequality  that
\begin{align}\label{4.77}
&\dbE |\D Y^{N,i}(t)|^p\leq C  \dbE\int_t^T \cW_p^p(\mu^N(r), \overline{\mu}(r))dr
 +C_1(T-t)^{p-1} \dbE\int_t^T\cW_p^p(\nu^N(t,r), \overline{\nu}(t,r))dr,
\end{align}
and
\begin{align}\label{4.78}
&\ \dbE\[  \(\int_t^T\sum\limits_{j=1}^N |\D  Z^{N,i,j}(t,r)|^2dr\)^\frac{p}{2}\]\nonumber\\
&\leq   C  \dbE\int_t^T \cW_p^p(\mu^N(r), \overline{\mu}(r))dr
+C \dbE\int_t^T\int_s^T \cW_p^p(\nu^N(s,r), \overline{\nu}(s,r))drds \\
&\q +C_1(T-t)^{p-1} \dbE\int_t^T\cW_p^p(\nu^N(t,r), \overline{\nu}(t,r))dr.\nonumber
\end{align}
 Integrating from $0$ to $T$ on both sides of (\ref{4.77}) and (\ref{4.78}), we get the desired result.
\end{proof}

\begin{theorem}\label{th4.2}\rm
	Let \autoref{ass3.1} be in force and $g$ is Lipschitz continuous with respect to  $(\mu,\nu)$ in 1-Wasserstein distance.
	 Let $(Y^{N,i}(\cd), Z^{N,i,j}(\cd,\cd))_{1\leq i,j\leq N}$, $(\cl{Y}^{i}(\cd),\cl{Z}^i(\cd,\cd))$ be the  unique solutions of   Eq. (\ref{1.3}) and  Eq. (\ref{5.2}), respectively.
 Then, for any $1<p<2$, there exists a positive constant  $C$  depending only on $L$ and $p$, such that
               \begin{equation*}
                   \begin{aligned}
                     \mathbb{E}\Big[\int_0^T| Y^{N,i}(t)- \cl{Y}^i(t)|^pdt
                       + \int_0^T\(\int_t^T \sum\limits_{j=1}^N|Z^{N,i,j}(t,s)- \d_{ij}\cl{Z}^i(t,s)|^2ds\)^\frac{p}{2} dt \Big]
                           \leq C \cQ(N),
                 \end{aligned}
                \end{equation*}
                where
                \begin{equation}\label{1.1111}
                \cQ(N) =
                \begin{cases}
                 N^{-(2-p)/2}            & \text{as}\q d=1,2,3,\\
                 N^{-\min\{p/d,(2-p)/2\}}            & \text{as}\q d\geq 4,
                \end{cases}
                \end{equation}
                and $d$ is  the dimension of Brownian motion.

\end{theorem}

\begin{proof}
              For  $i=1,2,\cd\cd\cd, N$,
              let $(\widetilde{Y}^{N,i}(\cd),\widetilde{Z}^{N,i}(\cd,\cd))$  be i.i.d. copies of  $(\overline{Y}^i(\cd),\overline{Z}^i(\cd,\cd))$
                    such that
                   \begin{equation}\label{4.15111}
                   \widetilde{Y}^{N,i}(t)=\psi^i(t)+\int_t^Tg(t, s,\widetilde{Y}^{N,i}(s),\widetilde{Z}^{N,i}(t,s), \bar{\mu}(s),\bar{\nu}(t,s))ds-\int_t^T\widetilde{Z}^{N,i}(t,s)dW^i(s),\q  t\in[0,T],
                   \end{equation}
                   where $\overline{\mu}(\cd) $ is given in (\ref{5.3-11}).
        Then
        \begin{align*}
        Y^{N,i}(t)-\widetilde{Y}^{N,i}(t)&=\int_t^T\( g(t,s,Y^{N,i}(s),Z^{N,i,i}(t,s),  \mu^N (s), \nu^N(t,s) )\\
                                   &\qq\qq - g(t,s,\widetilde{Y}^{N,i}(s),\widetilde{Z}^{N,i}(t,s), \cl{\mu}(s),\cl{\nu}(t,s) )\)ds\\
                                     &\q  -\int_t^T\sum\limits_{j=1}^N(Z^{N,i,j}(t,s)-\d_{ij}\widetilde{Z}^{N,i}(t,s))dW^j(s).
        \end{align*}
        Similar to (\ref{4.77}) and (\ref{4.78}), we have that for $1<p<2$ and $t\in[0,T]$,
        \begin{align*}
        &\ \mathbb{E}\bigg[ |  Y^{N,i}(t)- \widetilde{Y}^{N,i}(t)|^p
			+ \(\int_t^T \sum\limits_{j=1}^N|Z^{N,i,j}(t,s)- \d_{ij}\widetilde{Z}^{N,i}(t,s)|^2ds\)^\frac{p}{2}  \bigg] \nn\\
			&   \leq C  \dbE\int_t^T \cW_p^p(\mu^N(r), \overline{\mu}(r))dr
                +C \dbE\int_t^T\int_s^T \cW_p^p(\nu^N(s,r), \overline{\nu}(s,r))drds \\
            &   \q +C_1(T-t)^{p-1}  \dbE\int_t^T\cW_p^p(\nu^N(t,s), \overline{\nu}(t,s))ds,
        \end{align*}
         where $C_1$ is a constant depending only on $L$ and $p$, independent of $T$. Consequently,
        \begin{align*}
        &\ \dbE\cW^p_p(\mu^N(t), \widetilde {\mu}^N(t)) \leq \dbE\[\frac{1}{N}\sum\limits_{i=1}^N\Big|Y^{N,i}(t) -\widetilde{Y}^{N,i}(t) \Big|^p   \]\\
         &  \leq C  \dbE\int_t^T \cW_p^p(\mu^N(r), \overline{\mu}(r))dr
                +C \dbE\int_t^T\int_s^T \cW_p^p(\nu^N(s,r), \overline{\nu}(s,r))drds \\
         &  \q +C_1(T-t)^{p-1}   \dbE\int_t^T\cW_p^p(\nu^N(t,s), \overline{\nu}(t,s))ds,
        \end{align*}
        and
        \begin{align*}
        &\ \dbE\int_t^T\cW^p_p(\nu^N(t,s), \widetilde {\nu}^N(t,s))ds
        \leq \dbE\int_t^T\frac{1}{N}\sum\limits_{i=1}^N\Big|Z^{N,i,i}(t,s)-\widetilde{Z}^{N,i}(t,s)\Big|^p ds  \]\\
        &\leq \frac{1}{N}\sum\limits_{i=1}^N\dbE\(\int_t^T\Big|Z^{N,i,i}(t,s)-\widetilde{Z}^{N,i}(t,s)\Big|^2 ds\)^\frac{p}{2} \\
         &  \leq C  \dbE\int_t^T \cW_p^p(\mu^N(r), \overline{\mu}(r))dr
                +C \dbE\int_t^T\int_s^T \cW_p^p(\nu^N(s,r), \overline{\nu}(s,r))drds \\
         &  \q +C_1(T-t)^{p-1}   \dbE\int_t^T\cW_p^p(\nu^N(t,s), \overline{\nu}(t,s))ds.
        \end{align*}
        Here
        \begin{align}\label{6.2}
        \widetilde {\mu}^N(t)\deq\frac{1}{N}\sum\limits_{i=1}^N\delta_{\widetilde{Y}^{N,i}(t)}\q\hbox{and}
        \q \widetilde {\nu}^N(t,s)\deq\frac{1}{N}\sum\limits_{i=1}^N\delta_{\widetilde{Z}^{N,i}(t,s)}.
        \end{align}
        Then, it follows from the triangle inequality that
        \begin{align*}
           &\ \dbE\cW^p_p(\mu^N(t), \overline{\mu} (t))+ \dbE\int_t^T\cW^p_p(\nu^N(t,s), \overline {\nu} (t,s))ds \\
           &\leq   \dbE\cW^p_p(\mu^N(t), \widetilde {\mu}^N(t)) +\dbE\cW^p_p(\widetilde{\mu}^N(t), \overline{\mu}(t))\\
           &\q+\dbE\int_t^T\cW^p_p(\nu^N(t,s), \widetilde {\nu}^N(t,s))ds+ \dbE\int_t^T\cW^p_p(\widetilde{\nu}^N(t,s),  \overline{\nu}(t,s))ds   \\
            &\leq \dbE\cW^p_p(\widetilde{\mu}^N(t), \overline{\mu}(t))+ \dbE\int_t^T\cW^p_p(\widetilde{\nu}^N(t,s),  \overline{\nu}(t,s))ds\\
            & \q +C  \dbE\int_t^T \cW_p^p(\mu^N(r), \overline{\mu}(r))dr
                +C \dbE\int_t^T\int_s^T \cW_p^p(\nu^N(s,r), \overline{\nu}(s,r))drds \\
         &  \q +C_1(T-t)^{p-1}   \dbE\int_t^T\cW_p^p(\nu^N(t,s), \overline{\nu}(t,s))ds.
        \end{align*}
        Taking a proper $\e^*$ such that $C_1(\e^*)^{p-1}=\frac{1}{2}$. Then, for any $t\in[T-\e^*,T]$, it is easy to see that
        $C_1(T-t)^{p-1}\leq \frac{1}{2}.$  Thereby,  Gronwall's inequality provides
        \begin{align*}
         &\ \dbE\cW^p_p(\mu^N(t), \overline{\mu} (t))+ \dbE\int_t^T\cW^p_p(\nu^N(t,s), \overline {\nu} (t,s))ds \\
          &\leq C \bigg\{\dbE\cW^p_p(\widetilde{\mu}^N(t), \overline{\mu}(t))+ \dbE\int_t^T\cW^p_p(\widetilde{\nu}^N(t,s),  \overline{\nu}(t,s))ds\bigg\}.
        \end{align*}
       Hence,
       \begin{align*}
         &\dbE\int_{T-\e^*}^T\cW^p_p(\mu^N(t), \overline{\mu} (t))dt+ \dbE\int_{T-\e^*}^T\int_t^T\cW^p_p(\nu^N(t,s), \overline {\nu} (t,s))dsdt \\
          &\leq C \(\dbE\int_{T-\e^*}^T\cW^p_p(\widetilde{\mu}^N(t), \overline{\mu}(t))dt+ \dbE\int_{T-\e^*}^T\int_t^T\cW^p_p(\widetilde{\nu}^N(t,s),  \overline{\nu}(t,s))dsdt\).
        \end{align*}
        According to Fournier and Guillin \cite[Theorem 1]{Fournier-Guillin-2015}, we have
        \begin{align}\label{4.009}
        &\dbE\int_{T-\e^*}^T\cW^p_p(\widetilde{\mu}^N(t), \overline{\mu}(t))dt
         \leq\cQ(N)\int_{T-\e^*}^T \{\dbE|\overline{Y}^{i}(t)|^2\}^\frac{p}{2}dt
        \leq\cQ(N)\int_{T-\e^*}^T (1+\dbE|\overline{Y}^{i}(t)|^2)dt<\i,
        \end{align}
         and
         \begin{align}
         &\ \dbE\int_{T-\e^*}^T\int_t^T\cW^p_p(\nu^N(t,s), \overline {\nu} (t,s))dsdt \nonumber\\
         &\leq\cQ(N)\int_{T-\e^*}^T\int_t^T\{\dbE|\overline{Z}^{i}(t,s)|^2\}^\frac{p}{2}dsdt
        \leq\cQ(N)\int_{T-\e^*}^T\int_t^T(1+\dbE|\overline{Z}^{i}(t,s)|^2) dsdt<\i. \label{4.100}
        \end{align}
       Hence,
       \begin{align*}
         &\dbE\int_{T-\e^*}^T\cW^p_p(\mu^N(t), \overline{\mu} (t))+ \dbE\int_{T-\e^*}^T\int_t^T\cW^p_p(\nu^N(t,s), \overline {\nu} (t,s))dsdt
         \leq C \cQ(N).
        \end{align*}

        Next, let us consider the time terminal $[T-2\e^*, T-\e^*]$.  On the one hand, making a similar analyses, one gets
        \begin{align*}
         &\dbE\int_{T-2\e^*}^{T-\e^*}\cW^p_p(\mu^N(t), \overline{\mu} (t))dt \leq C \cQ(N).
        \end{align*}
       On the other hand, we have, for $t\in[T- 2\e^*, T- \e^*]$,
        \begin{align*}
        &\ \dbE\int_{T-2\e^*}^{T-\e^*}\int_{t}^T\cW^p_p(\nu^N(t,s), \overline {\nu} (t,s))dsdt
          \leq\cQ(N)\int_{T-2\e^*}^{T-\e^*}\int_t^T\{\dbE|\overline{Z}^{i}(t,s)|^2\}^\frac{p}{2}dsdt\\
        & \leq\cQ(N)  \dbE\int_{T-2\e^*}^{T}\int_t^T(1+|\overline{Z}^{i}(t,s)|^2) dsdt  \leq C \cQ(N).
        \end{align*}
       By repeating the above process a finite number of times, we finally obtain
  \begin{equation*}
                   \begin{aligned}
                     \mathbb{E}\Big[\int_0^T| Y^{N,i}(t)- \cl{Y}^i(t)|^pdt
                       + \int_0^T\(\int_t^T \sum\limits_{j=1}^N|Z^{N,i,j}(t,s)- \d_{ij}\cl{Z}^i(t,s)|^2ds\)^\frac{p}{2} dt \Big]
                           \leq C \cQ(N).
                 \end{aligned}
                \end{equation*}
                This completes the proof.
\end{proof}

%{\color{blue}
\begin{remark}\label{re4.33} \rm 
   From the proof of \autoref{th4.2}, we know that when $g$ depends on the law of $Z(\cd,\cd)$, we should work within the space
$L^p_{\dbF}(\D [0,T];\dbR^d)$ with $1<p<2$, because we only have that  $\int_{0}^T\int_t^T\dbE|\overline{Z}^{i}(t,s)|^2dsdt<\i$  (see (\ref{4.100})).
\end{remark}

 If the generator $g$ is independent of the law of $Z(\cd,\cd)$, we have a better convergence rate. More precisely, it is not necessary to restrict $p$ to   the interval  $(1, 2)$. Indeed, for any $p > 1$, we have the following result concerning the convergence rate.

\begin{proposition}\label{pro4.44} \rm
Let \autoref{ass3.1} hold. Assume that $g$ is independent of the law of $Z(\cd,\cd)$ and $g$ is Lipschitz continuous with respect to $\mu$ in 1-Wasserstein distance. Furthermore, suppose that for $p>1$, there exists a $q>p$ such that
 the free term $\psi(\cdot) \in L^q_{\mathcal{F}_T}([0,T])$ and  $  g(t,s,0,0,\delta_0,\delta_0)\in L^q_{\mathbb{F}}(\Delta [0,T])$.
Then,  there exists a positive constant  $C$ depending on the Lipschitz constant of $g$ and $p$ such that
               \begin{equation*}
                   \begin{aligned}
                     \mathbb{E}\Big[\int_0^T| Y^{N,i}(t)- \cl{Y}^i(t)|^pdt
                       + \int_0^T\(\int_t^T \sum\limits_{j=1}^N|Z^{N,i,j}(t,s)- \d_{ij}\cl{Z}^i(t,s)|^2ds\)^\frac{p}{2} dt \Big]
                           \leq C \cQ(N),
                 \end{aligned}
                \end{equation*}
        where  $(Y^{N,i}(\cd), Z^{N,i,j}(\cd,\cd))_{1\leq i,j\leq N}$, $(\cl{Y}^{i}(\cd),\cl{Z}^i(\cd,\cd))$ is the  solutions of   (\ref{1.3}) and  (\ref{5.2}), respectively.
\end{proposition}

\begin{proof}
First, since $g$ is independent of the law of $Z(\cd,\cd)$, we do not need to estimate the term
$\dbE\int_0^T\int_s^T \cW_p^p(\nu^N(s,r), \overline{\nu}(s,r))drds $. Second,
the assumption that  $\psi(\cdot) \in L^q_{\mathcal{F}_T}([0,T])$ and  $g(t,s,0,0,\delta_0,\delta_0)\in L^q_{\mathbb{F}}(\Delta [0,T])$
allows to show $\sup\limits_{t\in[0,T]}\dbE|\overline{Y}^{i}(t)|^q<\i$.  Consequently, making a calculation similar to (\ref{4.009}) yields the desired result.
\end{proof}
%}

%%%%%%%%%%%%%%%%%%%%%%%%%%%%%%%%%%%%%%%%%%%%%%%%%%%%%%%%%%%%%%%%%%%%%%%%%%%%%%%%%%%%%%%%%%%%%%%%%%%%%%%%%%%%%%%%%%%%%%%%%%%%%%%%%%%%%%%%%%%%
\subsection{Quadratic growth case}

{This subsection is dedicated to studying the convergence and its rate for the particle system (\ref{1.3}) in the case where $g$ is independent of the law of $Z(\cd,\cd)$ and $g$ exhibits quadratic growth in $z$.}

%%%%%%%%%%%%%%%%%%%%%%%%%%%%%%%%%%%%%%%%%%%%%%%%%%%%%%%%%%%%%%%%%%%%%%%%%%%%%%%%%%%%%%%%%%%%%%%%%%%%%%%%%%%%%%%%%%%%%%%%%%%%%%%%%%%%%%%%%%%%%
%\subsubsection{$g$ is bounded in $\dbP_{Z(t,s)}$}

\begin{proposition}\label{pro4.4}\rm
Let \autoref{ass4.1-1} hold. Assume  $g$ is independent of the law of $Z(\cd,\cd)$ and  $\max\limits_{1\leq i\leq N}\|\psi^i(\cd)\|_{L^\i_{\sF_T}[0,T]}\leq K_1$. Then there exist two positive constants $\overline{C}$ and $\overline{\overline{C}}$, depending only on $(\b,\b_0, K_1, K_2, T, \gamma)$, such that the unique adapted solution $(Y^{N,i}(\cdot), Z^{N,i,j}(\cdot, \cdot))_{1\leq i,j\leq N}\in L^\i_{\dbF}([0,T];\dbR^N)\ts  \cZ^2_{\dbF}(\D[0,T];(\dbR^{N\ts d})^N)$
 for the particle system described in Eq. \rf{1.3} satisfies the following estimate: for $i,j=1,...,N$,
              \begin{equation*}
                 ||Y^{N,i}(\cd)||_{L_\dbF^\infty[0,T]}\leq \overline{C}\q\hbox{and}\q
                 ||Z^{N,i,j}(\cd,\cd)||_{ \cZ^2_\dbF(\D[0,T])}\leq \overline{\overline{C}}.
              \end{equation*}
\end{proposition}

We postpone its proof to the appendix for now.

%%%%%%%%%%%%%%%%%%%%%%%%%%%%%%%%%%%%%%%%%%%%%%%%%%%%%%%%%%%%%%%%%%%%%%%%%%%%%%%%%%%%%%%%%%%%%%%%%%%%%%%%%%%%%%%%%%%%%%%%%%%%%%%%%%%%%%%%%%%%

\begin{theorem}\label{th 5.1111}\rm
Let \autoref{ass4.1-1} hold  and let $g$ be independent of  the law of $Z(\cd,\cd)$. Assume $\max\limits_{1\leq i\leq N}\|\psi^i(\cd)\|_{L^\i_{\sF_T}[0,T]}\leq K_1$. Then for any  $p\geq2$, there exist two  constants $l^*_0$ and $l^*_1$ with $\min\{l^*_0,l^*_1\}>1$, and a positive constant $C$, depending only on $(\b,\b_0,K_1,K_2,T,\phi(\cd),\g, l^*_0, l^*_1,p)$, such that for almost all $t\in[0,T]$ and for  $i=1,\cds,N,$
	\begin{equation}
		\begin{aligned}
			&\ \mathbb{E}\bigg[| Y^{N,i}(t)- \cl{Y}^i(t)|^p
			+\(\int_t^T \sum\limits_{j=1}^N|Z^{N,i,j}(t,s)- \d_{ij}\cl{Z}^i(t,s)|^2ds\)^\frac{p}{2} \bigg] \nn\\
			&\leq   C\bigg\{\dbE\[
 \int_t^T  \cW_2^{p(l_0^*l_1^*)^2}(\mu^N(s),\overline{\mu}(s))ds\]\bigg\}^{\frac{1}{(l_0^*l_1^*)^2}}.
		\end{aligned}
	\end{equation}
\end{theorem}

 \begin{proof}

Let $( \cl{\mathbb{Y}}^i(t,\cd),  \cl{\mathbb{Z}}^{i}(t,\cd))$ be the unique adapted solution to the following  BSDE (parameterized by $t$):
\begin{equation}\label{4.6111}
         \overline{\mathbb{Y}}^i(t,r)=\psi^i(t)+\int_r^Tg(t,s,\overline{Y}^i(s), \mathbb{\overline{Z}}^i(t,s), \overline{\mu}(s))ds-
         \int_r^T\mathbb{\overline{Z}}^i(t,s)dW^i(s).
         \end{equation}
 Thanks to \autoref{pro3.1}, it follows
\begin{equation}\label{4.7111}
 \cl{Y}^i(t)=\overline{\mathbb{Y}}^i(t,t)\q\hbox{and}\q  \cl{Z}^i(t,s)=\overline{ \mathbb{Z}}^i(t,s),\q (t,s)\in\D[0,T].
 \end{equation}
First, by combining
        (\ref{4.4111}) and (\ref{4.7111}), we have that
        $$
        \begin{aligned}
        &Y^{N,i}(t)=\mathbb{Y}^{N,i}(t,t),\q\, Z^{N,i,j}(t,s)=\mathbb{Z}^{N,i,j}(t,s),\q\  (t,s)\in \D[0,T],\\
        &\overline{Y}^i(t)= \overline{\mathbb{Y}}^i(t,t),\qq\q   \overline{Z}^i(t,s)=\mathbb{\overline{Z}}^i(t,s), \qq\qq  (t,s)\in\D[0,T],
         \end{aligned}
        $$
         where the pairs $(Y^{N,i} , Z^{N,i,j})$, $(\overline{Y}^i,\overline{Z}^i)$,
         $(\mathbb{Y}^{N,i}, \mathbb{Z}^{N,i,j})$ and $(\overline{\mathbb{Y}}^i, \mathbb{\overline{Z}}^i)$   are the unique solutions of
         Eqs. (\ref{1.3}), (\ref{5.2}),  (\ref{4.3-111}), (\ref{4.6111}), respectively.
Second, for simplicity presentation, we denote
         \begin{align*}
         & \D \mathbb{Y}^{N,i}(t,r)=\mathbb{Y}^{N,i}(t,r)-\mathbb{\overline{Y}}^{i}(t,r),\q
         \D \mathbb{Z}^{N,i,j}(t,r)=\mathbb{Z}^{N,i,j}(t,r)-\d_{ij}\mathbb{\overline{Z}}^{i}(t,r), \\
          & \D Y^{N,i}(t)=Y^{N,i}(t)- \overline{Y}^{i}(t),\q\q\q\,\,
         \D Z^{N,i,j}(t,r)=Z^{N,i,j}(t,r)-\d_{ij} \overline{Z}^{i}(t,r).
         \end{align*}
         Then, it follows
         \begin{align*}
         \D \mathbb{Y}^{N,i}(t,r)&=\int_r^T (I_1(t,s)+I_2(t,s))ds-\int_r^T\sum\limits_{j=1}^N \D \mathbb{Z}^{N,i,j}(t,s)d W^j(s),
         \end{align*}
        where
         \begin{align*}
         I_1(t,s)&:= g(t,s, Y^{N,i}(s),\mathbb{Z}^{N,i,i}(t,s), \mu^N(s))
                                 -g(t,s, Y^{N,i}(s),\mathbb{\overline{Z}}^{i}(t,s), \mu^N(s)),\\
         I_2(t,s)&:= g(t,s, Y^{N,i}(s),\mathbb{\overline{Z}}^{i}(t,s), \mu^N(s))
                                 -g(t,s,\overline{Y}^{i}(s), \mathbb{\overline{Z}}^{i}(t,s), \overline{\mu}(s)).
         \end{align*}
         Note that, since
         \begin{align*}
         &I_1(t,s)=\widehat{L}^i(t,s)\cd  \D \mathbb{Z}^{N,i,i}(t,s),\\
         &|\widehat{L}^i(t,s)|\leq \phi(|Y^{N,i}(s)|\vee  \cW_2(\mu^N(s), \d_0))(1+|\mathbb{Z}^{N,i,i}(t,s)|+|\mathbb{\overline{Z}}^{i}(t,s)|),\
          (t,s)\in \D[0,T],
         \end{align*}
       then  it follows from  Girsanov's theorem that
                $$\h{W}^{j,i}(s;t)=
\left\{\begin{array}{ll}
\ds W^j(s)-\int_t^s\widehat{L}^i(t,r)dr,\q\ j=i;\\
\ns\ds W^j(s),\qq\qq\qq\q \  \ \  j\neq i, \ t\leq s\leq T
\end{array}\right.$$
  is a Brownian motion under
         the probability $d\widehat{\dbP}^t=\cE(\widehat{\mathbf{L}}(t,\cd)\cd \mathbf{W})_t^Td\dbP$, where
         $\widehat{\mathbf{L}}(t,\cd)=\underbrace{(0,\cd\cd\cd,0,\widehat{L}^i(t,\cd),0,\cd\cd\cd,0)}_N$ and $\mathbf{W}=(W^1,\cds, W^N)$. From which, one gets that
         \begin{align*}
          \D \mathbb{Y}^{N,i}(t,r)&=\int_r^T I_2(t,s)ds-\int_r^T\sum\limits_{j=1}^N \D \mathbb{Z}^{N,i,j}(t,s)d \widehat{W}^{j,i}(s;t),\q r\in[t,T].
         \end{align*}
      Next,   we set $\b= \b^2+\b_0^2+\phi^2(\overline{C}\vee \overline{M}_1) +1$ (Here $\overline{C}$ and $\overline{M}_1$  are given
         \autoref{pro4.4} and \autoref{th3.4}, respectively). Applying It\^{o}'s formula to $e^{\b r} |\D \mathbb{Y}^{N,i}(t,r)|^2$, one gets that for $r\in[t,T]$,
         \begin{equation}\label{4.9-111}
         \begin{aligned}
         &\ e^{\b r}|\D \mathbb{Y}^{N,i}(t,r)|^2+\int_r^Te^{\b s} \(\b |\D \mathbb{Y}^{N,i}(t,s)|^2
         + \sum\limits_{j=1}^N |\D \mathbb{Z}^{N,i,j}(t,s)|^2 \)ds\\
         &=\int_r^Te^{\b s}2\D \mathbb{Y}^{N,i}(t,s)I_2(t,s)ds-\int_r^T e^{\b s}2\D \mathbb{Y}^{N,i}(t,s)
         \sum\limits_{j=1}^N \D \mathbb{Z}^{N,i,j}(t,s)d\widehat{W}^{j,i}(s;t).
         \end{aligned}
         \end{equation}
        On the one hand,  according to \autoref{ass4.1-1}-(ii) and the definition of $\b$, it yields
      that
          \begin{equation*}
          \begin{aligned}
         &\ e^{\b r}|\D \mathbb{Y}^{N,i}(t,r)|^2
         +\widehat{\dbE}^{t}_r\[\int_r^Te^{\b s}  \sum\limits_{j=1}^N |\D \mathbb{Z}^{N,i,j}(t,s)|^2 ds\]\\
         &\leq \widehat{\dbE}^{t}_r\[\int_r^Te^{\b s}
         (|\D Y^{N,i}(s)|^2+\cW_2^2(\mu^N(s),\overline{\mu}(s)) ) ds   \].
         \end{aligned}
          \end{equation*}
          On the other hand, Doob's maximum and H\"{o}lder's  inequalities imply that for $p\geq 2$,
          \begin{equation}\label{4.11-111}
          \begin{aligned}
          &\ \h{\dbE}^{t}_t\[ \sup\limits_{t\leq r\leq T} e^{ \frac{p}{2}\b r}|\D \mathbb{Y}^{N,i}(t,r)|^p     \]
           \leq C_p  \h{\dbE}^{t}_t\[\int_r^Te^{\frac{p}{2}\b s}
         \(|\D Y^{N,i}(s)|^p+\cW_2^p(\mu^N(s),\overline{\mu}(s)) \)ds \].
          \end{aligned}
          \end{equation}
         In particular, by taking $r=t$, it leads to that
         \begin{equation}\label{4.12-111}
         \begin{aligned}
         & \ |\D Y^{N,i}(t)|^p=|\D \mathbb{Y}^{N,i}(t,t)|^p
           \leq C_p  \h{\dbE}^{t}_t\[\int_t^Te^{\frac{p}{2}\b s}
         \(|\D Y^{N,i}(s)|^p+\cW_2^p(\mu^N(s),\overline{\mu}(s))  \)ds \].
         \end{aligned}
         \end{equation}
        Due to the norms $\|\widehat{\mathbf{L}}( t,\cd)\cd \mathbf{W}\|_{\text{BMO}_\dbP[0,T]}$ and
        $\|\widehat{\mathbf{L}}(t,\cd)\cd \mathbf{W}\|_{\text{BMO}_{\dbP,q}[0,T]}$ with $q>2$ being equivalent, thanks to (\ref{2.2}),  there exists a constant $l_0>1$ such that
        $$\mathbb{E}_t\Big[(\cE(\widehat{\mathbf{L}}(t,\cd)\cdot \mathbf{W})_t^T)^{l_0}\]\leq C_{l_0},$$
        which, combining (\ref{4.12-111}) and H\"{o}lder's inequality,  deduce that
         \begin{equation}\label{4.13-111}
         \begin{aligned}
          \ |\D Y^{N,i}(t)|^p
          &\leq C \dbE_t\[\cE(\widehat{\mathbf{L}}(t,\cd)\cdot \mathbf{W})_t^T \cd
         \(\int_t^T|\D Y^{N,i}(s)|^p+\cW_2^p(\mu^N(s),\overline{\mu}(s)) )ds \)    \]\\
         &\leq C\bigg\{\dbE_t\[\int_t^T|\D Y^{N,i}(s)|^{pl_0^*}+\cW_2^{pl_0^*}(\mu^N(s),\overline{\mu}(s))
          ds \]\bigg\}^\frac{1}{l_0^*},
         \end{aligned}
         \end{equation}
         where $l_0^*=\frac{l_0}{l_0-1}>1$.
          This implies that
          \begin{equation*}
          \begin{aligned}
          &\dbE |\D Y^{N,i}(t)|^{pl_0^* }
           \leq C\dbE\[\int_t^T|\D Y^{N,i}(s)|^{pl_0^*}+\cW_2^{pl_0^*}(\mu^N(s),\overline{\mu}(s))
          ds \].
          \end{aligned}
          \end{equation*}
         Moreover, Gronwall's  inequality  show that for a.e. $t\in[0,T]$,
          \begin{equation}\label{4.15-111}
          \begin{aligned}
          &\dbE |\D Y^{N,i}(t)|^{pl_0^* }
           \leq C\dbE\[\int_t^T\cW_2^{pl_0^*}(\mu^N(s),\overline{\mu}(s))ds \].
          \end{aligned}
          \end{equation}
          By inserting (\ref{4.15-111}) into (\ref{4.13-111}), one gets from H\"{o}lder's inequality that
          \begin{equation}\label{4.20-111}
          \begin{aligned}
          \dbE\[|\D Y^{N,i}(t)|^p\]
          &\leq   C\bigg\{\dbE\[
 \int_t^T  \cW_2^{pl_0^* }(\mu^N(s),\overline{\mu}(s))  ds\]\bigg\}^{\frac{1}{l_0^* }}.
          \end{aligned}
          \end{equation}

         On the other hand,  it is  easy to see that
$$d\mathbb{P}=\cE(-\widehat{\mathbf{L}}(t,\cd)\cdot \mathbf{\widehat{W}}^i(\cd;t))_t^T d\widehat{\dbP}^t\q\hbox{and}\q
\mathbf{\widehat{ W}}^i(\cd;t)=(\widehat{ W}^{1,i}(\cd;t),\cds,\widehat{ W}^{N,i}(\cd;t) ).$$
From  (\ref{2.2}),  there exists a  constant $l_1>1$ such that
$$\max\limits_{1\leq i\leq N}\h{\dbE}^{t}_t\[\(\cE(-\widehat{\mathbf{L}}(t,\cd)\cdot \mathbf{\widehat{ W}}^i(\cd;t))_t^T \)^{l_1}\]\leq C_{l_1}.$$
Consequently,  we have from (\ref{4.11-111}) and H\"{o}lder's inequality that
\begin{equation*}
\begin{aligned}
&\ \dbE_t\[ \sup\limits_{t\leq r\leq T}  |\D \mathbb{Y}^{N,i}(t,r)|^p\]\\
&\leq \h{\dbE}^{t}_t\[ \cE(-\widehat{\mathbf{L}}(t,\cd)\cdot \mathbf{\widehat{ W}}^i(\cd;t))_t^T\cd
\( \int_t^T (|\D Y^{N,i}(s)|^p+\cW_2^p(\mu^N(s),\overline{\mu}(s)) )ds                                              \)\]\\
&\leq C \bigg\{\h{\dbE}^{t}_t\[
 \int_t^T |\D Y^{N,i}(s)|^{pl_1^*}+\cW_2^{pl_1^*}(\mu^N(s),\overline{\mu}(s))
  ds\]\bigg\}^{\frac{1}{l_1^*}}\\
 &\leq C \bigg\{\dbE_t\[
 \int_t^T  |\D Y^{N,i}(s)|^{pl_0^*l_1^*}+\cW_2^{pl_0^*l_1^*}(\mu^N(s),\overline{\mu}(s))
  ds\]\bigg\}^{\frac{1}{l_0^*l_1^*}},
\end{aligned}
\end{equation*}
where $l_1^*=\frac{l_1}{l_1-1}$.
This together with (\ref{4.15-111}) provide that
\begin{equation}\label{4.16-111}
\begin{aligned}
&\ \dbE\[ \sup\limits_{t\leq r\leq T}  |\D \mathbb{Y}^{N,i}(t,r)|^p\]
  \leq  C\bigg\{\dbE\[
 \int_t^T  \cW_2^{pl_0^*l_1^*}(\mu^N(s),\overline{\mu}(s))
  ds\]\bigg\}^{\frac{1}{l_0^*l_1^*}}.
\end{aligned}
\end{equation}

Next, let us analyze the term for $\D \mathbb{Z}$. Thanks to (\ref{4.9-111}),  we arrive at, for $r\in[t,T],$
\begin{equation*}
\begin{aligned}
&  \      \h{\dbE}^{t}_t\[ \(  \int_r^Te^{\b s} \sum\limits_{j=1}^N |\D \mathbb{Z}^{N,i,j}(t,s)|^2 ds\)^\frac{p}{2}\]\\
         &\leq  C   \h{\dbE}^{t}_t\[ \( \int_r^Te^{\b s} |\D \mathbb{Y}^{N,i}(t,s)||I_2(t,s)|ds\)^\frac{p}{2}\]\\
        & \q +
          C\h{\dbE}^{t}_t\[ \(\int_r^T e^{2\b s} |\D \mathbb{Y}^{N,i}(t,s)|^2
         \sum\limits_{j=1}^N |\D \mathbb{Z}^{N,i,j}(t,s)|^2ds\)^\frac{p}{4}\]\\
       &\leq  C  \h{\dbE}^{t}_t\[ \( \int_r^Te^{\b s} |\D \mathbb{Y}^{N,i}(t,s)||I_2(t,s)|ds\)^\frac{p}{2}\]\\
        & \q +
          C\h{\dbE}^{t}_t\[\sup\limits_{t\leq s\leq T}e^{\frac{p}{4} \b s} |\D \mathbb{Y}^{N,i}(t,s)|^\frac{p}{2}
           \cd \(\int_r^T e^{\b s}
         \sum\limits_{j=1}^N |\D \mathbb{Z}^{N,i,j}(t,s)|^2ds\)^\frac{p}{4}\].
\end{aligned}
\end{equation*}
The fact $ab\leq \frac{a^2}{2}+\frac{b^2}{2}$ and the Lipschitz property of $g$ with respect to $(y,\mu,\nu)$ lead to
\begin{equation*}
\begin{aligned}
&    \    \h{\dbE}^{t}_t\[ \(  \int_r^Te^{\b s} \sum\limits_{j=1}^N |\D \mathbb{Z}^{N,i,j}(t,s)|^2 ds\)^\frac{p}{2}\]\\
         &\leq C\h{\dbE}^{t}_t\[\int_t^T  |\D Y^{N,i}(s)|^p+\cW_2^p(\mu^N(s),\overline{\mu}(s))ds\]
         +C\h{\dbE}^{t}_t\[\sup\limits_{t\leq s\leq T } |\D \mathbb{Y}^{N,i}(t,s)|^p\].
\end{aligned}
\end{equation*}
Consequently, we get by taking $r=t$ and using reverse H\"{o}lder's inequality that
\begin{equation}\label{4.18-111}
\begin{aligned}
&   \     \dbE \[ \(  \int_t^T  \sum\limits_{j=1}^N |\D \mathbb{Z}^{N,i,j}(t,s)|^2 ds\)^\frac{p}{2}\]\\
         &\leq C\bigg\{\dbE\[
 \int_t^T  |\D Y^{N,i}(s)|^{pl_0^*l_1^*}+\cW_2^{pl_0^*l_1^*}(\mu^N(s),\overline{\mu}(s))
 ds\]\bigg\}^{\frac{1}{l_0^*l_1^*}}\\
         &\q +C\bigg\{\dbE\[\sup\limits_{t\leq s\leq T } |\D \mathbb{Y}^{N,i}(t,s)|^{pl_0^*l_1^*}\]\bigg\}^{\frac{1}{l_0^*l_1^*}}.
\end{aligned}
\end{equation}
Combining (\ref{4.15-111}), (\ref{4.16-111}) and (\ref{4.18-111}), we deduce

\begin{align*}
& \       \dbE \[ \(  \int_t^T  \sum\limits_{j=1}^N |\D \mathbb{Z}^{N,i,j}(t,s)|^2 ds\)^\frac{p}{2}\]\\
&\leq  C\bigg\{\dbE\[
 \int_t^T  \cW_2^{pl_0^*l_1^*}(\mu^N(s),\overline{\mu}(s)) ds\]\bigg\}^{\frac{1}{l_0^*l_1^*}}  +C\bigg\{\dbE\[
 \int_t^T  \cW_2^{p(l_0^*l_1^*)^2}(\mu^N(s),\overline{\mu}(s))ds\]\bigg\}^{\frac{1}{(l_0^*l_1^*)^2}}.
 \end{align*}
Finally, by combining (\ref{4.20-111}), the fact that  $\D Z^{N,i,j}=\D \mathbb{Z}^{N,i,j}$,  and H\"{o}lder's inequality, we derive   the desired result.
\end{proof}

Based on the convergences for the particle systems as stated in
\autoref{th 5.1111}, we have the following   convergence rate for the particle system  (\ref{1.3}).

 \begin{theorem}\label{th 5.6} \rm
              Let \autoref{ass4.1-1} and  $\max\limits_{1\leq i\leq N}\|\psi^i(\cd)\|_{L^\i_{\sF_T}[0,T]}\leq K_1$ hold. If the generator $g$ is independent of the law of $Z(\cd,\cd)$,
                then for any $p\geq2$, there exist  constants $\l>1 $ and
               $C>0$, depending only on $(T,\b,\b_0, K_1,K_2,\phi(\cd),\g,p,\l)$,   such that for $i=1,\cds,N,$
               \begin{equation*}
                   \begin{aligned}
                       \esssup\limits_{t\in[0,T]}\mathbb{E}\Big[| Y^{N,i}(t)- \cl{Y}^i(t)|^p
                       +\(\int_t^T \sum\limits_{j=1}^N|Z^{N,i,j}(t,s)- \d_{ij}\cl{Z}^i(t,s)|^2ds\)^\frac{p}{2} \Big]
                           \leq CN^{-\frac{1}{2\l}}.
                 \end{aligned}
                \end{equation*}
                 \end{theorem}
                 \begin{proof}
                   Let $(\widetilde{Y}^{N,i}(\cd),\widetilde{Z}^{N,i}(\cd,\cd))$ be the solution of Eq. (\ref{4.15111}).
                   Following a similar analysis to that in (\ref{4.15-111}), there exist  constants $l_2^*>1$ and  $C>0$ such that
                   for almost all $t\in[0,T]$ and for $q>2$,
          \begin{equation*}
          \begin{aligned}
          &\dbE | Y^{N,i}(t)-\widetilde{Y}^{N,i}(t)|^{q l_2^*}
           \leq C \dbE \int_t^T\cW_2^{ql_2^*}(\mu^N(s),\overline{\mu}(s)) ds  .
          \end{aligned}
          \end{equation*}
        For the term in  right hand side of the above inequality, H\"{o}lder's inequality  shows  that for $t\in[0,T]$,
             \begin{equation*}\label{8.1.004}\begin{aligned}
               &\ \dbE\bigg[\cW_2^{ql_2^*}(\mu^N(t),\widetilde{\mu}^N(t))\bigg]
             \leq\dbE\bigg[ \bigg( \frac{1}{N}\sum_{i=1}^N|Y^{N,i}(t)-\widetilde{Y}^{N,i}(t)|^2\bigg)^\frac{ql_2^*}{2}\bigg] \\
            & \leq\frac{1}{N}\sum_{i=1}^N\dbE\bigg[ |Y^{N,i}(t)-\widetilde{Y}^{N,i}(t)|^{ql_2^* }\bigg]
             \leq C\dbE \[\int_t^T\cW_2^{ql_2^* }(\mu^N(s),\bar{\mu}(s))ds\],
             \end{aligned}\end{equation*}
             where $\widetilde{\mu}^N(t)$ is given in (\ref{6.2}).
            Moreover, the triangle inequality implies that for $t\in[0,T]$,
              \begin{equation*}\label{8.1.005}
              \begin{aligned}
              & \ \dbE \bigg[\cW_2^{ql_2^*}(\mu^N(t), \bar{\mu}(t))\bigg]
             \leq  \dbE \bigg[\(\cW_2(\mu^N(t),\widetilde{\mu}^N(t))
             +\cW_2(\widetilde{\mu}^N(t), \bar{\mu}(t))\)^{ql_2^*}\bigg]\\
                &\leq2^{ql_2^*} \dbE \bigg[\cW_2^{ql_2^*}(\mu^N(t),\widetilde{\mu}^N(t))
                +\cW_2^{ql_2^*}(\widetilde{\mu}^N(t), \bar{\mu}(t))\bigg]\\
                &\leq C \dbE\[\int_t^T\cW_2^{ql_2^*}(\mu^N(s), \bar{\mu}(s))ds\]
             +C\dbE\bigg[\cW_2^{ql_2^*}(\widetilde{\mu}^N(t), \bar{\mu}(t))\bigg].
             \end{aligned}\end{equation*}
             Hence, by using Gronwall's inequality, we have that for any $q>2$,
             \begin{align*}
              \dbE \bigg[\cW_2^{ql_2^*}(\mu^N(t), \bar{\mu}(t))\bigg] \leq C\dbE\bigg[\cW_2^{ql_2^*}(\widetilde{\mu}^N(t), \bar{\mu}(t))\bigg].
             \end{align*}
            Now, by applying  Hao et al. \cite[(5.14)]{Hao-Hu-Tang-Wen-2022} and Fournier and Guillin \cite[Theorem 1]{Fournier-Guillin-2015}, we deduce that for $q>2$,
             \begin{align*}
              \esssup_{t\in[0,T]} \dbE \bigg[\cW_2^{ql_2^*}(\mu^N(t), \bar{\mu}(t))\bigg] \leq C \esssup_{t\in[0,T]}\dbE\bigg[\cW_{ql_2^*}^{ql_2^*}(\widetilde{\mu}^N(t), \bar{\mu}(t))\bigg] \leq N^{-\frac{1}{2}}.
             \end{align*}
                 Finally, it follows  from \autoref{th 5.1111} and H\"{o}lder's inequality that
                 \begin{equation*}
            \begin{aligned}
                 &\ \esssup\mathbb{E}\Big[| Y^{N,i}(t)- \cl{Y}^i(t)|^p
                       +\(\int_t^T \sum\limits_{j=1}^N|Z^{N,i,j}(t,s)- \d_{ij}\cl{Z}^i(t,s)|^2ds\)^\frac{p}{2} \Big]\\
               &\leq C\mathbb{E}\Big[\int_0^T \mathcal{W}^{p(l^*_0l^*_1)^2 l_2^*}_2(\mu^{N}(s),  \cl{ \mu  }(s))ds\Big]^\frac{1}{(l^*_0l^*_1)^2l_2^*}
                \leq N^{-\frac{1}{2(l^*_0l^*_1)^2 l_2^*}}.
            \end{aligned}
            \end{equation*}
            This completes the proof.
                 \end{proof}

\section{Appendix}\label{appendix}

In this section, we give the detailed proofs for the  supported results used above.
%\subsection{Proof of \autoref{pro3.1}}

\subsection{Proof of \autoref{pro3.4}}

The proof is divided into two steps.

\textbf{Step 1.} The map $\Gamma$ is well-defined and stable in $\mathcal{B}_\varepsilon(\bar{R}_1, \bar{R}_2)$.
For given $y(\cd)\in L_{\dbF}^\i [0,T]$ and $z(\cd,\cd)\in  \cZ^2_{\dbF}(\D[0,T];\dbR^d)$,   consider
 \begin{equation}\label{3.109}
		Y^{y,z}(t)=\psi(t)+\int_t^Tg(t,r,y(r),Z^{y,z}(t,r),\dbP_{y(r)}, \dbP_{z(t,r)})dr-\int_t^TZ^{y,z}(t,r)dW(r), \  t\in[0,T].
	\end{equation}
We  study the above equation  using the parameterization method for $t$. More precisely, consider a  BSDE
      \begin{equation}\label{3.110}
		\mathcal{Y}^{y,z}(t,s)=\psi(t)+\int_s^Tg(t,r,y(r),\mathcal{Z}^{y,z}(t,r), \dbP_{y(r)}, \dbP_{z(t,r)})dr
-\int_s^T\cZ^{y,z}(t,r)dW(r),\  s\in[t,T].
	 \end{equation}
For each $t\in[0,T]$, according to \cite[Proposition 3.2]{Hao-Hu-Tang-Wen-2022}, BSDE (\ref{3.110})
 admits a unique solution $(\cY^{y,z}(t,\cd), \cZ^{y,z}(t,\cd))\in L^\i_{\dbF}(\O;C([t,T];\dbR))\ts \cZ^2_{\dbF}([t,T];\dbR^d)$,
 and moreover, for $s\in[t,T]$ and $\t\in\sT[t,T]$,
\begin{equation} \label{3.19-1}
 \begin{aligned}
 &|\cY^{y,z}(t,s)|\leq  \|\psi(\cd)\|_{L^\i_{\sF_T} [0,T]}+\|\ell(\cd,\cd)\|_{L^{\i}_{\dbF}(\D[0,T])}+(T-t) (\b+\b_0)\|y(\cd)\|_{L^\i_{\dbF}[t,T]},\\
 &  \dbE_\t\[\int^T_\t|\cZ^{y,z}(t,r)|^2dr\] \leq
        \frac{1}{\gamma^2}\exp\big(2\gamma\|\psi(\cd)\|_{L^\i_{\sF_T} [0,T]}\big)
+\frac{2}{\gamma}\exp\big(2\gamma\|\cY^{y,z}(t,\cd)\|_{L_\dbF^\infty[t,T]}\big)\\
&\quad\quad \times \bigg(\|\ell(\cd,\cd)\|_{L^{\i}_{\dbF}(\D[0,T])}+
(T-t) (\b+\b_0)\|y(\cd)\|_{L^\i_{\dbF}[t,T]}\bigg).
 \end{aligned}
 \end{equation}
Define
	\begin{equation*}
		Y^{y,z}(t)= \mathcal{Y}^{y,z}(t,t), \q Z^{y,z}(t,s)=\mathcal{Z}^{y,z}(t,s),\q   (t,s)\in\D[0,T].
	\end{equation*}
	Then the pair $(Y^{y,z}(\cd), Z^{y,z}(\cd,\cd))\in L_{\dbF}^\i[0,T]\ts \cZ^2_{\dbF}(\D[0,T];\dbR^d)$
	is an adapted solution of Eq. (\ref{3.109}).
 In addition, similar to the proof of \autoref{pro3.1}, one can show the uniqueness.
Taking $s=t$ in (\ref{3.19-1}), it follows from  H\"{o}lder's inequality that
         	\begin{equation}\label{3.8}
         \begin{aligned}
& |Y^{y,z}(t)| \leq  \|\psi(\cd)\|_{L^\i_{\sF_T} [0,T]}+\Big\{\||\ell(\cd,\cd)|^2\|_{L^{\i}_{\dbF}(\D[0,T])}\Big\}^\frac{1}{2}+(T-t) (\b+\b_0)\|y(\cd)\|_{L^\i_{\dbF}[t,T]},\\
   &  \dbE_\t\[\int^T_\t|Z^{y,z}(t,r)|^2dr\] \leq
        \frac{1}{\gamma^2}\exp\big(2\gamma\|\psi(\cd)\|_{L^\i_{\sF_T} [0,T]}\big)
+\frac{2}{\gamma}\exp\big(2\gamma\|\cY^{y,z}(t,\cd)\|_{L_\dbF^\infty[t,T]}\big)\\
&\quad\quad \times \bigg(\Big\{\||\ell(\cd,\cd)|^2\|_{L^{\i}_{\dbF}(\D[0,T])}\Big\}^\frac{1}{2}+
(T-t) (\b+\b_0)\|y(\cd)\|_{L^\i_{\dbF}[t,T]}\bigg).
		\end{aligned}
\end{equation}
 Define
    \begin{align*}
     \e_1=\frac{1}{2(\b+\b_0)},\q  \e_2=\frac{\g \bar{R}_2e^{-2\g \bar{R}_2}}{4(\b+\b_0)\bar{R}_1}.
    \end{align*}
    If $\|y(\cd)\|_{L^\i_{\dbF}[T-\e,T]}\leq \bar{R}_1$ and $\|z(\cd,\cd)\|^2_{ \cZ^2_{\dbF}(\D[T-\e,T])}\leq \bar{R}_2$,
    then it follows (\ref{3.19-1}) and (\ref{3.8}) that
    $$
    \begin{aligned}
    &\esssup_{t\in[T-\e,T]}\|\cY^{y,z}(t,\cd)\|_{L_\dbF^\infty[t,T]}\leq \bar{R}_1,\  \|Y^{y,z}(\cd)\|_{L^\i_{\dbF}[T-\e,T]}\leq \bar{R}_1,\\
    & \|Z^{y,z}(\cd,\cd)\|^2_{ \cZ^2_{\dbF}(\D[T-\e,T])}\leq \bar{R}_2,\ \forall \e\in(0,\e^\#],
     \end{aligned}
     $$
where $\e^\#=\min\{\e_1,\e_2\}$.  This means that the map
$$
(Y^{y,z}(\cd),Z^{y,z}(\cd,\cd))\deq \G(y(\cd),z(\cd,\cd))
$$
is well-defined and stable in  $\sB_\e(\bar{R}_1,\bar{R}_2)$.

\ms
	
	\textbf{Step 2.}  Prove that $\G$ is a contraction mapping.
	
	\ms

 Let $(y(\cdot),z(\cdot,\cdot))$ and $(\overline{y}(\cdot),\overline{z}(\cdot,\cdot))$ be elements of $\mathcal{B}_\varepsilon(\bar{R}_1,\bar{R}_2)$
 and let $(\cY^{y,z},\cZ^{y,z}), (\cY^{\overline{y},\overline{z}}, \cZ^{\overline{y},\overline{z}})$ the  unique solution
 to Eq. (\ref{3.110})  with $(y,z)$ and $(\overline{y},\overline{z})$, respectively.
 For $h=y,z,Y$  and $Z$, we denote $\d h=h-\bar{h}$  and
 $\d\cY=\cY^{y,z}-\cY^{\overline{y},\overline{z}}, \d\cZ=\cZ^{y,z}-\cZ^{\overline{y},\overline{z}}.$
Then, for almost all $t\in[0,T],$
	\begin{equation}\label{6.111}
		\begin{aligned}
			&\d \cY(t,r)+\int_r^T\d \cZ(t,s)dW(s)= \int_r^T \breve{J}(t,s)ds\\
			&\q +\int_r^T\big(g(t,s,y(s),\cZ^{y,z}(t,s),\dbP_{y(s)}, \dbP_{z(t,s)})
			-g(t,s,y(s), \cZ^{\overline{y},\overline{z}}(t,s),\dbP_{y(s)}, \dbP_{z(t,s)})ds, \ r\in[t,T],
		\end{aligned}
	\end{equation}
	where
	\begin{equation*}
		\breve{J}(t,s)\deq  g(t,s,y(s),\cZ^{\overline{y},\overline{z}}(t,s),\dbP_{y(s)}, \dbP_{z(t,s)})-
		g(t,s,\cl{y}(s),\cZ^{\overline{y},\overline{z}}(t,s),\dbP_{\cl{y}(s)}, \dbP_{\cl{z}(t,s)}).
	\end{equation*}
	 \autoref{ass4.1-1} implies that there exists a stochastic process $\breve{L}(\cd,\cd)$ such that, for almost all  $t\in [T-\e,T]$,
	\begin{equation*}
		\left\{
		\begin{aligned}
		\ds 	& g(t,s,y(s),\cZ(t,s),\dbP_{y(s)}, \dbP_{z(t,s)})
			-g(t,s,y(s),\cZ^{\overline{y},\overline{z}}(t,s),\dbP_{y(s)}, \dbP_{z(t,s)})=\breve{L}(t,s)\d \cZ(t,s),\\
	\ns\ds 		& |\breve{L}(t,s)|\leq \phi(|y(s)|\vee \cW_2(\dbP_{y(s)},\d_0))(1+|\cZ(t,s)|+|\cZ^{\overline{y},\overline{z}}(t,s)|),\ s\in [t, T].
		\end{aligned}
		\right.
	\end{equation*}
Note that	since all  the pairs $(y(\cd),z(\cd,\cd))$, $(Y(\cd),\cZ(\cd,\cd)),$ and $ ( \cY^{\overline{y},\overline{z}} (\cd), \cZ^{\overline{y},\overline{z}}(t,s)(\cd,\cd))$  belong to the space
	$ \sB_\e(\bar{R}_1,\bar{R}_2)$,
	one has  that
	\begin{equation*}
		\begin{aligned}
		\ds 	\|\breve{L}(\cd,\cd)\|_{\overline{\cZ}^2_{\dbF}(\D[T-\e,T])}
			&\leq \phi(\|y(\cd)\|_{L_{\dbF}^\i[T-\e,T]})(1+\|\cZ(\cd,\cd)\|_{\overline{\cZ}^2_{\dbF}(\D[T-\e,T])}
			+\|\cZ^{\overline{y},\overline{z}}(\cd,\cd)\|_{\overline{\cZ}^2_{\dbF}(\D[T-\e,T])})\\
		\ns\ds 	&\leq \phi(\bar{R}_1)(1+2\sqrt{\bar{R}_2}).
		\end{aligned}
	\end{equation*}
	Consequently, the process
		$\breve{W}(u;t)\deq W(r)-\int_t^u\breve{L}(t,s)ds,\   u\in[t,T]$
	is a standard Brownian motion under the probability
	$
		d\breve{\dbP}^t\deq \cE(\breve{L}(t,\cd)\cd W)^T_td\dbP.
$
	Eq. (\ref{6.111}) can be rewritten as
	\begin{equation*}
		\begin{aligned}
			&\d \cY(t,\t)+\int_\t^T\d \cZ(t,s)d\breve{W}(s;t)=\int_\t^T \breve{J}(t,s)ds,\q  \t\in\sT[T-\e,T],
		\end{aligned}
	\end{equation*}
	which implies that for any stopping time $\t\in\sT[T-\e,T]$,
	\begin{equation*}
		|\d \cY(t,\t)|^2+\breve{\dbE}^t_\t\int_\t^T|\d \cZ(t,s)|^2ds=\breve{\dbE}^t_\t\(\int_\t^T |\breve{J}(t,s)|ds\)^2.
	\end{equation*}
%	where $\cl{\dbE}^t_\t[\cd]=\cl{\dbE}^t[\cd|\sF_\t]$ denotes the conditional expectation under the probability $\cl{\dbP}^t$.
	In particular, when $\t=t$, we have from  the fact $\d Z=\d \cZ$ that
	\begin{equation*}
		|\d Y(t)|^2+\breve{\dbE}^t_t\int_t^T|\d Z(t,s)|^2ds=\breve{\dbE}^t_t\(\int_t^T |\breve{J}(t,s)|ds\)^2.
	\end{equation*}
	Notice that, for  $T-\e\leq t\leq s\leq T$, one has
	\begin{equation*}
		\begin{aligned}
		 	|\breve{J}(t,s)|&\leq
			\phi(|y(s)|\vee|\overline{y}(s)|\vee \cW_2(\dbP_{y(s)},\d_0)
                 \vee \cW_2(\dbP_{\overline{y}(s)},\d_0) )\cW_2(\dbP_{z(t,s)},\dbP_{\overline{z}(t,s)})\\
			 &\q +\b|\d y(s)|+\b_0\cW_2(\dbP_{y(s)},\dbP_{\overline{y}(s)}),\\
%%%%%%
%%%%%
%%%%%%%%%%%%%
%%%%%%%%%%
%%%%%%%%%%
 	&\leq\phi(\bar{R}_1)\|\d z(t,s)\|_{L^2(\O)}+\b|\d y(s)|+\b_0\|\d y(s)\|_{L^2(\O)}.
		\end{aligned}
	\end{equation*}
	Consequently, one can deduce from H\"{o}lder's inequality that
	\begin{equation*}
		\begin{aligned}
		\ds 	& |\d Y(t)|^2+\breve{\dbE}^t_\t\int_\t^T|\d Z(t,s)|^2ds\\
		\ns\ds	&\leq 8\e   \[ \phi^2(\bar{R}_1)\int_\t^T\|\d z(t,s)\|^2_{L^2(\O)}ds+(\b^2 +\b_0^2)  \|\d y(\cd)\|^2_{L^\i_{\dbF}[T-\e,T]} \].
		\end{aligned}
	\end{equation*}
	Applying  (\ref{2.5})  leads to
	\begin{equation*}
		\begin{aligned}
			  \|\d Y(\cd)\|^2_{L^\i_{\dbF}[T-\e,T]}+(c_1)^2\|\d Z(\cd,\cd)\|^2_{ \cZ^2_{\dbF}(\D[T-\e,T])}
			\leq C_0 \e\Big\{  \|\d y(\cd)\|^2_{L^\i_{\dbF}[T-\e,T]}+
			\|\d  z(\cd,\cd)\|^2_{\cZ^2_{\dbF}(\D[T-\e,T])}  \Big\},
		\end{aligned}
	\end{equation*}
where $C_0$ depends on $c_2, \bar{R}_1,\bar{R}_2,T$ and $\phi(\cd)$.
Now, by	choosing $\e$ small enough, we obtain
	\begin{equation*}
		\begin{aligned}
		\ds 	&  \|\d Y(\cd)\|^2_{L^\i_{\dbF}[T-\e,T]}+\|\d Z(\cd,\cd)\|^2_{ \cZ^2_{\dbF}(\D[T-\e,T])}\\
	  	&  \leq \frac{1}{2} \(\|\d y(\cd)\|^2_{L^\i_{\dbF}[T-\e,T]}+ \|\d z(\cd,\cd)\|^2_{\cZ^2_{\dbF}(\D[T-\e,T])}\).
		\end{aligned}
	\end{equation*}
	Thus $\Gamma$ is a contraction mapping on the space $\mathcal{B}_\varepsilon(\bar{R}_1, \bar{R}_2)$, which implies that BSVIE \eqref{1.1} admits a unique adapted solution.

\subsection{Proof of \autoref{th 4.2}}

Before proving  \autoref{th 4.2}, we give the following lemma, which states an  a priori estimate.

\begin{assumption}\label{assumption5.111}\rm
Suppose that $g:\O\ts \D[0,T]\ts\dbR^{d}\ts \cP_2(\dbR^{d})\ra \dbR $ is
$\sF_T\otimes\sB(\D[0,T]\ts\dbR^{d}\ts   \cP_2(\dbR^{d}))$
measurable  such that $s\mapsto g(t,s,z,\nu)$ is $\dbF$-progressively measurable for
all $(t,z,\nu)\in[0,T]\ts  \dbR^{ d} \ts  \cP_2(\dbR^{d}) $, and for given $y(\cd)\in L_{\dbF}^\i [0,T] $, the following conditions hold:

\begin{enumerate}[~~\,\rm (i)]
	\item
	For any $(t,s)\in\D[0,T]$,  $z\in \dbR^d$, $\nu\in\cP_2(\dbR^d)$, $\dbP$-a.s.,
	\begin{equation*}
		\begin{aligned}
		& |g(t,s, z,\nu)|\leq \frac{\gamma}{2}|z|^2+\ell(t,s)+\b|y(s)|
		+\b_0\|y(s)\|_{L^2(\Om)}+ \gamma_0\cW_2(\nu,\d_{0})^{1+\alpha}.
	\end{aligned}
\end{equation*}
	\item
	For any
	$(t,s)\in\D [0,T]$, $z,\bar{z}\in \dbR^d$, $\nu,\overline{\nu}\in\cP_2(\dbR^d)$, $\dbP$-a.s.,
	$$
	\begin{aligned}
\ds 	&	|g(t,s,z,\nu)-g(t,s, \bar{z},\overline{\nu})|\\
	\ns\ds 	&\leq\phi\big(|y(s)|\vee\|y(s)\|_{L^2(\Om)}\big)\cd\big[\big(1+|z|+|\bar{z}|\big)|z-\bar{z}|
		+(1+\cW_2(\nu,\d_{0})^{\a}+\cW_2(\overline{\nu},\d_{0})^\a)\cW_2(\nu,\overline{\nu})\big].
	\end{aligned}$$
\item   For $(t,s)\in \D[0,T]$  and  $(z, \n)\in \dbR^d \ts \cP_2(\dbR^d)$, it holds that $\dbP$-a.s.,
        $$
             g(t,s, z,\nu)\leq
            -\frac{\tilde{\gamma}}{2}|z|^2+\ell(t,s)+\beta|y(s)|+\beta_0  \|y(s)\|_{L^2(\Om)}
             +\gamma_0  \cW_2(\n,\delta_{ 0 })^{1+\alpha}
         $$
or
          \begin{equation*}
           g(t,s,z,\nu)\geq \frac{\tilde{\gamma}}{2}|z|^2-\ell(t,s)
             -\beta|y(s)|-\beta_0  \|y(s)\|_{L^2(\Om)}
                     -\gamma_0   \cW_2(\n,\delta_{0})^{1+\alpha}.
           \end{equation*}
	\item   The free term $\psi(\cd)$ is bounded with $\|\psi(\cd)\|_{L^\i_{\sF_T} [0,T] }\leq K_1$, and the process $\ell(\cd,\cd)$  belongs to the space
	$  L^\i([0,T];  L^{1,\i}_{\dbF}([\cd,T];\dbR^+))$ with
	$ \|\ell(\cd,\cd)\|_{L^{\i}_{\dbF}(\D[0,T])} \leq K_3$.
\end{enumerate}
\end{assumption}

Note that those constants $L_1,\cds, L_6$ used in the proof of the following \autoref{pro 4.1} are introduced in (\ref{4.1-1}).

\begin{lemma}\label{pro 4.1}\rm
Under \autoref{assumption5.111} with given $(y(\cd),z(\cd,\cd))\in L_{\dbF}^\i[0,T]\ts   \cZ^2_{\dbF}(\D[0,T];\dbR^d)$,
	  Eq. (\ref{3.1110}) with $(y(\cd),z(\cd,\cd))$
	possesses a unique adapted solution $(Y(\cd),Z(\cd,\cd))\in L_{\dbF}^\i[0,T]\times  \cZ^2_{\dbF}(\D[0,T];\dbR^d)$.
	Moreover, there exists  a   positive constant $\overline{L}$  depending on    $K_1,K_3,\tilde{\g}, \g_0,\b,\b_0,\a,T$ such that
	\begin{align}
\ds &		\|Y(\cd)\|_{L^\i_{\dbF}[t,T]} \leq  \overline{L}\(1+\int_t^T\|y(\cd)\|_{L^\i_\dbF[s,T]}   ds\), \q  \forall t\in[0,T], \label{4.2-1}\\
\ns\ds 	&		 \|Z(\cd,\cd)\|^2_{ \cZ^2_\dbF(\D[0,T])} \leq  \overline{L}    \exp\( \overline{L} \int_0^T\|y(\cd)\|_{L^\i_\dbF[s,T]}  ds\). \label{4.2-2}
		\end{align}
\end{lemma}
%%%%%%%%%%%%%%%%%%%%%%%%%%%%%%%%%%%%%%%%

%%%%%%%%%%%%%%%%%%%%%%%%%%%%%%%%%%%%%%%%
\begin{proof}

	For  fixed  $(y(\cd),z(\cd,\cd))\in L_{\dbF}^\i[0,T]\ts   \cZ^2_{\dbF}(\D[0,T];\dbR^d)$  and for almost all $t\in[0,T]$,
	under  \autoref{assumption5.111},  Hao et al. \cite[Theorem 3.8]{Hao-Hu-Tang-Wen-2022} can  show that the mean-field QBSDE (\ref{3.211})
	possesses a unique solution $(\cY(t,\cd), \cZ(t,\cd))\in L^\i_{\dbF}(\O;C([t,T];\dbR))\ts  \cZ^2_{\dbF}([t,T];\dbR^d)$.
	 Thanks to  \autoref{pro3.1}, Eq. (\ref{3.1110}) with $(y(\cd),z(\cd,\cd))$
	 admits a unique adapted solution $(Y(\cd),Z(\cd,\cd))\in L_{\dbF}^\i[0,T]\times  \cZ^2_{\dbF}(\D[0,T];\dbR^d)$.
	
	\ms
	
	Next, we prove the estimates (\ref{4.2-1}) and (\ref{4.2-2}),  whose proof will be  split into several steps.

\ms
	
	\textbf{Step 1.}   Estimate the following term  $$\exp(\tilde{\g}|\cY(t,r)|)+\dbE_r\int_r^T|\cZ(t,s)|^2ds,\  r\in[t,T].$$

 For this, we define
	\begin{equation*}
		M\deq  3\esssup_{t\in[0,T]}\esssup_{\t\in\sT[t,T]}\Big\| \dbE_\t\[\int_\t^T|\cZ(t,s)|^2ds \] \Big \|_\i<\i.
	\end{equation*}
Then, it follows  from Young's inequality that  for any $p\geq1$,
	\begin{equation*}
		p\Big\{\dbE|\cZ(t,s)|^2\Big\}^{\frac{1+\a}{2}}\leq \frac{1}{M}\dbE|\cZ(t,s)|^2+L_{M,\a,p},
	\end{equation*}
	where $L_{M,\a,p}$ depends on $M,\a$ and $p$.
 John-Nirenberg's inequality implies that
	\begin{equation*}
		\dbE_t\[ \exp\Big\{\int_t^T \frac{1}{M} |\cZ(t,s)|^2ds\Big\} \]=
		\dbE_t\[ \exp\Big\{\int_t^T   |\widehat{\cZ}(t,s)|^2ds\Big\} \]\leq\frac{3}{2},\q t\in[0,T],
	\end{equation*}
	where $\widehat{\cZ}(t,s)\deq \frac{\cZ(t,s)}{\sqrt{M}}$ with 	$\|\widehat{\cZ}(t,\cd)\|_{ \cZ^2_\dbF[t,T]}\leq \frac{1}{3}$.
   Hence,   Jensen's inequality allows us to show
	\begin{equation*}
		\begin{aligned}
			& \dbE_t\[\exp\Big\{p\int_t^T\|\cZ(t,s)\|_{L^2(\O)}^{1+\a}ds \Big\} \]=\exp\Big\{p\int_t^T\big\{\dbE|\cZ(t,s)|^2\big\} ^{\frac{1+\a}{2}}ds \Big\}\\
			&\leq  \widetilde{L} \exp\Big\{\int_t^T \frac{1}{M}\dbE|\cZ(t,s)|^2ds  \Big\}
			\leq \widetilde{L} \dbE\[ \exp\Big\{\int_t^T \frac{1}{M} |\cZ(t,s)|^2ds\Big\}  \]\\
			&\leq \widetilde{L }\dbE\[\dbE_t\[ \exp\Big\{\int_t^T \frac{1}{M} |\cZ(t,s)|^2ds\Big\} \] \]\leq \frac{3}{2}\widetilde{L}, \  t\in[0,T],
		\end{aligned}
	\end{equation*}
	where $\widetilde{L}$ is a constant depending only on $M,\a,p$ and $T$.
	From this and  note that  $y(\cd)\in L^\i_{\dbF}[0,T]$ and $\|\ell(\cd,\cd)\|_{L^{\i}_{\dbF}(\D[0,T])}\leq K_3$,  one could   check that
	\begin{align*}
		\overline{\ell}(t,s)&\deq \ell(t,s)+\b|y(s)|
		+\b_0\|y(s)\|_{L^2(\Om)}+\gamma_0\|\cZ(t,s)\|_{L^2(\Om)}^{1+\alpha}\\
& \in  \bigcap\limits_{p\geq 1}\cE_{\dbF}^{p}(\O;L^1([t,T];\dbR^+)).
	\end{align*}
	Consequently, according to  Fan, Hu, Tang  \cite[Proposition 2.1]{Fan-Hu-Tang-2019} (see also Fan,Wang,Yong \cite[Lemma 2.7]{Fan-Wang-Yong-2023}),
	there exists a positive constant $L_0$ depending on $K_1, K_3,\b,\b_0$ and $\g_0$ such that, for
	almost all $t\in[0,T]$ and for  $r\in[t,T]$,
	\begin{equation}\label{3.8-1}
		\begin{aligned}
			& \exp\big\{\tilde{\g}|\cY(t,r)| \big \}+\dbE_r\int_r^T|\cZ(t,s)|^2ds\\
			& \leq L_0\dbE_r\exp\Big\{L_0|\psi(t)|+L_0\int_r^T[\ell(t,s)+\b|y(s)|
			+\b_0\|y(s)\|_{L^2(\Om)}+\gamma_0\|\cZ(t,s)\|_{L^2(\Om)}^{1+\alpha}] ds   \Big\} \\
			&\leq L_0\exp\Big\{L_0(K_1+K_3)\Big\}\cdot    \exp\Big\{L_0(\b+\b_0)\int_r^T\|y(\cd)\|_{L^\i_\dbF[s,T]}  ds\Big\}\\
			 &\q\ \cdot \exp\Big\{L_0\g_0\int_r^T\Big[\dbE[|\cZ(t,s)|^2]\Big]^{\frac{1+\a}{2}}ds\Big\}.
		\end{aligned}
	\end{equation}

	\textbf{Step 2.}    Estimate the last term of the previous inequality, i.e.,
	$$L_0\g_0\int_r^T\Big\{\dbE[|\cZ(t,s)|^2]\Big\}^{\frac{1+\a}{2}}ds.$$
	
	Define the stopping time
	$$\tau_k=\inf\Big\{u\in[r,T]: \int_r^u|\cZ(t,s)|^2ds\geq k\Big\}\wedge T.$$
Then, thanks to the strictly quadratic condition (iii) of \autoref{assumption5.111}, we have that for almost all $t\in[0,T]$,
	\begin{equation*}
		\begin{aligned}
			\ds \cY(t,& r) -\cY(t,\tau_k)+\int_r^{\tau_k}\cZ(t,s)dW(s)
			=\int_r^{\tau_k}g(t,s,y(s),\cZ(t,s), \dbP_{y(s)}, \dbP_{\cZ(t,s)})ds\\
			\ns\ds &\geq\int_r^{\tau_k}\(
			\frac{\tilde{\gamma}}{2}|\cZ(t,s)|^2-\ell(t,s)-\beta|y(s)|-\beta_0\|y(s)\|_{L^2(\Om)}
			-\gamma_0\|\cZ(t,s)\|_{L^2(\Om)}^{1+\alpha}\)ds,
		\end{aligned}
	\end{equation*}
	which by taking the expectation on both sides, implies that
	\begin{equation*}
		\begin{aligned}
			   \dbE \int_r^{\tau_k} \frac{\tilde{\gamma}}{2}|\cZ(t,s)|^2ds
			&   \leq \dbE\int_r^{T}\(\ell(t,s)+\beta|y(s)|+\beta_0\|y(s)\|_{L^2(\Om)}
			+\gamma_0\|\cZ(t,s)\|_{L^2(\Om)}^{1+\alpha}\)ds \\
&\q +			
			\dbE\big[ \cY(t,r)-\cY(t,\tau_k)\big].
		\end{aligned}
	\end{equation*}
	Letting $k\ra \i$, and note   the boundedness of $\ell(\cd,\cd)$ and $\psi(\cd)$,   we have from Fatou's lemma  that
	\begin{equation} \label{4.5-1}
		\begin{aligned}
			&   \dbE \int_r^{T} \frac{\tilde{\gamma}}{2}|\cZ(t,s)|^2ds\\
			&   \leq\dbE\Big\{ |\cY(t,r)|+|\psi(t)|
			+\int_r^{T}\(\ell(t,s)+\beta|y(s)|+\beta_0\|y(s)\|_{L^2(\Om)}
			+\gamma_0\|\cZ(t,s)\|_{L^2(\Om)}^{1+\alpha}\)ds \Big\}\\
			&   \leq  \dbE[|\cY(t,r)|] +K_1+K_3
			+\int_r^{T}\( (\beta+\beta_0) \|y(\cd)\|_{L^\i_{\dbF} [s,T] }
			+\gamma_0\|\cZ(t,s)\|_{L^2(\Om)}^{1+\alpha}\)ds .
		\end{aligned}
	\end{equation}
	Recall  that, for two positive constants $a$, $b$,  and  $\a\in[0,1)$, Young's inequality implies that
	\begin{equation}\label{4.7-1}
		a b^{1+\a}=\left(\left(\frac{1+\a}{2}\right)^{\frac{1+\a}{1-\a}} a^{\frac{2}{1-\a}}\right)^{\frac{1-\a}{2}}\left(\frac{2}{1+\a} b^{2}\right)^{\frac{1+\a}{2}} \leq b^{2}+\frac{1-\a}{2}\left(\frac{1+\a}{2}\right)^{\frac{1+\a}{1-\a}} a^{\frac{2}{1-\a}}.
	\end{equation}
	We let $\e_0$ be a positive constant (to be specified later), and take
	$a=\frac{4\gamma_0}{\tilde{\gamma}}$ and $b=\|\cZ(t,s)\|_{L^2(\Om)}$, then \rf{4.7-1} implies that
	\begin{equation}\label{4.8}
		\begin{aligned}
			\e_0\gamma_0\|\cZ(t,s)\|^{1+\alpha}_{L^2(\Om)}&=\frac{\tilde{\gamma}\e_0}{4}
			\Big[\frac{4\gamma_0}{\tilde{\gamma}}\|\cZ(t,s)\|^{1+\alpha}_{L^2(\Om)}\Big]
			\leq \frac{\tilde{\gamma}\e_0}{4}\dbE|\cZ(t,s)|^2+L_1.
		\end{aligned}
	\end{equation}
	Multiplying both sides of (\ref{4.5-1}) by $\e_0$ and substituting (\ref{4.8}) yields
	\begin{equation*}
		\begin{aligned}
			 \dbE\int_r^T \frac{\tilde{\gamma}\e_0}{2}|\cZ(t,s)|^2ds
			&   \leq  \e_0  \dbE |\cY(t,r)|  + (K_1+K_3)\e_0
			+ (\beta+\beta_0)\e_0\int_r^{T}\|y(\cd)\|_{L^\i_{\dbF} [s,T] } ds\\
			&\q +\dbE\int_r^T \frac{\tilde{\gamma}\e_0}{4}|\cZ(t,s)|^2ds+L_1 T.
		\end{aligned}
	\end{equation*}
	In other words,
	\begin{equation}\label{4.9}
		\begin{aligned}
			\dbE\int_r^T \frac{\tilde{\gamma}\e_0}{4}|\cZ(t,s)|^2ds
			\leq  \e_0  \dbE |\cY(t,r)|  + (K_1+K_3)\e_0+L_1 T
			+ (\beta+\beta_0)\e_0\int_r^{T}\|y(\cd)\|_{L^\i_{\dbF} [s,T] } ds.
		\end{aligned}
	\end{equation}
	By (\ref{4.7-1}) again, for any $\d>0$, we have
	\begin{equation}\label{4.9-1}
 	\d\|\cZ(t,s)\|_{L^2(\Om)}^{1+\alpha}\leq\frac{\tilde {\gamma}\e_0}{4}\dbE|\cZ(t,s)|^2+L_{2,\d},
	\end{equation}
	which combining (\ref{4.9})  leads to
	\begin{equation}\label{4.10}
		\begin{aligned}
			& \int_r^TL_0\g_0\Big\{\dbE[|\cZ(t,s)|^2]\Big\}^{\frac{1+\a}{2}}ds=  \int_r^T L_0\g_0\|\cZ(t,s)\|_{L^2(\Om)}^{1+\alpha} ds\\
			&  \leq\int_r^T \frac{\tilde {\gamma}\e_0}{4} \dbE|\cZ(t,s)|^2ds+L_{2,L_0\g_0}T\\
			&\leq  \e_0  \dbE |\cY(t,r)|  + (\beta+\beta_0)\e_0\int_r^{T}\|y(\cd)\|_{L^\i_{\dbF} [s,T] } ds+L_3.
		\end{aligned}
	\end{equation}

 	\textbf{Step 3.}      Prove the estimates (\ref{4.2-1}) and (\ref{4.2-2}).
	
	\ms
	
	Inserting (\ref{4.10})  into (\ref{3.8-1}) yields
	\begin{equation}\label{3.15-1}
		\begin{aligned}
			 \exp\big\{\tilde{\g}|\cY(t,r)| \big\}+\dbE_r\int_r^T|\cZ(t,s)|^2ds\leq  L_4 \exp\Big\{L_4\int_r^T\|y(\cd)\|_{L^\i_\dbF[s,T]}  ds\Big\}
			\exp\big\{\e_0  \dbE|\cY(t,r)|    \big\}.
		\end{aligned}
	\end{equation}
Now, if   taking	 $\e_0=\frac{\tilde{\g}}{2}$, then it follows from Jensen's inequality that
	\begin{equation}\label{4.12}
		\dbE|\cY(t,r)| \leq L_5\Big\{1+\int_r^T\|y(\cd)\|_{L^\i_\dbF[s,T]}   ds\Big\},\  r\in[t,T].
	\end{equation}
	Combining (\ref{3.15-1}) and (\ref{4.12}), we have, for $r\in[t,T]$,
	\begin{equation}\label{4.12-1}
\begin{aligned}
		&|\cY(t,r)| \leq L_5\(1+\int_r^T\|y(\cd)\|_{L^\i_\dbF[s,T]}   ds\), \\
			& \dbE_r\int_r^T|\cZ(t,s)|^2ds \leq L_6    \exp\Big\{L_6\int_r^T\|y(\cd)\|_{L^\i_\dbF[s,T]}  ds\Big\}
			\leq L_6    \exp\Big\{ L_6\int_t^T\|y(\cd)\|_{L^\i_\dbF[s,T]}  ds\Big\}.
		\end{aligned}
	\end{equation}
	Now let $r=t$ in the first inequality of  (\ref{4.12-1}), we arrive at
%	\begin{equation*}
%		|Y(t)|\leq L_5\(1+\int_t^T\|y(\cd)\|_{L^\i_\dbF([s,T];\dbR)}   ds\),
%	\end{equation*}
%	%
%	which implies
	\begin{equation*}
		\|Y(\cd)\|_{L^\i_{\dbF}[t,T]}\leq L_5\(1+\int_t^T\|y(\cd)\|_{L^\i_\dbF[s,T]}   ds\).
	\end{equation*}
	Since the second inequality of  (\ref{4.12-1}) still holds   if replacing $r$ by stopping time $\t\in \sT[t,T]$, we have  from the definition of
   $ \|\cd\|^2_{ \cZ^2_\dbF[t,T]}$
	that
	\begin{equation*}\label{4.16}
		\begin{aligned}
			&  \|Z(t,\cd)\|^2_{\cZ^2_\dbF[t,T]} \leq L_6    \exp\(L_6\int_t^T\|y(\cd)\|_{L^\i_\dbF[s,T]}  ds\),\  t\in[0,T].
		\end{aligned}
	\end{equation*}
	Therefore,
	\begin{equation*}\label{4.16}
		\begin{aligned}
			&  \|Z(\cd,\cd)\|^2_{ \cZ^2_\dbF(\D[0,T])}
			= \esssup_{t\in[0,T]}\|Z(t,\cd)\|^2_{\cZ^2_\dbF[t,T]} \leq L_6    \exp\(L_6\int_0^T\|y(\cd)\|_{L^\i_\dbF[s,T]}  ds\).
		\end{aligned}
	\end{equation*}
\end{proof}

Based on the above lemma, now we can prove \autoref{th 4.2}.

\begin{proof}[Proof of \autoref{th 4.2}]
	
	%             For the constant $L_0$ given in \autoref{pro 4.1}, we introduce  a convex closed set $\cL$ of the Banach space $L^\i_{\dbF}(0,T;\dbR)\ts \overline{\cZ}^2_{\dbF}(\D[0,T];\dbR^d)$, which is  defined by
	%             \begin{equation*}
		%             \begin{aligned}
			%               \cL\deq &\Big\{( P(\cd),Q(\cd,\cd))\in L^\i(0,T;\dbR)\ts \overline{\text{BMO}}(\D[0,T];\dbR^n) \Big|\\
			%                    &\q \|P(t)\|_{\i}\leq L_0e^{L_0(T-t)},\q \text{for almost all}\ t\in[0,T],\q \text{and}\\
			%                    &\q \|Q(\cd,\cd)\|_{\overline{\cZ}^2_{\dbF}(\D[0,T])}\leq \sqrt{L_0}\exp\(L_0e^{L_0T}\)
			%                    \Big\}.
			%                \end{aligned}
		%             \end{equation*}
	
	The proof is split into two steps.
	
	\ms
	
	\textbf{Step 1.}  The map $\G$ is well-defined and stable in $\sB_\e(R_1,R_2)$.
	
	\ms
	
	  \autoref{pro 4.1} tells us that for any given pair $(y(\cd), z(\cd,\cd))\in L_{\dbF}^\i[0,T]\ts \cZ_{\dbF}^2(\D[0,T];\dbR^d)$,
	  Eq. (\ref{3.1110}) with $(y(\cd),z(\cd,\cd))$
	possesses a unique adapted solution $(Y(\cd),Z(\cd,\cd))\in L_{\dbF}^\i[0,T]\times \cZ^2_{\dbF}(\D[0,T];\dbR^d)$.
    Define a map $\G$ from $L_{\dbF}^\i[0,T]\ts \cZ_{\dbF}^2(\D[0,T];\dbR^d)$ to itself by
	\begin{equation*}
		(Y(\cd),Z(\cd,\cd))\deq \G(y(\cd),z(\cd,\cd)).
	\end{equation*}

	Next, we show that the map $\G$ is stable in $\sB_\e(R_1,R_2)$.
	%
	%For this, for any $\g>0$, we introduce the following norm:
	%            \begin{equation*}
		%              \|(P(\cd), Q(\cd,\cd))\|_{\cL^\g}\deq \sqrt{\esssup_{t\in[0,T]}\(\|e^{\frac{\g  \cd}{2}}P(\cd)\|^2_{L_{\dbF}^\i[t,T]}
			%              + \|e^{\frac{\g \cd}{2}}Q(t,\cd)\|^2_{\overline{\cZ}^2_{\dbF}(\D[t,T])}\)}.
		%            \end{equation*}
	%              Clearly, the norm $\|\cd\|_{\cL^\g}$ is equivalent to the norm  $\|\cd\|_{\cL^0}(=\|\cd\|_{\cL})$.
%	
From (\ref{4.2-1}) and (\ref{4.2-2}), we have
	\begin{equation*}
		\begin{aligned}
		\ds 	&\|Y(\cd)\|_{L^\i_{\dbF}[t,T]} \leq \overline{L}\(1+\int_t^T\|y(\cd)\|_{L^\i_\dbF[s,T]}   ds\)\leq \overline{L}+\overline{L}(T-t)\|y(\cd)\|_{L^\i_{\dbF}[t,T]},\\
		\ns\ds 	&\qq\q\ \|Z(\cd,\cd)\|^2_{\cZ^2_\dbF(\D[0,T])} \leq \overline{L}    \exp\(\overline{L}\int_0^T\|y(\cd)\|_{L^\i_\dbF[s,T]}  ds\),
		\end{aligned}
	\end{equation*}
	where   $\overline{L}$  is given in \autoref{pro 4.1}. Now, we define
	\begin{equation*}
		R_1=2\overline{L}\q \hb{and} \q R_2=\overline{L}e^{\overline{L}R_1T}.
	\end{equation*}
	Then, for  $\e\in(0,\frac{1}{2\overline{L}}]$, if $\|y(\cd)\|_{L^\i_{\dbF}[T-\e,T]}\leq R_1$, we have
	\begin{equation*}
		\|Y(\cd)\|_{L^\i_{\dbF}[T-\e,T]}\leq R_1\q\hb{and}\q  \|Z(\cd,\cd)\|^2_{\cZ^2_{\dbF}(\D[T-\e,T])}\leq R_2,
	\end{equation*}
	which implies that $\G$ is stable in $\sB_\e(R_1,R_2)$.

\ms
	
	\textbf{Step 2.}  Prove that   $\G$ is a contraction mapping.

We continue to use those notations in Step 2 of  \autoref{pro3.4}.
Let us  consider two pairs
	$(y(\cd),z(\cd,\cd))$ and $(\cl{y}(\cd),\cl{z}(\cd,\cd))$, both in  $L_{\dbF}^\i[0,T]\ts \cZ_{\dbF}^2(\D[0,T];\dbR^d)$, and  define
	\begin{equation*}
		(Y(\cd),Z(\cd,\cd))\deq \G(y(\cd),z(\cd,\cd))\q\ \hb{and}\q\    (\cl{Y}(\cd),\cl{Z}(\cd,\cd))\deq  \G(\cl{y}(\cd),\cl{z}(\cd,\cd)).
	\end{equation*}
    Let $(\cY(t,\cd), \cZ(t,\cd))\in  L^\i_{\dbF}(\O;C([t,T];\dbR))\ts \cZ^2_{\dbF}([t,T];\dbR^d)$ and $(\cl{\cY}(t,\cd), \cl{\cZ}(t,\cd))\in  L^\i_{\dbF}(\O;C([t,T];\dbR))\ts \cZ^2_{\dbF}([t,T];\dbR^d)$ be the solutions to Eq. (\ref{3.211}) with
     $(y(\cd),z(\cd,\cd))$ and $ (\cl{y}(\cd),\cl{z}(\cd,\cd))$, respectively.
   From \autoref{pro3.1}, for $(t,s)\in \D[0,T]$,
	\begin{equation*}
 Y(t)=\cY(t,t),\q  Z(t,s)=\cZ(t,s),\q \text{and}\q \cl{Y}(t)=\cl{\cY}(t,t),\q \cl{Z}(t,s)=\cl{\cZ}(t,s).
	\end{equation*}
 For $h=y(\cd),z(\cd,\cd),\cY(\cd,\cd),\cZ(\cd,\cd),Y(\cd), Z(\cd,\cd),$  denote $\d h=h-\bar{h}$. Then, for $(y(\cd),z(\cd,\cd)), (\overline{y}(\cd),$ $\overline{z}(\cd,\cd))\in\sB_\e(R_1,R_2)$,
             \begin{equation}\label{4.18}
             \begin{aligned}
               &\d \cY(t,r)+\int_r^T\d \cZ(t,s)dW(s)\\
               &=\int_r^T\big(g(t,s,y(s),\cZ(t,s),\dbP_{y(s)}, \dbP_{\cZ(t,s)})
                 -g(t,s,y(s),\cl{\cZ}(t,s),\dbP_{y(s)}, \dbP_{\cZ(t,s)})+\tilde{\tilde{J}}(t,s)\big)ds,
               \end{aligned}
             \end{equation}
             where
             \begin{equation*}
            \tilde{\tilde{J}}(t,s)\deq g(t,s,y(s),\cl{\cZ}(t,s),\dbP_{y(s)}, \dbP_{\cZ(t,s)})-
                 g(t,s,\cl{y}(s),\cl{\cZ}(t,s),\dbP_{\cl{y}(s)}, \dbP_{\cl{\cZ}(t,s)}).
                 \end{equation*}
From \autoref{ass4.1}, for almost all $t\in[T-\e,T]$, there exists a stochastic process $\tilde{\tilde{L}}(\cd,\cd)$ such that
             \begin{equation*}
             \left\{
             \begin{aligned}
              & g(t,s,y(s),\cZ(t,s),\dbP_{y(s)}, \dbP_{\cZ(t,s)})
                 -g(t,s,y(s),\cl{\cZ}(t,s),\dbP_{y(s)}, \dbP_{\cZ(t,s)})=\tilde{\tilde{L}}(t,s)\d \cZ(t,s),\\
             & |\tilde{\tilde{L}}(t,s)|\leq \phi(|y(s)|\vee \cW_2(\dbP_{y(s)},\d_0))(1+|\cZ(t,s)|+|\bar{\cZ}(t,s)|+\|\cZ(t,s)\|_{L^2(\O)}),\  s\in [t, T].
             \end{aligned}
             \right.
             \end{equation*}
             Since $(y(\cd),z(\cd,\cd)),(Y(\cd),\cZ(\cd,\cd)), (\overline{Y}(\cd),\overline{\cZ}(\cd,\cd))
             \in \sB_\e(R_1,R_2)$,
             it follows that
             \begin{equation*}
              \begin{aligned}
              \|\tilde{\tilde{L}}(t,\cd)\|_{\overline{\cZ}^2_{\dbF}( [t,T])}
              &\leq \phi(\|y(\cd)\|_{L_{\dbF}^\i[T-\e,T]})(1+2\|\cZ(t,\cd)\|_{ \cZ^2_{\dbF}([t,T])}
             +\|\bar{\cZ}(t,\cd)\|_{ \cZ^2_{\dbF}([t,T])})\\
             &\leq \phi(R_1)(1+3\sqrt{R_2}).
             \end{aligned}
             \end{equation*}
             Consequently, the process
              $\tilde{\tilde{W}}(r;t)\deq W(r)-\int_t^r\tilde{\tilde{L}}(t,s)ds,\  r\in[t,T]$
             is a standard Brownian motion under the probability
            $d\tilde{\tilde{\dbP}}^t\deq \cE(\tilde{\tilde{L}}(t,\cd))^T_td\dbP.$
             Thereby, (\ref{4.18}) can be rewritten as
             \begin{equation*}
             \begin{aligned}
               &\d \cY(t,r)+\int_r^T\d \cZ(t,s)d\tilde{\tilde{W}}(s;t)=\int_r^T \tilde{\tilde{J}} (t,s)ds,\ r\in[t,T],
               \end{aligned}
             \end{equation*}
             which implies, for $ r\in[t,T],$
             \begin{equation*}
              |\d \cY(t,r)|^2+\tilde{\tilde{\dbE}}^t_r\[\int_r^T|\d \cZ(t,s)|^2ds\]=\tilde{\tilde{\dbE}}^t_r\[\(\int_r^T |\tilde{\tilde{J}}(t,s)|ds\)^2\].
             \end{equation*}
           %  Here $\widetilde{\dbE}^t_r[\cd]=\widetilde{\dbE}^t[\cd|\cF_r]$ denotes the conditional expectation under the probability $\widetilde{\dbP}^t$.
             In particular, let $r=t$,
             \begin{equation*}
              |\d Y(t)|^2+\tilde{\tilde{\dbE}}^t_t\[\int_t^T|\d \cZ(t,s)|^2ds\]=\tilde{\tilde{\dbE}}^t_t\[\(\int_t^T |\tilde{\tilde{J}}(t,s)|ds\)^2\].
             \end{equation*}
            Notice that, for  $T-\e\leq t\leq s\leq T$,
             \begin{equation*}
             \begin{aligned}
             |\tilde{\tilde{J}}(t,s)|&\leq
             \phi(|y(s)|\vee|\overline{y}(s)|\vee \cW_2(\dbP_{y(s)},\d_0)\vee \cW_2(\dbP_{\overline{y}(s)},\d_0) )\\
              &\q  \cd\[(1+| \overline{\cZ}(t,s)|+\cW_2(\dbP_{\cZ(t,s)}, \d_0)+\cW_2(\dbP_{\overline{\cZ}(t,s)}, \d_0))
              (|\d y(s)|+\cW_2(\dbP_{y(s)},\dbP_{\overline{y}(s)}))\\
              &\qq   +(1+\cW_2(\dbP_{\cZ(t,s)},\d_0)^\a+\cW_2(\dbP_{\overline{\cZ}(t,s)},\d_0)^\a) \cW_2(\dbP_{\cZ(t,s)},\dbP_{\overline{\cZ}(t,s)})\] \\
               &\leq\phi(R_1)\[ (1+|\overline{\cZ}(t,s)|+\|\cZ(t,s)\|_{L^2(\O)}+\|\overline{\cZ}(t,s)\|_{L^2(\O)}  ) (|\d y(s)|+\|\d y(s)\|_{L^2(\O)})\\
               &\qq + (1+\|\cZ(t,s)\|_{L^2(\O)}^\a+\|\overline{\cZ}(t,s)\|_{L^2(\O)}^\a)
               \|\d\cZ(t,s)\|_{L^2(\O)} \].
               \end{aligned}
             \end{equation*}
            Consequently, we have from H\"{o}lder's inequality that
            \begin{equation*}
            \begin{aligned}
              & |\d Y(t)|^2+\tilde{\tilde{\dbE}}^t_t\[\int_t^T|\d \cZ(t,s)|^2ds\]\\
              &\leq 8\phi(R_1)^2\Big\{\tilde{\tilde{\dbE}}^t_t\[ \( \int_t^T(1+|\overline{\cZ}(t,s)|^2
              + \|\cZ(t,s)\|^2_{L^2(\O)}+\|\overline{\cZ}(t,s)\|^2_{L^2(\O)}   )ds \)^2     \] \Big\}^\frac{1}{2}\\
              &\qq \cdot\Big\{ \tilde{\tilde{\dbE}}^t_t\[ \( \int_t^T(|\d y(s)|^2+\|\d y(s)\|^2_{L^2(\O)})ds \)^2     \] \Big\}^\frac{1}{2}\\
              &\q + 6\phi(R_1)^2 \int_t^T(1+\|\cZ(t,s)\|_{L^2(\O)}^{2\a}+\|\overline{\cZ}(t,s)\|_{L^2(\O)}^{2\a}) ds
              \cd \int_t^T   \|\d\cZ(t,s)\|^2_{L^2(\O)}  ds.
              \end{aligned}
             \end{equation*}

Now,	for each $t\in[T-\e,T]$,   by the properties (\ref{2.5}) and (\ref{4.3.1}), we have that there exists a positive constant $c_2$, independent of $t$, such that
	\begin{equation*}
		\begin{aligned}
			\|\overline{\cZ}(t)\cd \tilde{\tilde{W}}(t)\|^4_{ \text{BMO}_{\tilde{\tilde{\dbP}}^t}[t,T]}
			\leq c_2\|\overline{\cZ}(t)\cd W\|^4_{\text{BMO}_{\dbP}([t,T])}
			= c_2\|\overline{\cZ}(t,\cd)\|^4_{\cZ^2_{\dbF} [t,T]}
			\leq c_2\|\overline{\cZ}(\cd,\cd)\|^4_{ \cZ^2_{\dbF}(\D[T-\e,T])}.
		\end{aligned}
	\end{equation*}
	According to energy inequality and Jensen's inequality, and notice that  both the pairs $(Y(\cd),\cZ(\cd,\cd))$ and  $(\overline{Y}(\cd),\overline{\cZ}(\cd,\cd))$ belong to the space
	$\sB_\e(R_1,R_2)$, one has that for any $\t\in\sT[T-\e,T]$,
	\begin{equation*}
		\begin{aligned}
			&\tilde{\tilde{\dbE}}^t_\t\[ \( \int_\t^T(|\d y(s)|^2+\|\d y(s)\|^2_{L^2(\O)})ds \)^2 \]\leq 4 \e^2\|\d y(\cd)\|^4_{L^\i_{\dbF}[T-\e,T]},
		\end{aligned}
	\end{equation*}
	and
	\begin{equation*}
		\begin{aligned}
\ds 			&\tilde{\tilde{\dbE}}^t_\t\Big\{ \( \int_\t^T(1+|\overline{\cZ}(t,s)|^2
			+ \|\cZ(t,s)\|^2_{L^2(\O)}+\|\overline{\cZ}(t,s)\|^2_{L^2(\O)}   )ds \)^2     \Big\}  \\
	\ns\ds 		&\leq  4\bigg\{T^2+ \tilde{\tilde{\dbE}}^t_\t  \Big\{ \(\int_\t^T|\overline{\cZ}(t,s)|^2ds\)^2     \Big\}
			+\dbE  \Big\{\dbE_\t \(\int_\t^T |\cZ(t,s)|^2 ds  \)^2\Big\}
			+\dbE \Big\{  \dbE_\t  \(\int_\t^T |\overline{\cZ}(t,s)|^2 ds  \)^2\Big\} \bigg\}\\
	\ns\ds 		&\leq  4\bigg\{T^2+ \| \cl\cZ(t)\cd\tilde{\tilde{W}}(t)\|^4_{\text{BMO}_{\tilde{\tilde{\dbP}}^t}[t,T]}
			+\|\overline{\cZ}\cd W\|^4_{ \text{BMO}(\D[T-\e,T])}
			+\|\cZ\cd W\|^4_{ \text{BMO} (\D[T-\e,T])}
			\bigg\}\\
			%
			%&\leq  4 \bigg\{T^2+ \|\overline{\cZ}\cd W\|^4_{\overline{\text{BMO}}(\D[T-\e,T])}+\|\cZ\cd W\|^4_{\overline{\text{BMO}}(\D[T-\e,T])}
			%                              \bigg\}\\
			%
			&\leq  4 \bigg\{T^2+(c_2+1)\|\overline{\cZ}(\cd,\cd)\|^4_{ \cZ^2_{\dbF}(\D[T-\e,T])}
                +\|\cZ(\cd,\cd)\|^4_{ \cZ^2_{\dbF}(\D[T-\e,T])}\bigg\}\\
			&\leq  4 \{T^2+(c_2+2)(R_2)^2\}.
		\end{aligned}
	\end{equation*}
	In addition,
	\begin{equation*}
		\begin{aligned}
		\ds 	& \int_\t^T(1+\|\cZ(t,s)\|_{L^2(\O)}^{2\a}+\|\overline{\cZ}(t,s)\|_{L^2(\O)}^{2\a}) ds
			\cd \int_t^T   \|\d\cZ(t,s)\|^2_{L^2(\O)}  ds\\
	\ns\ds 		&\leq \e^{1-\a}C_\a\(\int_\t^T(1+\|\cZ(t,s)\|_{L^2(\O)}^{2}+\| \cl\cZ(t,s)\|_{L^2(\O)}^{2}) ds\)^\a
			\cd  \|\d \cZ(\cd,\cd)\|^2_{ \cZ^2_{\dbF}(\D[T-\e,T])}\\
	\ns\ds 		&\leq \e^{1-\a}C_\a\(T+ \|\cZ(\cd,\cd)\|^2_{\cZ^2_{\dbF}(\D[T-\e,T])}+ \|\overline{\cZ}(\cd,\cd)\|^2_{\cZ^2_{\dbF}(\D[T-\e,T])}   \)^\a
			\cd  \|\d \cZ(\cd,\cd)\|^2_{\cZ^2_{\dbF}(\D[T-\e,T])}\\
		\ns\ds 	& \leq \e^{1-\a}C_\a\(T+  2R_2  \)^\a
			\cd  \|\d \cZ(\cd,\cd)\|^2_{\cZ^2_{\dbF}(\D[T-\e,T])},
		\end{aligned}
	\end{equation*}
	where $C_\a$ depends only on $\a$ and $T$.  Hence, we arrive at
	\begin{equation}\label{4.25}
		\begin{aligned}
			& |\d \cY(t,\t)|^2+\tilde{\tilde{\dbE}}^t_\t\int_\t^T|\d \cZ(t,s)|^2ds
			\leq C_0\Big\{ \e \|\d y(\cd)\|^2_{L^\i_{\dbF}[T-\e,T]}+ \e^{1-\a}
			\|\d \cZ(\cd,\cd)\|^2_{\cZ^2_{\dbF}(\D[T-\e,T])}  \Big\}.
		\end{aligned}
	\end{equation}
	Here $C_0$ depends on $\a, c_2, R_1,R_2,T$ and $\phi(\cd)$. Now, by	choosing $\e$ small enough, it follows
	\begin{equation*}
		\begin{aligned}
		 	&  \|\d Y(\cd)\|^2_{L^\i_{\dbF}[T-\e,T]}+\|\d \cZ(\cd,\cd)\|^2_{ \cZ^2_{\dbF}(\D[T-\e,T])}
			\leq \frac{1}{2} \|\d y(\cd)\|^2_{L^\i_{\dbF}[T-\e,T]}\\
		 	& \leq \frac{1}{2} \(\|\d y(\cd)\|^2_{L^\i_{\dbF}[T-\e,T]}+ \|\d z(\cd,\cd)\|^2_{\cZ^2_{\dbF}(\D[T-\e,T])}\),
		\end{aligned}
	\end{equation*}
	which implies that
	\begin{equation*}
		\begin{aligned}
			&  \|\d Y(\cd)\|^2_{L^\i_{\dbF}[T-\e,T]}+\|\d Z(\cd,\cd)\|^2_{\cZ^2_{\dbF}(\D[T-\e,T])}
			\leq \frac{1}{2} \(\|\d y(\cd)\|^2_{L^\i_{\dbF}[T-\e,T]}+ \|\d z(\cd,\cd)\|^2_{\cZ^2_{\dbF}(\D[T-\e,T])}\).
		\end{aligned}
	\end{equation*}
	Thus $\G$ is a contraction mapping on the space $\sB_\e(R_1,R_2)$.
	Thereby, BSVIE (\ref{1.1}) admits  a unique adapted solution
	$(Y(\cd), Z(\cd,\cd))\in \sB_\e(R_1, R_2)$.
%
%%%%%%%%%%%%%%%%%%%%%%%%%%%%%%%%%%%%%%%%%%%%%%%%%%%%%%%%%%%%%%%%%%%%%%%%%%%%%%%%%%%%%%%%%%%%%%%%%%%%%%%%%%%%%%%%%%%%%%%%%%%%%%%%%%%%
%%%%%%%%%%%%%%%%%%%%%%%%%%%%%%%%%%%%%%%%%%%%%%%%%%%%%%%%%%%%%%%%%%%%%%%%%%%%%%%%%%%%%%%%%%%%%%%%%%%%%%%%%%%%%%%%%%%%%%%%%%%%%%%%%%%%%%%%
%
\end{proof}

\subsection{Proof of \autoref{pro4.4}}
\begin{proof}
For any given $(y^{N}(\cd),z^{N}(\cd,\cd))\in {L^\i_{\dbF}([0,T];\dbR^N)}\ts   \cZ^2_{\dbF}(\D[0,T];(\dbR^{N\ts d})^N)$, the following multi-dimensional  BSVIE,
for $i=1,\cds,N$,
         \begin{equation}\label{6.311}
         \begin{aligned}
          Y^{N,i}(t)&=\psi^i(t)+\int_t^Tg(t, s,y^{N,i}(s),Z^{N,i,i}(t,s), \frac{1}{N}\sum\limits_{i=1}^N \d_{y^{N,i}(s)})ds\\
                &\q-\int_t^T\sum\limits_{j=1}^NZ^{N,i,j}(t,s)dW^j(s),\q t\in[0,T]
            \end{aligned}
         \end{equation}
        has a unique adapted solution $(\mathbf{Y}^N(\cd),\mathbf{Z}^N(\cd,\cd))=(Y^{N,i}(\cdot), Z^{N,i,j}(\cdot, \cdot))_{1\leq i,j\leq N}$.

         Indeed, for each  $t\in[0,T]$, consider a BSDE
         \begin{align*}
         \cY^{N,i}(t,s)&=\psi^i(t)+\int_s^Tg(t, r,y^{N,i}(r),\cZ^{N,i,i}(t,r), \frac{1}{N}\sum\limits_{i=1}^N \d_{y^{N,i}(r)})dr\\
                &\q-\int_s^T\sum\limits_{j=1}^N\cZ^{N,i,j}(t,r)dW^j(r),\q s\in[t,T].
         \end{align*}
         For each $i$, the above equation is a one-dimensional QBSDE  (due to $ y^N(\cd)=(y^{N,1},\cds,y^{N,N})$  being given). According to
         \cite[Lemma A.1]{Fan-Hu-Tang-2023},
         this BSDE possesses a unique adapted solution $ (\cY^{N,i}(t,\cd),\cZ^{N,i}(t,\cd))=(\cY^{N,i}(t,\cd),(\cZ^{N,i,j}(t,\cd))_{j=1,\cds,N}).$
         Moreover,  for $s\in[t,T]$ and any stopping time $\t\in\sT[t,T]$,
         \begin{equation}\label{4.13}
         |\cY^{N,i}(t,s)|\leq \frac{\ln2}{\g}+\|\psi(t)\|_{L^\i_{\sF_T}[0,T] }+\|\ell(\cd,\cd)\|_{L^\i_{\dbF}(\D[0,T])}
                        +(\b+\frac{\b_0}{N})\|y^{N}(t,\cd)\|_{L^\i_{\dbF}[t,T]},
         \end{equation}
         and
         \begin{equation}\label{4.14}
         \begin{aligned}
          \dbE_\t\[ \int_\t^T \sum\limits_{j=1}^N| \cZ^{N,i,j}(t,s)|^2ds\]
         &\leq \frac{1}{\g^2}\exp\(2\g\|\psi(\cd)\|_{L^\i_{\sF_T}[0,T]}\)
         +\frac{1}{\g}\exp\(2\g\Big\| \cY^{N,i}(t,\cd)  \Big\|_{L^\i_{\dbF}[t,T]}\)  \\
         &\qq\cd\( 1+2\|\ell(\cd,\cd)\|_{L^2_{\dbF}(\D[0,T])}    +(\b+\frac{\b_0}{N})\|y^{N}(t,\cd)\|_{L^\i_{\dbF}[t,T]}\).
         \end{aligned}
         \end{equation}
         Thanks to \autoref{pro3.1},  $(Y^{N,i}(\cd), Z^{N,i,j}(\cd,\cd))_{1\leq i,j\leq N}$ admits a unique solution to Eq. \rf{6.311}
         and
         \begin{align*}
          Y^{N,i}(t)=\cY^{N,i}(t,t), \  Z^{N,i,j}(t,s)=\cZ^{N,i,j}(t,s), \ (t,s)\in \D[0,T].
         \end{align*}
         Meanwhile, thanks to (\ref{4.13})-(\ref{4.14}) and \autoref{ass4.1-1}, we have, for $t\in[0,T],$
 \begin{equation}\label{4.15-1}
          \Big\|\cY^{N}(t,\cd)  \Big\|_{L^\i_{\dbF}[t,T]}\leq \frac{N\ln2}{\g}+N(K_1+\sqrt{K_3})
                        +(N\b+ \b_0 )\|y^{N}(t,\cd)\|_{L^\i_{\dbF}[t,T]}(T-t),
         \end{equation}
  \begin{equation}\label{4.15-2}
          |Y^{N}(t)| \leq \frac{N\ln2}{\g}+N(K_1+\sqrt{K_3})
                        +(N\b+ \b_0 )\|y^{N}(t,\cd)\|_{L^\i_{\dbF}[t,T]}(T-t),
         \end{equation}
and
\begin{equation}\label{4.16-1}
         \begin{aligned}
          \|Z^N(\cd,\cd)\|^2_{ \cZ^2_{\dbF}(\D[t,T])}
          &\leq \frac{N}{\g^2}\exp\(2\g  K_1\)
         +\frac{N}{\g}\exp\(2\g \Big\| \cY^{N}(t,\cd)  \Big\|_{L^\i_{\dbF}[t,T]}  \)  \\
         &\qq\cd\( 1+2\sqrt{K_3 }   +(\b+\frac{\b_0}{N})\|y^{N}(t,\cd)\|_{L^\i_{\dbF}[t,T]}(T-t)\),
         \end{aligned}
         \end{equation}
          where  $\cY^{N}=(\cY^{N,1},\cds,\cY^{N,N}),\ Y^{N}=(Y^{N,1},\cds,Y^{N,N}),\ Z^{N}=(Z^{N,1},\cds,Z^{N,N})$.

Define
    \begin{equation*}
     \begin{aligned}
    &\widehat{R}_1\deq \frac{N\ln2}{\g}+N(K_1+\sqrt{K_3}),\\
    &\widehat{R}_2\deq \frac{N}{\g^2}\exp\(2\g  K_1\)
         +\frac{N}{\g}\exp\(4\g \widehat{R}_1 \) (1+2\sqrt{K_3}),\\
    & \e_1\deq  \min\bigg\{\frac{1}{2(N\b+  \b_0  )},
      \frac{\widehat{R}_2 \g\exp\(-4\g \widehat{R}_1\)}{2\widehat{R}_1( N\b+ \b_0 )}
        \bigg\}.
   \end{aligned}
   \end{equation*}
According to (\ref{4.15-1})-(\ref{4.16-1}), if for  $\e\in(0,\e_1]$,
\begin{align*}
 \| y^{N}(\cd)\|_{L^\i_{\dbF}([T-\e,T];\dbR^N)}\leq 2\widehat{R}_1,
\end{align*}
then it follows
\begin{align*}
 \| Y^{N}(\cd)\|_{L^\i_{\dbF}([T-\e,T];\dbR^N)}\leq 2\widehat{R}_1,\q  \|Z^N(\cd,\cd)\|^2_{ \cZ^2_{\dbF}(\D[T-\e,T];(\dbR^{N\ts d})^N)}\leq 2\widehat{R}_2.
\end{align*}
This means that for the map
\begin{align*}
&\Upsilon(y^N(\cd),z^N(\cd,\cd))\deq (Y^N(\cd),Z^N(\cd,\cd)),
 \end{align*}
$\forall (y^N(\cd),z^N(\cd,\cd))\in  L^\i_{\dbF}([T-\e,T];\dbR^N)\ts   \cZ^2_{\dbF}(\D[T-\e,T]; (\dbR^{N\ts d})^N),$
 if $\forall (y^N(\cd),z^N(\cd,\cd))\in \sB_{N,\e}(\widehat{R}_1, \widehat{R}_2),$  then we have
\begin{equation*}
\Upsilon(y(\cd),z(\cd,\cd))\in \sB_{N,\e}(\widehat{R}_1, \widehat{R}_2),
\end{equation*}
where
\begin{equation}\label{3.2222}
	\begin{aligned}
		\sB_{N,\e}(R_1,R_2):=  \Big\{ &
		(P(\cd), Q(\cd,\cd))\in L_\dbF^\infty([T-\e,T];\dbR^N)\ts \cZ^2_\dbF(\D[T-\e,T]; (\dbR^{N\ts d})^N) \Big|\\
		& \| P(\cd)\|_{L_\dbF^\infty([T-\e,T];\dbR^N)}\les R_1\q
		\text{and}\q \|Q(\cd,\cd)\|^2_{\cZ^2_\dbF( \D[T-\e,T];(\dbR^{N\ts d})^N) }\les R_2\Big\},
	\end{aligned}
\end{equation}
endowed with the norm
\begin{equation*}
	\|(P(\cd), Q(\cd,\cd))\|_{\sB_{N,\e}}\deq \sqrt{\|P(\cd)\|^2_{L_{\dbF}^\i([T-\e,T];\dbR^N)}
		+ \|Q(\cd,\cd)\|^2_{\cZ^2_{\dbF}(\D[T-\e,T];(\dbR^{N\ts d})^N)}}.
\end{equation*}

Similar to \autoref{pro3.4}, one can show $\Upsilon$ is contractive.  Hence, Eq. \rf{1.3}
has a local solution. Based on the method used in the proof of \autoref{th3.4} to extend the local solution to a global one, we can obtain that Eq. \eqref{1.3} admits a unique global adapted solution.

Next, let us show that the bounds of  $||Y^{N,i}(\cd)||_{L_\dbF^\infty[0,T]}$ and
$||Z^{N,i,j}(\cd,\cd)||_{\cl{\cZ}^2_\dbF(\D[0,T])}$ are independent of $N$.

Let  $(\mathbb{Y}^{N,i}(t,\cd), \mathbb{Z}^{N,i,j}(t,\cd))$ be the unique adapted solution to Eq. (\ref{4.3-111}). Recall  that
$$
Y^{N,i}(t)=\mathbb{Y}^{N,i}(t,t), \q  Z^{N,i,j}(t,s)=\mathbb{Z}^{N,i,j}(t,s), \q (t,s)\in\D[0,T]
$$
(see (\ref{4.4111})).
Set $\widehat{\th}=\b^2+\b_0^2+1$. Making a similar deduction as (\ref{3.20}), one has
\begin{equation}\label{5.27}
\begin{aligned}
&  | \mathbb{Y}^{N,i}(t,s)|^2
\leq e^{\th T}[(K_1)^2+K_2]+ e^{\th T} \int_s^T\|Y^{N,i}(\cd)\|^2_{L^\i_{\dbF}[r,T]}dr
 +e^{\th T}\int_t^T\frac{1 }{N}\sum\limits_{i=1}^N\|Y^{N,i}(\cd)\|^2_{L^\i_{\dbF}[r,T]}dr.
\end{aligned}
\end{equation}
Taking $s=t$, we have from Gronwall's inequality that
$$
 ||Y^{N,i}(\cd)||_{L_\dbF^\infty[0,T]}\leq \overline{C},
$$
where $\overline{C}$ depends only on $\b,\b_0, K_1,K_2,T$, independent of $N$.
Substituting the above inequality into (\ref{5.27}) yields
$$
| \mathbb{Y}^{N,i}(t,s )|\leq C_0,\  s\in[t,T].
$$

              Recall the definition of the function  $\Xi$ (see (\ref{3.9-111})).
              Applying  It\^{o}'s formula to $\Xi( \dbY^{N,i}(t,\cd))$ provides,  for any stopping time $\t\in\sT[t,T],$
           \begin{equation*}
     \begin{aligned}
          &\dbE_\t[ \Xi(\dbY^{N,i}(t,\t))]+\frac{1}{2}\mathbb{E}_\t\bigg[\int_{\t}^{T}\sum\limits_{j=1}^N|\dbZ^{N,i,j}(t,s)|^2ds \bigg]   \\
            &\les \dbE_\t[  \Xi(\psi^i(t))] +\mathbb{E}_\t\bigg[\int_{\t}^{T}|\Xi'(\dbY^{N,i}(t,s))|\\
           &\q
           \cd\Big(\ell(t,s)+\b|Y^{N,i}(s)|+\b_0 \cW_2(\mu^N(s), \delta_{0})\Big)ds\bigg]\\
           &\les  \Xi(K_1) +\mathbb{E}_\t\bigg[\int_{\t}^{T}|\Xi'(C_0)|
            (\ell(t,s)+(\b+\b_0)\overline{C}] ) ds\bigg].
                       \end{aligned}
              \end{equation*}
           Hence,
           \begin{equation*}
         \begin{aligned}
            \frac{1}{2}\mathbb{E}_\t\bigg[\int_{\t}^{T}\sum\limits_{j=1}^N|\dbZ^{N,i,j}(t,s)|^2ds \bigg]
           &\les     \Xi(K_1)  + |\Xi'(C_0)|[K_2+(\b+\b_0)\overline{C}T]\deq \overline{\overline{C}}.
            \end{aligned}
         \end{equation*}
         Notice that $Z^{N,i,j}(t,s)=\mathbb{Z}^{N,i,j}(t,s),\ (t,s)\in\D[0,T]$, we arrive at
          $$||Z^{N,i,j}(\cd,\cd)||^2_{\cl{\cZ}^2_\dbF(\D[0,T])}\leq \overline{\overline{C}},$$
        where $\overline{\overline{C}}$  is independent of $N$.
\end{proof}

%\subsection{Proof of \autoref{prop 5.1}}
%
% \begin{proof}
%             	%
%        The particle system (\ref{1.3}) is a multi-dimensional BSVIE whose generator $g$ satisfies conditions (3.3)-(3.4) in Fan, Wang, and Yong \cite{Fan-Wang-Yong-2023}.
%            %
%             Therefore, by \cite[Theorem 3.3]{Fan-Wang-Yong-2023}, BSVIE (\ref{1.3}) admits a unique solution  denoted by $(Y^{N,i}(\cd), Z^{N,i,j}(\cd,\cd))_{1\leq i,j\leq N}$ in the space $ L^\i_\dbF ([0,T];\dbR^N) \ts  \cZ^2_\dbF(\D[0,T]; ( \dbR^{N\ts d})^N  ).$
%
%
%             We now show that there exists a positive constant $C$, independent of $N$, such that inequality \rf{24771} holds.
%   %
%              Following a derivation similar to (\ref{4.34-1})-(\ref{4.34-1-1}), we obtain
%              \begin{align*}
%              |Y^{N,i}(t)| & \leq3(K_1+K_2) +2L_1T+2L_{2,\g_0}T
%			   +3\beta\int_t^T\|Y^{N,i}(\cd)\|_{L_\dbF^\infty[u,T]}du \\
%              &\q  +3\beta_0\int_t^T\Big\{\frac{1}{N}\sum\limits_{i=1}^N\|Y^{N,i}(\cd)\|^2_{L_\dbF^\infty[u,T]}\Big\}^\frac{1}{2}du.
%              \end{align*}
%%
%Squaring both sides of the inequality and applying Gronwall's inequality yields an estimate for $\frac{1}{N}\sum\limits_{i=1}^N\|Y^{N,i}(\cd)\|^2_{L_\dbF^\infty[t,T]}$. A second application of Gronwall's inequality then gives $\|Y^{N,i}(\cd)\|_{L_\dbF^\infty[t,T]}\leq L_0$, where $L_0$ depends only on $K_1,K_2,L_1,T,L_{2,\g_0},\b,\b_0$ and is independent of $N$.			
%The bounds for $||Z^{N,i,j}(\cd,\cd)||_{\cZ^2_\dbF(\D[0,T])}$ are derived similarly to \autoref{pro4.4}, and are therefore omitted.
%          \end{proof}

\section{Conclusions}\label{sec 6}

This paper studies  the well-posedness  of mean-field BSVIEs and the convergence rate of related particle systems. When the generator is of  linear growth in the variables $(y,z,\mu,\nu)$, the convergence rate is $\cQ(N)$. However, for generators with quadratic growth in $z$, the convergence rate becomes $\cO(N^{-\frac{1}{2\l}})$ for some $\l>1$ when $g$ is independent of the law of $Z(\cd,\cd)$. An extra constant $\lambda$ appears   in the denominator of the exponent for the quadratic growth case. This phenomenon arises primarily from applying Girsanov's theorem to eliminate the effect of the quadratic term. Whether the convergence rate for the quadratic case can achieve $\cO(N^{-\frac{1}{2}})$  is an interesting topic, which we leave for future research.

\end{document}